\documentclass[11pt]{amsart}
\usepackage{amssymb, amsfonts, stmaryrd,fancyhdr,amsmath,amsthm,amsbsy,amsfonts, titletoc} 
\usepackage{latexsym,amscd, wrapfig}
\usepackage{amsmath, multirow}
\usepackage{amssymb}
\usepackage{graphicx,rotating}   %
\usepackage[utf8]{inputenc}
\usepackage{multirow,bigstrut}
\usepackage[usenames, dvipsnames]{color}
\usepackage[citecolor=blue]{hyperref}
\hypersetup{colorlinks=true,linkcolor=blue,urlcolor=blue}
\graphicspath{{Images/}}

\theoremstyle{plain}

\newcommand{\beq}{\begin{equation}}
\newcommand{\eeq}{\end{equation}}
\newcommand{\bna}{\begin{eqnarray}}
\newcommand{\ena}{\end{eqnarray}}
\newcommand{\bea}{\begin{eqnarray*}}
\newcommand{\eea}{\end{eqnarray*}}

\newcommand{\bt}[1]{ \begin{tabular} { #1 } }
\newcommand{\et} {\end{tabular}}
\newcommand{\normmm}[1]{{\left\vert\kern-0.25ex\left\vert\kern-0.25ex\left\vert #1 
    \right\vert\kern-0.25ex\right\vert\kern-0.25ex\right\vert}}

 \newcommand{\jmp}[1]{[\![#1]\!]}
 \newcommand\smbull{%
    \raisebox{-0.25ex}{\scalebox{1.7}{$\cdot$}}%
}
\numberwithin{equation}{section}
\def\cA{{\mathcal A}}

\def\cC{{\mathcal C}}

\def\cE{{\mathcal E}}

\def\cI{{\mathcal I}}

\def\cL{{\mathcal L}}
\def\cM{{\mathcal M}}
\def\cN{{\mathcal N}}

\def\cR{{\mathcal R}}
\def\cS{{\mathcal S}}

\def\mbbA{\mathbb{A}}

\def\mbbC{\mathbb{C}}

\def\mbbR{\mathbb{R}}
\def\mbbS{\mathbb{S}}

\def\mbbV{\mathbb{V}}

\def\mrA{\textrm A}
\def\mrB{\textrm B}

\def\mrD{\textrm D}
\def\mrE{\textrm E}

\def\mrG{\textrm G}
\def\mrH{\textrm H}
\def\mrI{\textrm I}
\def\mrJ{\textrm J}
\def\mrK{\textrm K}
\def\mrL{\textrm L}

\def\mrW{\textrm W}

\def\mrd{\textrm d}

\def\ts{\tilde{s}}

\def\td{\bar{\ell}}

\def\tU{\tilde{U}}
\def\vp{\vec{p}}
\def\vtau{{\tau}}
\def\del{\delta}
\def\tg{\tilde{g}}
\def\tgam{\tilde{\gamma}}
\def\eps{\varepsilon}

\def\bpm{\begin{matrix}}
\def\epm{\end{pmatrix}}

\def\bpm{\begin{pmatrix}}
\def\epm{\end{pmatrix}}
\DeclareMathOperator{\Span}{span}
 
   \def\rmd{\mathrm d}
   \def\lie{\mathrm {LI}}

\newtheorem{thm}{Theorem}
\newtheorem{lemma}[thm]{Lemma}

\newtheorem{remark}{Remark}
\newtheorem{prop}[thm]{Proposition}

\begin{document}
\title{\textbf{Quasisteady Patterns in Interfaces: Folding and Faceting}}
\author[Nguyen]{Vinh Nguyen}
\author[Promislow]{Keith Promislow}
\author[Wetton]{Brian Wetton}
\address{Vinh Nguyen, Department of Mathematics, Michigan State University,
East Lansing, MI 48824, USA, nguy1685@msu.edu}
\address{Keith Promislow, Department of Mathematics, Michigan State University,
East Lansing, MI 48824, USA, promislo@msu.edu}
\address{Brian Wetton, Department of Mathematics, University of British Columbia,
Vancouver, BC V6T 1Z2, Canada, wetton@math.ubc.ca}
\date{\today}

\begin{abstract}
We present a systematic derivation of the gradient flows associated to a broad class of interfacial energies, emphasizing the relation between intrinsic and extrinsic variations of the interface. 
 We show that the intrinsic variables formulation brings the gradient flow into alignment with the traditional analysis of quasi-steady dynamical systems defined on a stationary domain.  Gradient flows are derived for model systems which exhibit quasi-steady pattern formation including coarsening among faceted interfaces and nonlocal interactions that model membrane self-adhesion and self-avoidance.  
\end{abstract}
\maketitle

\section{Introduction}
There is a rich literature describing quasi-steady pattern formation via reductive limits of activator-inhibitor and chemotactic reaction diffusion systems.  Generically, these system have been posed on a fixed domain. Many applications in biological and chemical process are mediated by an evolving membrane.  This includes  fluid membranes, such as lipid bilayers, composed of thin layers of materials that diffuse relatively easily in the in-plane direction. The goal of this work is to present a minimal framework for the derivation of gradient flows of energies associated to interfaces as a tool for quasi-steady analysis of the evolving interface. 

The local self-energy of an infinitely thin fluidic membrane does not depend upon deformation from a reference configuration but rather on its intrinsic local variables -- the metric and curvatures as defined by the first and second fundamental forms.    The iconic example is the Canham-Helfrich energy, \cite{Can_70}-\cite{Hel_73}, which models a two-dimensional interface embedded in $\mbbR^3$ through a simple quadratic dependence of the interfacial energy on the two membrane curvatures.  Subject to various constraints, its $L^2$-gradient flows give variants of a Willmore flow,  \cite{Rupp_24}.
The complexity induced by the parameterization of an interface embedded in $\mbbR^n$ can inhibit model development and analysis of more complicated surface energies. Frameworks set in higher dimensional settings are available in the literature, \cite{Barrett_2020} but are generally geared towards computational implementation.

We present a minimal framework for the derivation of gradient flows of interfacial energies embedded in $\mbbR^2$, casting them as dynamical systems of the intrinsic variables of the interface. We consider a general local energy incorporating terms up to second order in surface diffusion, and then extend this to incorporate two-point interaction kernels that describe membrane self-interactions at a distance. Throughout we emphasize how the tools from quasi-adiabatic reduction can be brought to bear upon these important classes of problems.  

 We present a systematic derivation of first and second variations of these energies with respect to the intrinsic coordinates. This is conducted within the framework of classical calculus of variations, with the conditions that the intrinsic coordinates generate a smoothly closed curve captured through an explicit set of constraints.  The $L^2$ gradient flow is described as an evolution system for the intrinsic variables,  highlighting the relation between the arc length variation of the energy and the first integral of the functional variation at an equilibrium.
Reducing the scope to a family of energies that are first order in surface gradient of curvature, explicit relations for the second variation are extracted and their connection to the linearization of the $L^2$ gradient flow at equilibrium and quasi-equilibrium are investigated. As an application we introduce a faceting energy, essentially an Allen-Cahn energy in the interfacial curvature, emphasizing the parallels between this Allen-Cahn curvature gradient flow and that of the famous Cahn-Hilliard system on a fixed domain.

In Section\,\ref{s:TP}, we present a class of two point energies that include adhesion-repulsion effects. These are modeled through interaction kernels that characterize the energetic cost of surface self-proximity through double surface integrals. These model electrostatic interaction at longer range and hard-core repulsion at short range, evocative of Lenard-Jones potentials. We derive expressions for their gradient flow and apply them to an energy combining a Canham-Helfrich curvature and a surface area penalty term. We show formally that the system stabilizes two points of a curve at closest approach at a prescribed distance. Surprisingly, the two-point energies induce a set of non-trivial curve-constraints associated to energy invariance under rigid body motions, as outlined in Section\,\ref{s:TP_invar}.  Simulations of the associated $L^2$ gradient flow show a rich  regime of quasi-steady folded states, similar to labyrinthine patterns. The gradient flow navigates  these labyrinthine states to arrive at a final folded equilibrium state.
\begin{figure}
\begin{center}
\includegraphics[height=2.5in]{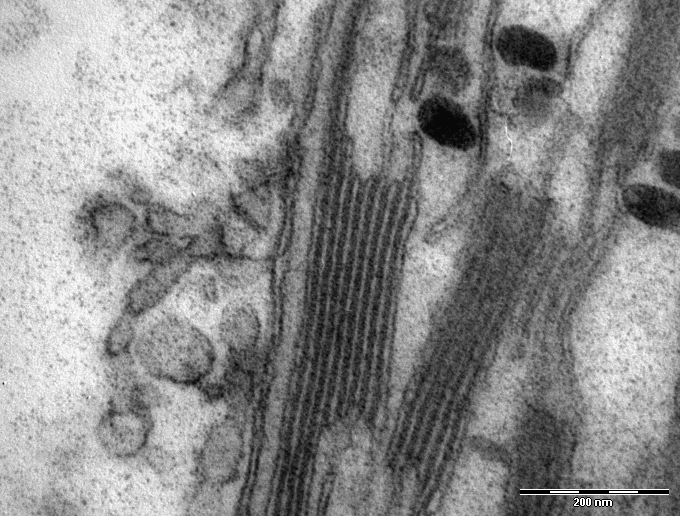}
\end{center}
\caption{ An in-situ transmission electron micrograph of ``pancake stacks'' of thylakoid membranes within a chloroplast of an Anemone leaf. The stacks self assemble from inter-membrane electrostatic forces and act as light absorbing panels for photosynthesis, \cite{WC_chloro}}
\label{f:Thylakoid}
\end{figure}
The Canham-Helfrich adhesion-repulsion energy captures stacking and folding effects commonly seen in cellular organelles, such as the Golgi apparatus \cite{Golgi-17}, the endoplasmic reticulum \cite{Shemesh-14}, and thylakoid membranes that enclose the chloroplasts that conduct photosynthesis, see Figure\,\ref{f:Thylakoid} and \cite{AS-11}. It also has applications to folding of charged polymer chains \cite{Tag24}.
The appendix presents a discussion of the numerical scheme and technical lemmas that support the main text.



\section{General Formalism for Curve Evolution}
We focus on an  evolving curve $\Gamma$ parameterized by a map $\gamma:\mbbS\times[0,T]\mapsto{\mathbb R}^2$
where $\mbbS$ is the circle of unit circumference.  Denoting the unit tangent vector and outward unit normal vector of the curve at the point $\gamma(s,t)$ by $\tau(s,t)$ and $n(s,t)$ respectively. We denote arc length by
\beq\label {e:metric-def}
g:=|\partial_s\gamma|.
\eeq
In $\mbbR^d$, it is natural to use the metric, defined through the first fundamental form of the surface. In $\mbbR^2$ using arc-length affords some simplification in the formulation,  in $\mbbR^2$ the metric reduces to the square of the arc-length.
As a minor notational point, we refer to $g$ as the arc length as a shortened version of the arc length parameterization. We normalize the gradients with arc-length as this leads to a simpler formulation, in particular surface area scales linearly with $g.$ With this caveat we define the surface gradient
$$ \nabla_s f:= \frac{1}{g}\partial_s f,$$
the Laplace-Beltrami operator
$$\Delta_s f :=  \frac{1}{g} \partial_s\left( \frac{1}{g}\partial_s f \right)=\nabla_s(\nabla_s f),$$
the surface measure
 $$\rmd\sigma=g\,\rmd s,$$
and the curvature
\beq\label{e:curv-def}
\kappa:=-\nabla_s\vtau\cdot n=-|\nabla_s \vtau|.
\eeq
Under these definitions a circle admits a positive constant curvature. 
The following relations hold
\beq
\label{e:tau-n}
\nabla_s\gamma=\vtau, \quad \nabla_s\vtau=-\kappa n, \quad \nabla_s n=\kappa \vtau.
\eeq
The curve $\gamma$ is uniquely defined, up to rigid body motion, by its intrinsic coordinates $U=(\kappa,g)^t$ which maps $\mbbS\times\mbbR_+$ into $\mbbR^2.$ Here and below superscript $t$ denotes transpose of a vector in $\mbbR^d.$
The deformation of the curve $\gamma$ can be specified through an extrinsic vector field that establishes  normal and tangential velocities, denoted  $\mbbV=(\mbbV^n(s,t),\mbbV^\tau(s,t))^t\in\mbbR^2$. This corresponds to the evolution 
\beq \label{e:gam-vel}
\gamma_t=\mbbV^n n+\mbbV^\tau \vtau.
\eeq
The extrinsic vector field  $\mbbV$ defined on $\mbbS\times[0,T]$ induces the intrinsic vector field  $U_t$ through the relations
\beq\label{e:FSFF}
\begin{aligned}
\kappa_t &= -(\Delta_s + \kappa^2) \mbbV^n+ \mbbV^\tau\nabla_s \kappa,\\
g_t & =  g (\nabla_s\mbbV^\tau +\kappa \mbbV^n), \end{aligned}
\eeq
see \cite{Pismen_06}[Sec 3.4].   

Denote the angle of the tangent vector $\tau(s)$ to the x-axis by $\theta=\theta(s,t)$. It is helpful to recast the unit tangent and normal vectors $\vtau$ and $n$ as explicit functions of $\theta$:
$$ \vtau(\theta):=\begin{pmatrix} \cos\theta\cr \sin \theta\end{pmatrix}, \quad n(\theta):=\begin{pmatrix} \sin\theta\cr -\cos \theta\end{pmatrix},$$
which satisfy the trigonometric relations
\beq
\label{e:trig}
\vtau\,^\prime(\theta)=-n(\theta), \quad \vtau\,^{\prime\prime}(\theta)=-\vtau(\theta),
\eeq
where here and below $\prime$ denotes differentiation with respect to $\theta.$ Since $\partial_s\vtau=\vtau'(\theta)\theta_s=-n(\theta)\theta_s,$ the expression for  $\nabla_s\vtau$ from \eqref{e:tau-n} yields the relation
\beq\label{e:theta-kappa}
 \nabla_s \theta=\kappa.
 \eeq

A curve $\gamma:\mbbS\mapsto \mbbR^2$ is smoothly closed if its end points and tangent vectors both align at the point of periodicity of $\mbbS$. More specifically defining the jump of a function $f:\mbbS\mapsto\mbbR,$
$$\jmp{f}:=f(|\mbbS|)-f(0),$$
then a pair $U=(\kappa,\gamma)^t$ of  $\mbbS$-periodic curvature $\kappa$ and arc-length $g$ corresponds to a curve $\gamma$ with a smoothly closed images $\Gamma$ if and only if
\begin{align}
  \jmp{\gamma}= \int_\mbbS \vtau(\theta)\,\rmd \sigma &=0, \label{e:C1}\\
  \jmp{\theta}=  \int_\mbbS \kappa(s)\rmd \sigma &=2\pi.\label{e:C2}
\end{align}
This forms the admissible class of curvature-arc length pairs that lead to smoothly closed images,
\beq \label{e:cA-def}
\cA:=\{U\in H^2(\mbbS)\, \bigl|\, \cC(U)=0\},
\eeq
where $\cC:H^1(\mbbS)\mapsto \mbbR^3$ is defined as $\cC(U)=(\jmp{\gamma},\jmp{\theta}-2\pi)^t$. We set aside issues of self-intersection of the curve, these must be handled on an energy and flow dependent manner.

To understand the evolution of these jump quantities, it is instructive to take the time derivative of the first equality in \eqref{e:tau-n}. Recalling the $g$ dependence of $\nabla_s$ yields
$$
- \frac{g_t}{g} \nabla_s \gamma +\nabla_s(\mbbV^n n+\mbbV^\tau\vtau) =\vtau^\prime(\theta)\theta_t=-n\theta_t,$$
and distributing the derivatives we obtain
$$- \frac{g_t}{g}\vtau +n\nabla_s\mbbV^n +\kappa\vtau \mbbV^n +\vtau \nabla_s \mbbV^\tau-\kappa \mbbV^\tau n =-n\theta_t.
$$ 
Taking the dot product with the normal $n$ establishes the $\theta$ evolution equation
\beq 
\label{e:theta_t}
\theta_t = -\nabla_s \mbbV^n +\kappa\mbbV^\tau.
\eeq
Since  the curvature $\kappa$ and the velocity $\mbbV$ are $\mbbS$-periodic it follows that $\jmp{\theta_t}=0.$ 

\begin{lemma}
Consider a curve evolving under a prescribed  $\mbbS$-periodic extrinsic velocity $\mbbV.$ 
If the initial curve $\gamma\bigl|_{t=0}$ satisfies $\jmp{\theta}\bigl|_{t=0}=2\pi$ then the jump quantities $\jmp{\gamma}$ and $\jmp{\theta}$ are invariant for smooth solutions of the flow \eqref{e:FSFF} subject to $\mbbS$-periodic boundary conditions. In particular a curve that is initially smoothly closed will remain smoothly closed.
\end{lemma}
\begin{proof}
From \eqref{e:theta_t} $\theta_t$ is $\mbbS$-periodic so that
$$   \frac{\rmd}{\rmd t} \jmp{\theta}=\jmp{\theta_t}=0.$$
Thus if $\jmp{\theta}\bigl|_{t=0}=2\pi$ then $\theta$ satisfies \eqref{e:C2} for all $t>0.$ In particular we deduce that the tangent $\vtau$ and normal $n$ are $\mbbS$-periodic.  
This allows us to integrate-by-parts in integrals involving $\theta.$ Specifically we evaluate the time derivative of $\jmp{\gamma}$,
$$\begin{aligned}
   \frac{\rmd}{\rmd t}\jmp{\gamma}=\frac{\rmd}{\rmd t} \left(\int_\mbbS \vtau \,\rmd \sigma \right)&=
   \int_\mbbS \left(\vtau'(\theta)\theta_t  + \frac{\vtau(\theta)g_t}{g}\right)\, \,\rmd \sigma,\\
   &= \int_\mbbS \vtau'(\theta)(-\nabla_s\mbbV^n +\kappa\mbbV^\tau) +
\vtau(\theta)(\nabla_s\mbbV^\tau+\kappa \mbbV^n)\,
\, \rmd \sigma.\\
   \end{aligned}
$$
Using \eqref{e:theta-kappa} to replace $\kappa$ we integrate-by-parts
$$\begin{aligned}
\frac{\rmd}{\rmd t}\jmp{\gamma} 
&= \int_\mbbS \vtau'(\theta)(-\nabla_s\mbbV^n+\nabla_s\theta \mbbV^\tau) +
\vtau(\theta)(\nabla_s\mbbV^\tau+\nabla_s\theta\mbbV^n)\,\rmd \sigma,\\
&=\int_\mbbS (\vtau''(\theta)+\vtau(\theta))(\nabla_s\theta \mbbV^n)+\nabla_s(\vtau(\theta)\mbbV^\tau)\,\rmd \sigma =0.
\end{aligned}$$
The last equality follows from the identity  \eqref{e:trig}, and the periodicity of $\theta$. 
\end{proof}

The first and second fundamental from evolution equations \eqref{e:FSFF} can be written as a linear map from the extrinsic  vector field $\mbbV=(\mbbV^n,\mbbV^\tau)^t$ to the intrinsic vector field $U_t=(\kappa_t, g_t)^t,$ 
\beq\label{e:In-Ext_VF} 
U_t=
\cM \mbbV,
\eeq
through the operator
\beq \label{e:Meq}
\cM:= \begin{pmatrix} \mrG & \nabla_s \kappa \cr
  g\kappa & g\nabla_s
 \end{pmatrix}.
\eeq
Here we have introduced the surface Helmoltz operator $\mrG:=-\Delta_s -\kappa^2.$ The operator $\mrG$ is not strictly positive, in particular
$$ \langle \mrG 1,1\rangle_{L^2(\mbbS)} = -\langle 1, \kappa^2\rangle_{L^2(\mbbS)}<0,$$
where the pairing denote the usual $L^2(\mbbS)$ inner product. The operator satisfies the identities
\beq \label{e:G_ident}
\mrG \vtau = n\nabla_s \kappa, \hspace{0.5in}
\mrG n =-\vtau \nabla_s \kappa,
\eeq
which allow a characterization of the kernel of $\cM.$
\begin{lemma}\label{lem:kerM}
For any smooth curve $\gamma$ the operator $\cM$ subject to periodic boundary conditions has a three dimensional kernel corresponding to rigid body motions of the curve, 
$$ \ker\cM =\Span\left\{\begin{pmatrix}e_1\cdot n \cr e_1\cdot \vtau\end{pmatrix},\begin{pmatrix}e_2\cdot n \cr e_2\cdot \vtau\end{pmatrix}, \begin{pmatrix}\gamma\cdot \vtau \cr -\gamma\cdot n\end{pmatrix} \right\}\subset L^2(\mbbS),$$
where $\{e_1,e_2\}$ denote the canonical basis of $\mbbR^2.$
\end{lemma}
\begin{proof}
For any fixed vector $v_0\in\mbbR^2$ the extrinsic vector field $(\mbbV^n,\mbbV^\tau):=(v_0\cdot n, v_0\cdot\vtau)$ is the infinitesimal generator of rigid translation of $\Gamma$ in the direction $v_0$. The identities \eqref{e:G_ident} and the relations \eqref{e:tau-n} yield
$$\cM \begin{pmatrix} v_0\cdot n \cr v_0\cdot \vtau\end{pmatrix}=0.$$
The rigid body rotation of curve induced by $\gamma$ at fixed rate about the origin induces the extrinsic velocity  $\mbbV=(\mrJ \gamma\cdot n, \mrJ \gamma \cdot\vtau)^t$ where $\mrJ$ is the $2\times2$ matrix for rotation by $\pi/2.$ The $\mrG$ identities and the relations \eqref{e:tau-n} again yield
$$ \cM \begin{pmatrix}\mrJ\gamma\cdot n\cr
\mrJ\gamma\cdot\vtau\end{pmatrix}= 0. $$
However $J\gamma \cdot n=\gamma\cdot \vtau$ while $J\gamma\cdot\vtau=-\gamma\cdot n.$ These three vectors are linearly independent and since $\cM$ is a third order operator, its eigenspaces cannot be more than three dimensional, hence these vectors span the kernel. 
\end{proof}

\subsection{Reparameterization}
The image of $\gamma$ is denoted
$$\Gamma(\gamma):=\{\gamma(s)\,\bigl|\, s\in\mbbS\}.$$
The image  is independent of re-parameterization of the map $\gamma.$ More specifically, let $m:\mbbS\mapsto\mbbS$ be 1-1 and continuously differentiable. Without loss of generality we may assume that $\partial_s m>0$, uniformly on $\mbbS.$ The image 
 of the original and reparameterized map coincide $\Gamma(\gamma)=\Gamma(\tilde\gamma)$. From the chain rule
$$ \partial_s \tilde\gamma= \partial_s\gamma \partial_s m,$$
so that 
$$\tg(s):=|\partial_s \tgam|=g(m(s)) \partial_s m.$$
In particular given an arc length $g$ it can reparameterized into a $\mbbS$-constant $\tg$ via the map $m$ which solves the nonlinear ODE
$$ \partial_s m= \frac{\tg}{g(m(s))},$$
subject to $m(0)=0$. The constant value of $\tg$ is chosen to satisfy the periodicity condition $m(|\mbbS|)=|\mbbS|$ and is unique since the value $m(|\mbbS|)$ increases monotonically with $\tg$. This is the parameterization taken in the computations presented in Section\,\ref{s:numer}, however in the analysis  the arc length is treated generally to emphasize the structure of the systems.

Denoting $\ts=m(s),$  then for any $f:\mbbS\to X$ its reparameterized form $\tilde f:=f\circ m,$ satisfies
$$\tilde\nabla_{s}\tilde f:= \frac{1}{\tg}\partial_s \tilde f = \frac{1}{g}\frac{1}{\partial_s m}\partial_s f \partial_s m= (\nabla_s f)\circ m.$$
Similarly, under the change of variables, $\rmd\tilde{\sigma}=\rmd\sigma$ while the reparameterized tangent, $\tilde\tau=\tau\circ m,$ normal $\tilde n=n\circ m,$ and curvature $ \tilde\kappa=\kappa\circ m$ satisfy the analogous forms of \eqref{e:tau-n}. In particular the energies considered in the next section are independent of parameterization.
The equivalence of reparameterization and tangential extrinsic flow is established in Section\,\ref{s:reparam} of the Appendix.

\section{Critical points, Local Minima, and Gradient Flows of Local Intrinsic Energies}

A local intrinsic energy depends only upon the first and second fundamental forms, specifically the curvature and arc length and their local derivatives.
A generic local-intrinsic energy involving the curvature and its surface derivatives up to second order takes the form
\beq\label{e:LIE} \cE(U):= \int_\mbbS F(\Delta_s\kappa, \nabla_s\kappa,\kappa) \rmd \sigma,
\eeq
where  $U=(\kappa,g)^t\in\cA$ and $F:\mbbR^3\mapsto \mbbR$ is assumed to be smooth. Although the value of $\cE$ is independent of the choice of arc length $g$, it appears formally in the scaling of the surface measure $\rmd\sigma$ and of the surface differentials and in the constraints required for curve closure. 
We derive intrinsic and extrinsic representation for the variational derivatives of $\cE$ and define the associated $L^2(\mbbS)$ gradient flow of $\cE$  subject to the closed curve constraints.
\subsection{Lagrange Multiplier Formulation}
The functional $\cE$ is defined over the admissible space $\cA$ comprised of smooth arc length-curvature pairs that satisfy the curve closure relations \eqref{e:C1}-\eqref{e:C2}. Bounded energy generally does not enforce smoothness on the arc length, however the energy is invariant under reparameterization. In the time independent minimization problem for $\cE$ this allows one to restrict the admissible set to pairs $U$ with arc length $g$ that is constant over $\mbbS.$ Nonetheless, it is instructive for the gradient flow problem to consider the perturbation problem for $\cE$ without imposing spatially constant arc length. We consider a general class of admissible perturbation paths 
\beq\label{e:path}
\phi(\del)= U+\del U_1 +\frac{\del^2}{2} U_2 +O(\del^3),
\eeq
for $\del\in\mbbR$ with tangent $\partial_\del\phi(0)=U_1.$ 
The curve closure relations constrain these paths, so we expand both the energy and the constraint
\begin{align}
\label{e:E-Taylor}
\cE(\phi(\del))&=\cE(U)+
\del \langle \nabla_\mrI\cE, U_1\rangle_{L^2(\mbbS)} +
\frac{\del^2}{2}\left(\left\langle \nabla_\mrI\cE,U_2\right\rangle_{L^2(\mbbS)} +\left\langle \nabla_\mrI^2\cE\, U_1,U_1\right\rangle_{L^2(\mbbS)}\right)+O(\del^3),\\
\label{e:C-Taylor}
\cC(\phi(\del))&= \del \langle \nabla_\mrI\cC, U_1\rangle_{L^2(\mbbS)} +
\frac{\del^2}{2}\left(\left\langle \nabla_\mrI\cC,U_2\right\rangle_{L^2(\mbbS)} +\left\langle \nabla_\mrI^2\cC\, U_1,U_1\right\rangle_{L^2(\mbbS)}\right)+O(\del^3),
\end{align}
where $\nabla_\mrI$ and $\nabla^2_\mrI$ denote first and second  $L^2(\mbbS)$-variational derivatives with respect to variations in the intrinsic coordinate $U$. Unless otherwise indicated, all instances of $U$ correspond to its value at the perturbation point $\phi(0)$, in particular this applies to the measure $\sigma=\sigma(g)$ associated to $L^2(\mbbS)$. Since $\phi(\del)\in\cA$ we have $\cC(\phi)=0$ and equating orders of $\del$ in the $\cC$ expansion yields three orthogonality conditions that characterize the tangent plane
$$\cA'(U)=\left\{U_1\in H^2(\mbbS)\,\bigl|\, U_1\bot\nabla_\mrI\cC(U)\right\}.$$
At second order in $\delta$, the system places the correction term $U_2$ in a quadratic relation with $U_1,$
\beq\label{e:AbsC2} \left\langle\nabla_\mrI \cC, U_2\right\rangle_{L^2(\mbbS)}=-
\left\langle \nabla_\mrI^2\cC\, U_1,U_1\right\rangle_{L^2(\mbbS)}.
\eeq
The function $U\in\cA$ is a critical point of $\cE$ if and only if the coefficient of $\del$ in the energy expansion is zero, equivalently
$\nabla_\mrI \cE\bot U_1$ for all $U_1\in\cA'(U).$  
This requires that the first variation of $\cE$ lies in the row space of $\nabla_\mrI\,\cC(U).$ Equivalently, critical points of $\cE$ satisfy the, generally elliptic, PDE 
\beq
\label{e:Abs-CP} \nabla_\mrI \cE(U)=\Lambda\cdot \nabla_\mrI\,\cC(U),
\eeq
where $\Lambda=(\lambda_1,\lambda_2,\lambda_3)\in\mbbR^3$ is a Lagrange multiplier associated to the constraint.
If in addition $U$ is a local minima of $\cE$ it is necessary that the  $\del^2$-coefficient in the expansion be non-negative. The quadratic correction $U_2$ can be removed from this term via the critical point equation and the quadratic relation with $U_1.$ That is applying first \eqref{e:Abs-CP} and then \eqref{e:AbsC2} we have
\beq \label{e:Abs-2ndVar1}
\left\langle \nabla_\mrI\cE,U_2\right\rangle_{L^2(\mbbS)} =
\left\langle \Lambda\cdot \nabla_\mrI \cC,U_2\right\rangle_{L^2(\mbbS)} =-\Lambda\cdot
\left\langle \nabla_\mrI^2\cC\, U_1,U_1\right\rangle_{L^2(\mbbS)}.
\eeq
The non-negativity of the coefficient of $\del^2$ in the energy expansion is equivalent to the condition
\beq 
\label{e:Bilinear}
\left\langle \left(\nabla_\mrI^2\cE-\Lambda\cdot\nabla_\mrI^2 \cC\right)\, U_1,U_1\right\rangle_{L^2(\mbbS)}\geq 0,
\eeq
for all $U_1\in\cA'(U).$
We introduce the Lagrange second variational operator, $\cL_\Lambda$ associated to a critical point $U$ solving \eqref{e:Abs-CP} with Lagrange multiplier $\Lambda$,
\beq\label{e:Cons_2ndVar}
\cL_\Lambda:= \nabla_\mrI^2\cE -\Lambda\cdot \nabla_\mrI^2 \cC.
\eeq
This corresponds to the second variation of the Lagrange multiplier energy
\beq\label{e:LME}
\cE_{\Lambda}:=\cE-\Lambda\cdot\cC.
\eeq
In the next two sections we make these calculations explicit.
\subsection{Variation of the Constraints}
To compute the first and second variation requires regular expansions the surface measure and surface diffusion operators.
To $O(\del^3)$ the perturbed surface measure $\rmd\sigma^\delta:=\rmd\sigma(\phi(\delta))$ expands as
\beq\label{e:dsig_exp} \begin{aligned}
\rmd \sigma^\delta &= 
 & \Big(1 + \del \frac{g_1}{g} + \frac{\del^2}{2} \frac{g_2 }{g}\Big) g\rmd s + O(\del^3).
 \end{aligned}\eeq
We introduce the differential of $\rmd\sigma$ at $g$ with differential $\tg$
\beq\label{e:dsig1}
  \rmd \sigma_1(\tg) = \frac{\tg}{g}\rmd \sigma
\eeq
so that 
\beq\label{e:dsig_exp2}
\rmd\sigma^\delta = \rmd \sigma+\delta \rmd\sigma_1(g_1)+\frac{\delta^2}{2}\rmd\sigma_1(g_2)+O(\delta^3).
\eeq
The inverse arc length expansion
$$\begin{aligned}
    \frac{1}{g(\delta)} 
    = \left(1-\del \frac{g_1}{g} +\frac{\del^2}{2}\Big(\frac{2g_1^2}{g^2}- \frac{g_2}{g}\Big)\right)\frac{1}{g}+O(\del^3),
\end{aligned}$$
 yields the surface gradient expansion
\beq\label{e:nab_exp2}
 \nabla_s^\delta= \nabla_s + \delta\nabla_{s,1}(g_1) +\frac{\delta^2}{2}\big(\nabla_{s,2}(g_1)+\nabla_{s,1}(g_2)\big) +O(\delta^3),
 \eeq
 where we introduced the operators at $g$ in differential $\tg$ 
 $$\begin{aligned}
     \nabla_{s,1}(\tg)&=-\frac{\tg}{g}\nabla_s,\\
     \nabla_{s,2}(\tg)&=\frac{2\tg^2}{g^2}\nabla_s.
 \end{aligned}$$
The first constraint condition is given in terms of the  tangent angle $\theta$ defined through $\kappa$ via \eqref{e:theta-kappa}. Its $\del$-expansion
$$ \theta(\delta)=\theta+\del \theta_1+\frac{\del^2}{2}\theta_2 +O(\del^3)$$
satisfies
$$\begin{aligned}
\nabla_s^\delta \theta(\delta) &= 
\left(1-\del \frac{g_1}{g} +\frac{\del^2}{2}\Big(\frac{2g_1^2}{g^2}- \frac{g_2}{g}\Big)\right)\nabla_s\theta 
+ \left(\del-\del^2 \frac{g_1}{g}\right)\nabla_s\theta_1 +\frac{\del^2}{2}\nabla_s\theta_2 +O(\del^3),\\
&=\nabla_s\theta +
\del\left(\nabla_s\theta_1-\frac{g_1}{g} \nabla_s\theta\right)+
\frac{\del^2}{2}\left(\frac{2g_1^2-g_2g}{g^2}\nabla_s\theta-\frac{g_1}{g}\nabla_s\theta_1 +\nabla_s\theta_2\right)+O(\del^3).
\end{aligned}$$
We introduce the operator
\beq\label{e:D_def}
\mrD f:=\int_0^s f(\ts)\rmd\sigma(\ts),
\eeq
which inverts $\nabla_s$. From Fubini's Theorem its $L^2(\mbbS)$ adjoint takes the form
$$\mrD^\dag f:=\int_s^Lf(\ts)\rmd\sigma(\ts).$$
Clearly
$$\mrD f+ \mrD^\dag f=\int_\mbbS f\,\rmd \sigma, $$
so when acting on functions with zero mass the operator is skew adjoint $\mrD^\dag=-\mrD.$
Using \eqref{e:theta-kappa} we match orders of $\del$ in this expansion to those of $\kappa$. Integrating the orders of $\del$ with respect to $\rmd\sigma$, and back substituting for $\nabla_s\theta_k$, yields the relations
\beq
\label{e:theta-exp}
\begin{aligned}
\theta(s) &= \mrD\kappa,\\
\theta_1(s) &= 
\mrD\left(\kappa_1 +g_1\frac{\kappa}{g}\right),\\
\theta_2(s)&= \mrD\left(\kappa_2 +
g_2\frac{\kappa}{g} -2g_1^2\frac{\kappa}{g^2} +g_1\kappa_1\frac{1}{g}\right).
\end{aligned}
\eeq
Using the $\theta(\delta)$ and $\rmd\sigma^\delta$ expansions \eqref{e:dsig_exp}, the first closure relation admits the asymptotic form
$$\begin{aligned}
     0=\int_\mbbS \vtau^\delta  \rmd \sigma^\delta 
 %
=\del\int_\mbbS\left(\vtau^{\prime}\theta_1+\vtau\frac{g_1}{g}\right)\,\rmd \sigma 
      +\frac{\del^2}{2} \int_\mbbS\bigg(\vtau'\theta_2+\vtau''\theta_1^2 +  \frac{g_1\theta_1}{g} \vtau' +\vtau
      \frac{g_2}{g}
      \bigg)\,\rmd \sigma +O (\del^3),
\end{aligned}$$
where $\vtau^\delta=\vtau(\theta(\delta))$ and $\vtau=\vtau(\theta)$. Equating orders of $\del$ yields the integral conditions
\begin{align}
0&= \int_\mbbS\left(\vtau^{\prime}\theta_1+\vtau\frac{g_1}{g}\right)\,\rmd \sigma, \label{e:FC1a}\\
0&= \int_\mbbS\bigg(\vtau'\theta_2+\vtau''\theta_1^2 +  \frac{g_1\theta_1}{g} \vtau' +\vtau
      \frac{g_2}{g}
      \bigg)\,\rmd \sigma.
\label{e:FC2a}\end{align}
Replacing $\theta_1$ with its $\mrD$ representation and taking adjoints yields the integral condition
\beq 0=\int_\mbbS \kappa_1 \mrD^\dag\tau'
+g_1\left(\frac{\kappa}{g} \mrD^\dag\tau' + \vtau\frac{1}{g}\right) \,\rmd\sigma.
\eeq
Using the identity $\vtau''=-\vtau$ the quadratic correction integral condition \eqref{e:FC2a} takes the form
\beq
\begin{aligned}
0&=\int_\mbbS \kappa_2 \mrD^\dag\tau' +
g_2\frac{ \kappa\mrD^\dag\tau'  +\vtau}{g} 
+g_1\kappa_1 \frac{\mrD^\dag\tau' }{g}-2g_1^2 \frac{\kappa\mrD^\dag\tau' +\tau}{g^2} +
g_1\theta_1\frac{\vtau'}{g} -\vtau\theta_1^2\, \rmd \sigma,\\
&= \int_\mbbS \kappa_2 \mrD^\dag\tau'  +
g_2\frac{ \kappa\mrD^\dag\tau'  +\vtau}{g} 
+g_1\kappa_1 \frac{\mrD^\dag\tau' }{g}-2g_1^2 \frac{\kappa\mrD^\dag\tau' +\tau}{g^2}+\\
&\hspace{0.4in} \mrD^\dag\left(\frac{g_1\vtau'}{g}\right)\left(\kappa_1+g_1\frac{\kappa}{g}\right)-\tau\left(\mrD\left(\kappa_1+g_1\frac{\kappa}{g}\right)\right)^2\,\rmd \sigma.
\end{aligned}
\eeq
We introduce the vector functions
 \begin{align}\label{e:Psi1}
 \Psi_1&=\bpm  [\mrD^\dag\tau' ]_1\cr
              \frac{\kappa}{g}[\mrD^\dag\tau' ]_1+\frac{1}{g}[\vtau]_1\epm,\\  
 \Psi_2&=\bpm  [\mrD^\dag\tau' ]_2\cr
              \frac{\kappa}{g}[\mrD^\dag\tau' ]_2+\frac{1}{g}[\vtau]_2\epm, \label{e:Psi2}
\end{align}
where $[F]_j$ denotes the $j$th entry of a vector.
The integral conditions can be written as projections 
\begin{align} \label{e:FC1}
\left\langle U_1, \Psi_j\right\rangle_{L^2(\mbbS)}&=0,\\
\left\langle U_2,
\Psi_j
\right\rangle_{L^2(\mbbS)} 
&=\left\langle \left[\nabla_{\mrI}^2 \cC\right]_j U_1, U_1\right\rangle_{L^2(\mbbS)},
\end{align}
for $j=1,2$. Here the second variation of the constraint system $\cC$ is a $2\times2\times 3$ three-tensor with entries
\beq\label{e:cC-Hessian}
\begin{aligned}
\left[\nabla_{\mrI}^2 \cC\right]_j&=
\bpm 0 & -\frac{1}{2g}[\mrD^\dag\tau' ]_j \cr -\frac{1}{2g}[\mrD^\dag\tau' ]_j & \frac{2}{g^2}[\kappa\mrD^\dag\tau' +\vtau]_j\epm-
\bpm 0 & \frac{[\vtau']_j}{2g}\mrD\smbull\cr
\mrD^\dag\!\left( \frac{[\vtau']_j}{2g}\smbull\right)& 
\frac{[\vtau']_j}{g}\mrD\!\left(\frac{\kappa}{2g}\smbull\right)+
\frac{\kappa}{2g} \mrD^\dag\!\left( \frac{[\vtau']_j}{g} \smbull\right)
\epm+\\
&\hspace{0.2in}
\bpm 
\mrD^\dag\left([\vtau]_j\mrD\smbull\right) & 
\mrD^\dag\left([\vtau]_j\mrD\left(\frac{\kappa}{g}\smbull\right)\right) \cr
\frac{\kappa}{g} \mrD^\dag\left([\vtau]_j\mrD\smbull \right) &
\frac{\kappa}{g}\mrD^\dag\left([\vtau]_j\mrD\left(\frac{\kappa}{g}\smbull\right)\right)
\epm,
\end{aligned}
\eeq
for $j=1,2$.
For the second closure relation, expanding $\kappa(\delta)$ and the surface measure yields
$$
 2\pi=\int_\mbbS 
 \kappa(\delta) \rmd \sigma^\delta =2\pi+\del\int_\mbbS \left(\kappa_1+\frac{g_1\kappa}{g}\right)\rmd \sigma +\frac{\del^2}{2}\int_\mbbS\left(\kappa_2+\frac{\kappa_1g_1}{g}+\frac{\kappa g_2}{g}\right)\rmd \sigma +O(\del^3).
 $$
 Equating orders of $\del$ yields the relations
 \beq\begin{aligned}
 0&=  \int_\mbbS \left(\kappa_1+g_1\frac{\kappa}{g}\right)\rmd\sigma,\\
  0&= \int_\mbbS\left(\kappa_2+g_2\frac{\kappa}{g}+ \frac{\kappa_1g_1}{g}\right)\rmd \sigma.
 \end{aligned}
\eeq
We introduce the vector function
\beq\label{e:Psi3}
\Psi_3 = \bpm 1 \cr \frac{\kappa}{g}\epm,
\eeq
and write the second closure constraint conditions as projection conditions
\begin{align}\label{e:SC1}
    \left\langle U_1, \Psi_3\right\rangle_{L^2(\mbbS)}&=0, \\
    \label{e:SC2}
    \left\langle U_2 , \Psi_3 \right\rangle_{L^2(\mbbS)} &=\left\langle [\nabla_\mrI^2\cC]_3 U_1, U_1 \right\rangle_{L^2(\mbbS)}.
\end{align}
Here the three-tensor $\nabla_{\mrI}^2\cC$ has third entry
\beq \label{e:cC-Hessian3}
\left[\nabla_{\mrI}^2\cC\right]_3=\bpm 0 & -\frac12 \cr -\frac12 & 0\epm.
\eeq

\begin{prop}
    The tangent space of $\cA$ at $U=(\kappa,g)^t\in\cA$ is given by
    \beq\label{e:cA'-def}
    \cA'(U)=\{\Psi_1,\Psi_2,\Psi_3\}^\bot
    =\left\{\cM(U) \begin{pmatrix} \mbbV^n\cr \mbbV^\tau\end{pmatrix}\,\Bigl|\, \mbbV\in H^4(\mbbS)\right\},
    \eeq
    where the vector functions are defined \eqref{e:Psi1}-\eqref{e:Psi2}-\eqref{e:Psi3} and the perpendicular is with respect to the $L^2(\mbbS)$ inner product.
\end{prop}
\begin{proof}
The first statement follows from  \eqref{e:FC1} and \eqref{e:SC1}. For the second statement, one can verify by direct calculation that
$\ker(\cM^\dag)=\Span\{\Psi_1,\Psi_2,\Psi_3\}.$ From the Fredholm alternative the range of $\cM$ is the perpendicular complement of this set.
\end{proof}

\subsection{Variations of the Energy}
Expanding the energy $\cE$ along the path $\phi$ requires expansions of $\rmd \sigma$, $\nabla_s$ and $\Delta_s.$
The first two are given in \eqref{e:dsig_exp} and \eqref{e:nab_exp2} respectively. For the third we apply \eqref{e:nab_exp2} twice 
$$
\Delta_s^\delta = \left(1-\del \frac{g_1}{g} +\frac{\del^2}{2}\Big(\frac{2g_1^2}{g^2}- \frac{g_2}{g}\Big)\right)
\nabla_s \left(\left(1-\del \frac{g_1}{g} +\frac{\del^2}{2}\Big(\frac{2g_1^2}{g^2}- \frac{g_2}{g}\Big)\right)\nabla_s\smbull\right)+O(\delta^3),
$$
and introduce the operators
$$ \begin{aligned}
 \Delta_{s,1}(\tg)&:=  -  \frac{\tg}{g} \Delta_s\smbull- \nabla_s\left(\frac{\tg}{g} \nabla_s\smbull\right),\\
 \Delta_{s,2}(\tg)&:= \frac{2\tg^2}{g^2}\Delta_s + \nabla_s \left(\frac{2\tg^2}{g^2}\nabla_s\smbull\right)
+ \frac{\tg}{g}\nabla_s\left(\frac{\tg}{g}\nabla_s\smbull\right),
\end{aligned}$$
for which the surface diffusion admits the expansion
\beq\label{e:Del_exp2}
\Delta_s^\delta = \Delta_s +\delta \Delta_{s,1}(g_1) +\frac{\delta^2}{2}\left(\Delta_{s,2}(g_1)+\Delta_{s,1}(g_2)\right) +O(\delta^3).
\eeq
With this notation the surface diffusion of curvature admits the expansion 
$$
    \Delta_s^\delta\kappa(\delta) 
    = \Delta_s\kappa +\delta\left( \Delta_s \kappa_1 + \Delta_{s,1}(g_1)\kappa)\right)+ \frac{\delta^2}{2}\left(\Delta_s\kappa_2 +2\Delta_{s,1}(g_1)\kappa_1+\Delta_{s,2}(g_1)\kappa+\Delta_{s,1}(g_2)\kappa\right)+O(\delta^3).
$$
To illuminate the structure within this expansion we introduce the surface differential operators defined at $U$ with differential $\tU,$
$$\begin{aligned}
H_{21}(\tU)&:= \Delta_s\tilde\kappa+\Delta_{s,1}(\tg)\kappa, &
H_{22}(\tU)&:= 2\Delta_{s,1}(\tg)\tilde\kappa+\Delta_{s,2}(\tg)\kappa,\\
H_{11}(\tU)&:= \nabla_s\tilde\kappa+\nabla_{s,1}(\tg)\kappa, &
H_{12}(\tU)&:=2\nabla_{s,1}(\tg)\tilde\kappa+\nabla_{s,2}(\tg)\kappa,\\
H_{01}(\tU)&:= \tilde\kappa, &
H_{02}(\tU)&:=0.
\end{aligned}$$
This notation allows a compact expansion of curvature gradients
\beq\label{e:SD_exp}
(\nabla_s^\delta)^{(k)}\kappa(\delta) = (\nabla_s)^{(k)}\kappa +\delta H_{k1}(U_1)+\frac{\delta^2}{2}\left(H_{k2}(U_1)+H_{k1}(U_2)\right) +O(\delta^3),
\eeq
for $k=0, 1, 2.$
Applying the surface differential expansions \eqref{e:SD_exp} and the surface measure expansion \eqref{e:dsig_exp2}, to the $\phi$-path energy Taylor expansion \eqref{e:E-Taylor}, 
the first variation follows from the chain rule,
\beq\label{e:first-var_0}
\begin{aligned}
\partial_\del\cE(\phi)\bigl|_{\del=0}&=\int_\mbbS  \nabla_\mrI\cE\cdot U_1 \rmd \sigma
=\int_\mbbS \left(F_2 H_{21}+F_1H_{11}+F_0H_{01}\right)\,\rmd \sigma+\int_{\mbbS}F \rmd\sigma_{1}(g_1).
\end{aligned}\eeq
Here, for $k=0,1,2$, $F_k$ denotes the partial derivative of $F$ with respect to its dependence upon the k'th-order surface gradient  $(\nabla_s)^{(k)}\kappa$.  For simplicity we have used the compressed notation $H_{i1}=H_{i1}(U_1)$ and $F_{i}=F_{i}(U)$ for each $i=0,1,2.$
Grouping factors of $U_1$ and integrating by parts yields an explicit representation for the first intrinsic variation 
\beq\label{e:first_var}
\nabla_{\mrI}\cE(U)=\bpm \partial_\kappa \cE\cr\partial_g \cE\epm\!\!\Bigl|_{U}=
\begin{pmatrix}
   \Delta_s F_2-\nabla_sF_1+F_0 \cr
    \frac{1}{g}\left(
    -F_2\Delta_s\kappa+\nabla_s F_2\nabla_s\kappa - F_1 \nabla_s\kappa +F\right)
\end{pmatrix}.
\eeq
At a critical point $U\in\cA$ the intrinsic variation is orthogonal to all $U_1\in\cA'(U).$ To establish that a critical point is a local minimum requires non-negativity of  the second variation,
$$\partial_\del^2\cE(\phi)\bigl|_{\delta=0}=\int_\mbbS\left( U_1^t\left[\nabla_\mrI^2\cE\right] U_1 +\nabla_\mrI \cE \cdot U_2\right)\,\rmd \sigma.$$
The representation for the energy contribution from $U_2$ follows  directly from the first variation \eqref{e:first_var}, and its contribution will be replaced by the Lagrange multiplier form of the second constraint variation. The second energy variation and its quadratic dependence on $U_1$ is a novel quantity. It admits the expansion
\beq\label{e:2var-1}
\left\langle \left[\nabla_\mrI^2\cE\right] U_1,U_1\right\rangle_{L^2(\mbbS)}= \sum_{i,j=0}^2 \int_\mbbS F_{ij}H_{i1}H_{j1} \rmd \sigma +
\sum_{i=0}^2\int_\mbbS 2F_iH_{i1}\rmd \sigma_1(g_1)+
\sum_{i=0}^2\int_\mbbS F_{i}H_{i2}\rmd \sigma,
\eeq
where for $i,j=0,1,2$ we denote $H_{ij}=H_{ij}(U_1)$ and $F_{ij}=F_{ij}(U).$ Using \eqref{e:dsig1} to replace $\rmd\sigma_1(g_1)$ and the formula \eqref{e:first-var_0} for the first variation yields
\beq\label{e:2var-2}
\left\langle \left[\nabla_\mrI^2\cE\right] U_1,U_1\right\rangle_{L^2(\mbbS)}=  \int_\mbbS\left( \sum_{i,j=0}^2 F_{ij}H_{i1}H_{j1} +
\sum_{i=0}^2 F_{i}H_{i2}
+2\left(\nabla_\mrI\cE\cdot U_1-F\frac{g_1}{g}\right)\frac{g_1}{g}\right)\,\rmd \sigma.
\eeq
Integration by parts in the $H_{ij}$ operators allows the right-hand side to be rewritten in the bilinear form presented in the left-hand side, yielding an explicit representation for the self-adjoint operator $\nabla_\mrI^2\cE.$ At this level of generality the exact expression is cumbersome. An explicit representation is pursued for a first-order energy in Section\,\ref{s:FOE}. 
At a critical point $U\in\cA$ the $U_2$ variation satisfies \eqref{e:Abs-CP} and can be replaced with the bilinear action of $U_1$ on the Hessian of the constraint, yielding the formal Lagrange multiplier expansion associated to \eqref{e:LME}
$$ \cE(\phi)=\cE(U) +\frac{\delta^2}{2}\left\langle \left(\nabla_\mrI^2\cE(U)-\Lambda\nabla_\mrI^2\cC(U)\right)U_1,U_1\right\rangle_{L^2(\mbbS)} +O(\delta^3),$$
that incorporates the impact of the curve closure constraints.


\subsection{$L^2(\mbbS)$ Gradient Flows}
An extrinsic velocity $\mbbV=(\mbbV^n,\mbbV^\tau)^t$ induces an evolution in $U$ through the relation \eqref{e:In-Ext_VF}. The chain rule implies that the time derivative of the energy evaluated on the evolving $U$ satisfies the  relation
\beq\label{e:ED1} \frac{d}{dt}\cE(U(t)) =
\int_\mbbS 
\nabla_\mrI \cE \cdot \cM 
\mbbV\, \rmd \sigma=
\int_\mbbS 
\cM^\dag\nabla_\mrI \cE \cdot
\mbbV\, \rmd \sigma =\int_\mbbS 
\nabla_\mrE \cE \cdot \mbbV\, \rmd \sigma,
\eeq
where the $L^2(\mbbS)$ transpose of $\cM$ is
    $$ \cM^\dag = \begin{pmatrix} G &  g\kappa \cr \nabla_s\kappa & -\nabla_s( g\cdot) \end{pmatrix}, $$ 
and we have introduced the $L^2(\mbbS)$ extrinsic variation of $\cE$,
\beq\label{e:CVE}
\nabla_{\mrE}\cE := \cM^\dag \nabla_\mrI\cE.
\eeq
We verify in Lemma\,\ref{l:Wind} (see Section\,\ref{s:reparam} of Appendix) that the second entry of the extrinsic variation, $[\nabla_\mrE \cE]_2$ is zero, and the tangential velocity $\mbbV^\tau$ drops out of the energy dissipation mechanism. More specifically, denoting the $k$th row of $\cM^\dag$ by $\cM^\dag_k,$ we have
\beq\label{e:W-invar}
\cM_2^\dag \cdot \nabla_\mrI \cE= \partial_\kappa \cE \nabla_s \kappa - \nabla_s (g \partial_g \cE)=0.
\eeq
If $U$ is a critical point of $\cE,$  then $\partial_\kappa\cE=0$, and hence $g\nabla_g\cE$ is independent of $s$. This has the interpretation that the arc-length variation, $\partial_g\cE$, is an arc length-weighted first integral of the functional variation of $\cE$ and is the flip-side of the fact that tangential velocity is equivalent to reparameterization and hence does not impact the energy, see  Lemma\,\ref{l:DV} in Section\,\ref{s:reparam} of Appendix.

\begin{remark}
 The arc length variation $\cE_g$ can be viewed as a weighted version of the Hamiltonian that is constant along orbits corresponding to critical points of the functional variation.  This structure was exploited by Alikakos and co-authors  in their fixed-domain construction of heteroclinic orbits that are critical points of a multi-component Allen-Cahn energies with several global minima,  \cite{ABC_06}, \cite{AF_08}. Indeed they relate critical points of the Allen-Cahn energy to critical points of a Jacobi functional
 $$ \int_\mbbS \sqrt{W(\gamma(s))}\,\rmd \sigma,$$
through a geometric version of the least action principle \cite{GH_96}.
 \end{remark}
 
These calculations motivate the following definition of the $L^2(\mbbS)$ gradient normal velocity of $\cE,$
\beq \label{e:L2grad}
\mbbV^n_{\!\cE}:= -\left[\nabla_\mrE \cE\right]_1=-\cM^\dag_1 \cdot  \nabla_\mrI \cE= - G\partial_\kappa\cE-g\kappa\partial_g\cE.
\eeq

For an arbitrary, smooth tangential velocity $\mbbV^\tau$ the gradient flow takes the form
\beq\label{e:LI-L2GF}
\begin{pmatrix} \kappa_t\cr g_t\end{pmatrix}=\cM \begin{pmatrix} \mbbV^n_{\!\cE} \cr \mbbV^\tau\end{pmatrix}=
-\begin{pmatrix} \mrG^2 \partial_k\cE+\mrG(g\kappa\partial_g\cE)\cr 
g\kappa \mrG\partial_\kappa\cE +(g\kappa)^2\partial_g\cE
\end{pmatrix} +
\begin{pmatrix}  \mbbV^\tau\nabla_s\kappa \cr g\nabla_s\mbbV^\tau \end{pmatrix},
\eeq
which affords the $\mbbV^\tau$-independent dissipation mechanism
$$ \begin{aligned}
\frac{d}{dt}\cE& = \int_\mbbS \cM^\dag \begin{pmatrix}  \partial_\kappa \cE\cr
\partial_g \cE
\end{pmatrix}  \cdot 
 \begin{pmatrix}  \mbbV^n_{\!\cE} \cr\mbbV^\tau \end{pmatrix}\,
\rmd \sigma
=\int_\mbbS \begin{pmatrix} -\mbbV^n_{\!\cE} \cr 0 \end{pmatrix}\cdot \begin{pmatrix} \mbbV^n_{\!\cE} \cr \mbbV^\tau\end{pmatrix}\,\rmd\sigma =
-\int_\mbbS|\mbbV^n_{\!\cE}|^2 \rmd \sigma.
\end{aligned}
$$
A complete description of the gradient flow in the intrinsic coordinates requires a choice of tangential velocity. Although a tangential velocity is equivalent to a reparameterization and does not impact the system energy, it cannot be ignored within the intrinsic equation system.  A common choice is the so-called ``co-moving'' frame, for which $\mbbV^\tau\equiv 0$.  A second choice is scaled arc length, for which $g$ is constant over $\mbbS$ but may  not be constant in time. We remark that a constant in time and space arc-length parameterization (called fixed arc length) requires that the size of the reference domain $\mbbS$ change in time. This has the consequence that the curvature and tangential velocities are no longer $\mbbS$-periodic, and greatly complications the formulation. Fixed arc-length is not considered here. For scaled arc-length the tangential motion satisfies
\beq\label{e:SArc}\nabla_s \mbbV^\tau=-\kappa\mbbV^n +\frac{\int_\mbbS \mbbV^n\kappa\,\rmd\sigma}{\int_\mbbS \rmd \sigma} = -\kappa \mbbV^N +\partial_t\ln|\Gamma|.
\eeq
More specifically since $g=g(t)$ is a spatial constant for scaled arc-length this implies that
$$ \mbbV^\tau(s)-\mbbV^\tau(0)=-\mrD(\mbbV^n\kappa) + a(s)\int_\mbbS \mbbV^n\kappa\,\rmd\sigma, $$
where $\mrD$ is given in \eqref{e:D_def}  and $a(s)$ is the percentage of total curve length contained in the image of the curve restricted to $[0,s].$
From an analytical point of view the scaled arc-length tangential velocity reduces the arc length evolution to a scalar ODE, but considerably complicates the curvature evolution through the inclusion of $\mbbV^\tau=\mbbV^\tau(\mbbV^n_\cE)$ in the right-hand side of \eqref{e:LI-L2GF} as a convective term. For the derivation of the linearized system we consider the co-moving gauge.

\subsection{First Order Energies}
\label{s:FOE}
We simplify the formulas for the intrinsic Hessian $\nabla_\mrI^2\cE$  for a class of energies
that involve surface gradients up to first order,
\beq\label{e:FOE}
\cE(U)=\int_\mbbS F(\nabla_s\kappa,\kappa)\,\rmd \sigma.
\eeq
The representation \eqref{e:first_var} of the $L^2(\mbbS)$ intrinsic gradient reduces to
\beq\label{e:FoE_IG}
\nabla_{\mrI}\cE=  \bpm -\nabla_s F_1 +F_0\cr
                        \frac{1}{g}\left(-F_1\nabla_s\kappa +F \right)\epm.
\eeq
 Similarly, restricting the second variation of $\cE$ given in \eqref{e:2var-1} to the first-order energy yields,
 \beq
 \begin{aligned}
 \left\langle [\nabla_\mrI^2\cE] U_1,U_1\right\rangle&=\int_\mbbS \Bigl(F_{11}\left(\nabla_s\kappa_1-\frac{g_1}{g}\nabla_s\kappa\right)^2 +2F_{10}
 \left( \nabla_s\kappa_1-\frac{g_1}{g}\nabla_s\kappa\right)\kappa_1+F_{00}\kappa_1^2 +\\
 &\hspace{0.15in} \left(F_1\left(\nabla_s\kappa_1-\frac{g_1}{g}\nabla_s\kappa\right) +F_0\kappa_1 \right)\frac{2g_1}{g} +F_1\left(-\frac{2g_1}{g}\nabla_s\kappa_1 +\frac{2g_1^2}{g^2}\nabla_s\kappa\right)
 \Bigr)\,\rmd\sigma,\\
 &=\int_\mbbS \Bigl(F_{11}\left(\nabla_s\kappa_1-\frac{g_1}{g}\nabla_s\kappa\right)^2 \!\!\!+2F_{10}
 \left( \nabla_s\kappa_1-\frac{g_1}{g}\nabla_s\kappa\right)\kappa_1+F_{00}\kappa_1^2 
 +F_0\kappa_1 \frac{2g_1}{g} \Bigr)\,\rmd\sigma.
 \end{aligned}
 \eeq
 Factoring out the terms with $U_1$ and integrating by parts we deduce that
 \renewcommand*{\arraystretch}{1.8}
 \beq
 \label{e:cF-Hessian}
 \nabla_\mrI^2\cE = 
 \bpm 
 -\nabla_s(F_{11}\nabla_s\smbull) -\nabla_s(F_{10}\smbull) + F_{10}\nabla_s\smbull +F_{00} & \nabla_s\left(\frac{F_{11}\nabla_s\kappa}{g}\smbull\right) -\frac{F_{10}\nabla_s\kappa}{g}+\frac{F_0}{g} \\
 \frac{1}{g}\left(-F_{11}\nabla_s\kappa\nabla_s\smbull -F_{10}\nabla_s\kappa +F_0\right)  & F_{11}\left(\frac{\nabla_s\kappa}{g}\right)^2 
 \epm.
 \eeq
 The $(1,1)$ entry of $\nabla_\mrI^2\cE$ is the linearization of $\partial_\kappa\cE$ with respect to $\kappa$, which we denote by $\mrL_0.$ Similarly, denoting by $\mrK_0$ the linearization of $g\partial_g\cE$ with respect to $\kappa$, we have the $L^2(\mbbS)$ self-adjoint form
 \beq\label{e:FOE-Hessian_2}
 \nabla_\mrI^2\cE = \bpm\mrL_0 & \mrK_0^\dag \left(\frac{\smbull}{g}\right) \\
\frac{1}{g}\mrK_0   &  \frac{F_{11}|\nabla_s\kappa|^2}{g^2}\epm. 
 \eeq
 \renewcommand*{\arraystretch}{1.1}
Together with the expressions \eqref{e:cC-Hessian} and \eqref{e:cC-Hessian3} for $\nabla_{\mrI}^2\cC$, the result \eqref{e:cF-Hessian} yields an explicit formulation for the constrained second variation $\cL_\Lambda$ defined in \eqref{e:Cons_2ndVar}.

The normal velocity associated to the first-order energy $\cE$ of the form \eqref{e:FOE} takes the form
\beq\label{e:LIE_NV}
\mbbV^n_{\!\cE} =-\cM^\dag_1\cdot 
\nabla_{\mrI}\cE =
-\mrG( -\nabla_s F_1+F_0) -\kappa(-F_1\nabla_s\kappa +F).
\eeq
As an example, the first order energy \eqref{e:FOE} includes the classical Canham-Helfrich energy as a special case. The Canham-Helfrich energy has no surface diffusion term and a quadratic curvature dependence with even parity, reducing the energy density to $F(\kappa)=\frac12\kappa^2+\beta$. 
The associated Canham-Helfrich normal velocity recovers the classical Willmore flow with a $\beta$ dependent mean curvature
\beq
\label{e:NV-CH}
\mbbV^n_{\rm CH}=\left(\Delta_s+\frac{\kappa^2}{2}-\beta\right)\kappa.
\eeq

 The gradient flow associated to \eqref{e:FOE} can be written in the form \eqref{e:LI-L2GF}, however for an equilibrium analysis it is illustrative to separate out both prefactors of $\cM$, writing the flow as
\beq\label{e:FGF}
\begin{pmatrix}\kappa_t\cr g_t
\end{pmatrix} = -\cM\cM^\dag \nabla_\mrI\cE 
=-\cM\cM^\dag
 \bpm -\nabla_s F_1 +F_0\cr
                        \frac{1}{g}\left(-F_1\nabla_s\kappa +F \right)\epm,
\eeq
subject to periodic boundary conditions on $\mbbS.$
At a general point $U\in\cA$ the linearization of the right-hand side  yields a complicated operator. However at an equilibrium $U=(\kappa,g)^t$ the relations $\nabla_\mrI\cE(U)=0$ allow the linearization to be significantly simplified. The first simplification arises from the fact that the variation of $\cM$ and $\cM^\dag$ drop out since they act on the zero term $\nabla_\mrI\cE.$ The second simplification arises from using the equilibrium equations to rewrite the $[\nabla_{\mrI}^2\cE]_{21}$ entry, which allows the $[\nabla_\mrI^2\cE]_{12}$ and $[\nabla_\cI^2\cE]_{21}$ terms to be identified as adjoints of each other. The result is the anticipated linearization 
\beq \label{e:FOE-Lin}
\cL:= -\cM\cM^\dag \nabla_\mrI^2\cE,
\eeq
with $\nabla_\mrI^2\cE$ taking the equilibrium form \eqref{e:FOE-Hessian_2}. For a general $U\in\cA$ the self-adjoint geometric prefactor takes the form
$$\cM\cM^\dag(U) =\bpm \mrG^2 +|\nabla_s\kappa|^2 & \mrG(g\kappa\smbull) -\nabla_s\kappa \nabla_s(g\smbull)\\
g\kappa\mrG\smbull +g\nabla_s(\nabla_s\smbull\kappa)&g^2\kappa^2-g\nabla_s(\smbull\nabla_sg)\epm.$$
Since $\cM\cM^\dag\geq0$, with a maximal three dimensional kernel, it has a non-negative square root $\cS=\cS(U)$, with a three dimensional kernel corresponding to rigid body motions of the underlying interface. At an equilibrium the operator $\cL$ is a product of a non-negative, self-adjoint operator and a self-adjoint operator. This is evocative of the structure exploited in the spectral analysis of the quasi-steady states of the  Cahn-Hilliard equation, \cite{XChen_94}. While a comprehensive analysis of the spectrum of $\cL$ is an open problem, the kernel of $\cL$ is amenable to symmetry. For an equilibrium $U$ we act with the infinitesimal generator of translations, $\nabla_s$,  on the critical point equation $\nabla_\mrI\cE(U)=0$ to obtain
 $$ 0=\nabla_s\left(\bpm \mrI & 0\\ 0 & g\epm \nabla_\mrI\cE(U) \right) =\nabla_s\bpm\partial_\kappa\cE\\ g\partial_g\cE\epm=\bpm \mrL_0 \\ \mrK_0\epm \nabla_s\kappa,$$
 from which we infer the translational element of the kernel of the Hessian,
 \beq \label{e:Hessian_Kernel}
 \bpm \nabla_s\kappa\cr 0 \epm\in \ker\left(\nabla_\mrI^2\cE(U)\right).
 \eeq

 \begin{figure}
 \includegraphics[width=6in]{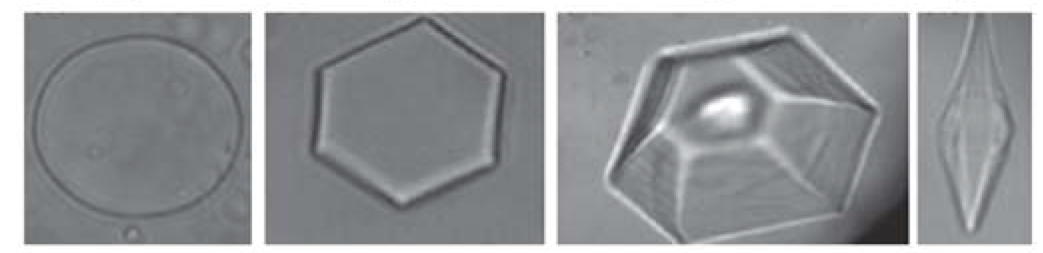}
 \caption{\small Faceting in  ice crystals. (left to right) The impact of increasing density of antifreeze glycoproteins on morphology, \cite{Gibson10}.
}
\label{f:Gibson}
 \end{figure} 

\section{Application: Faceting Energy}
To investigate quasi-adiabatic dynamics in a gradient flow, we consider a ``faceting energy'' within the first-order framework \eqref{e:FOE} with 
$$F(\nabla_s\kappa,\kappa)=\frac{\alpha^2}{2}|\nabla_s\kappa|^2 +W(\kappa),$$
 where $W$ is a  double well potential with minima at $\kappa=0$ and $\kappa=\kappa_*.$ This is an Allen-Cahn energy in the curvature, different than an Allen-Cahn energy for a scalar embedded in the interface. Since $W$ has regions of concavity the surface diffusion is required to make the system well-posed. The mass of $\kappa$ is conserved under \eqref{e:C2}, so a supporting line can be added to the well and there is no loss of generality in assuming that $W$ has equal depth wells. This is a simplistic model of interfaces formed by the freezing of water blended with antifreeze polymers that lower the freezing point. The polymers favor flat interfaces, or interfaces with large curvatures whose induced radius $1/\kappa$ is smaller than the radius of gyration of the polymer, see Figure\,\ref{f:Gibson}.  That is, high interface curvatures exclude the antifreeze polymers. A penalty term that constrains the perimeter of the interface is added as a simple proxy for enclosed volume. This accounts for the relation between freezing temperature and polymer density within the enclosed liquid, driving the enclosed liquid volume to a value predetermined by available polymer. The combined energy takes the form
\beq\label{e:Facet}
\cE_{\rm Fc}=  \frac{\beta}{2}(|\Gamma|-L_*)^2+\int_\mbbS \frac{\alpha^2}{2}|\nabla_s \kappa|^2 +W(\kappa)\,\rmd \sigma,
\eeq
where $|\Gamma|$ denotes the length of the curve $\Gamma$ and $L_*$ is a reference length. 
The curve length can be written as
$$ |\Gamma|=\int_\mbbS \rmd \sigma,$$
so that the perimeter energy can be written as
$$\beta(|\Gamma|-L_*)^2= \beta\left(\int_\mbbS \rmd \sigma-L_*\right)^2.$$
Using \eqref{e:FoE_IG} it is straight-forward to incorporate the perimeter energy yielding
$$\nabla_\mrI \cE_{\rm Fc}= \bpm -\alpha^2\Delta_s\kappa +W'(\kappa) \cr \frac{1}{g}\left(-\frac{\alpha^2}{2}|\nabla_s\kappa|^2 +W(\kappa)+\beta(|\Gamma|-L_*)\right)\epm, $$
and the associated faceting gradient-flow normal velocity 
\beq \label{e:Facet-NV} \mbbV^n_{\rm Fc}= -\mrG \left(-\alpha^2 \Delta_s\kappa +\mrW '(\kappa)\right) -\kappa\left( -\frac{\alpha^2}{2}|\nabla_s\kappa|^2+ \mrW(\kappa) +\beta(|\Gamma|-L_*)\right).
\eeq
\subsection{Formal Equilibrium and Quasi-Steady Analysis}
We provide a rough sketch of a construction of families of equilibrium and quasi-steady faceted states of the Faceting energy, emphasizing how fixed-domain tools can be brought to bear upon these geometric problems.
The structure implied in the functional-arc length variation relation \eqref{e:W-invar}
impacts the critical point construction.   Phase-plane methods allow the construction of families of equilibrium from the solutions of the $\mbbS$-periodic system
\beq\label{e:Facet-CP} 
\alpha^2 \Delta_s\psi -W'(\psi)=\nu_1, \eeq
where $\nu_1\in\mbbR$ is a free parameter associated to the total curvature constraint. Periodic solutions can be constructed by reparameterizing to $\mbbS$-constant arc length scaled so that the solution is $|\mbbS|$-periodic. For small values of $\alpha>0$ and $\nu_2<0$ this yields periodic $N$-pulse solutions that take  values near $0$ and $\kappa_*,$ the two minima of $W$.    
Using $\nabla_s\psi$ as an integrating factor yields the relation
\beq\label{e:Fint}
\nabla_s\left(\frac{\alpha^2}{2}|\nabla_s\psi|^2-W(\psi)-\nu_1\psi \right)=0,
\eeq
and the associated first integral of the system
$$ \frac{\alpha^2}{2}|\nabla_s\psi|^2-W(\psi)=\nu_1\psi+\nu_2,$$
for some $\nu_2\in\mbbR.$ The arc length variation satisfies
$$ g\partial_g\cE_{\rm Fc}(\psi)= \nu_1\psi+\nu_2 +\beta(|\Gamma|-L_*).$$

Evaluated at these special solutions the $\cE_{\rm Fc}$-gradient normal velocity reduces to
$$ \mbbV^n_{\rm Fc}(\psi)=-G\nu_1-\psi\left(\nu_1\psi+\nu_2+\beta(|\Gamma|-L_*)\right)=-\nu_2-\beta(|\Gamma|-L_*),$$
which for a fixed value of $\nu_2$ can be tuned to zero by adjusting either the equilibrium length $L_*$ or the coefficient $\beta.$ To generate an admissible equilibrium it remains to satisfy the closure conditions \eqref{e:C1}-\eqref{e:C2}. For this we outline a process. For a fixed value of $\alpha$ the area condition \eqref{e:C2} is a proxy for the sum of the interior angles in a regular $N$-gon. This condition is easily met since the area integral increases as $\nu_2\to 0_+$, but decreases with $\alpha.$  Adjusting the number of pulses $N$ one can make the net angle within $1/N$ of $2\pi$. The value of $\nu_1$ can be adjusted to yield an exact angle condition. The first closure relation ensures the curve terminates where it started. Exploiting periodicity, breaking the domain $\mbbS$ into $N$ equal pieces, $\gamma$ traces out identical curves over each piece. Moreover, for $\alpha\ll1$ the image of each piece approximates a corner of angle $2\pi/N$. Such shapes can generically be assembled in a $\mbbC^1$ fashion into a closed curve.

A more flexible construction yields quasi-stationary solutions by splicing together heteroclinic connections of the critical point equation on the line. Without loss of generality arc length $g$ is taken to be $\mbbS$-constant. Let $\psi_h$ be the heteroclinic solution to \eqref{e:Facet-NV} with $\nu_1=0$ on $\mbbR$ that connects $0$ to $\kappa_*.$
To satisfy the closure constraints \eqref{e:C1}-\eqref{e:C2} we fix $N\in{\mathbb N}_+$ and distance $\ell>0$, modifying $\psi_h$ to $\tilde\psi_h$ given by
\beq\label{e:psi_h}
\tilde\psi_h(s) =\left\{ 
\begin{array}{lcr}
\psi_h(s) & |s|<\alpha \ell,\\
0 & s<-2\alpha\ell,\\
\kappa & s>2\alpha\ell,
\end{array}\right. 
\eeq
 and is smooth and monotone for $\alpha\ell\leq|s|\leq 2\alpha\ell$.  So defined, $\tilde\psi_h$ is exponentially close, $O(e^{-\nu/(\alpha\ell)})$, to $\psi_h$ as $\alpha\to 0$ for some $\nu>0.$ 
 We define the $2N$ front quasi-equilibrium 
$$\psi_N:=\sum_{i=1}^{2N} (-1)^{i+1}\tilde\psi_h(s-p_i),$$
where $0<p_1<\cdots <p_i < p_{i+1}<\cdots <p_{2N} <|\mbbS|$ is a partition of $\mbbS$ into successive up and down heteroclinic orbits.  Under the condition that $|p_i-p_{i+1}|>4\alpha\ell$ the localized fronts do not overlap and $\phi_N$ has $N$ regions where it equals $\kappa_*$ and $N$ where it equals zero.
The function $\phi_N$ can be translated so that $\phi_N(0)=\phi_N(|\mbbS|)=0$, verifying that it is $\mbbS$-periodic. 
The closure constraints impose three restrictions on the shift locations,
%
accounting for translational invariance leaves $2N-4$ free parameters. Adjusting $L_*$ to balance the length eliminates the perimeter term and $\phi_N$ satisfies the first-integral equation up to exponentially small terms.  That is, $U_N:=(\psi_N,g)^t\in\cA$ exactly, while the functional and arc length variations satisfy
$$\begin{aligned}
\cR_{\kappa}&:=\partial_\kappa\cE(\psi_N)= -\alpha^2\Delta_s\psi_N+W'(\psi_N)\sim \delta, \\
\cR_g&:=g \partial_g\cE (\psi_N)= -\frac{\alpha^2}{2}|\nabla_s\psi_N|^2+W(\psi_N)+\beta(|\Gamma|-L_*)\sim \delta,
\end{aligned} $$
where $\delta=e^{-\nu/\alpha}$ characterizes the exponentially small residuals. 

\begin{figure}
 \begin{tabular}{ccccc}
\includegraphics[width=1.2in]{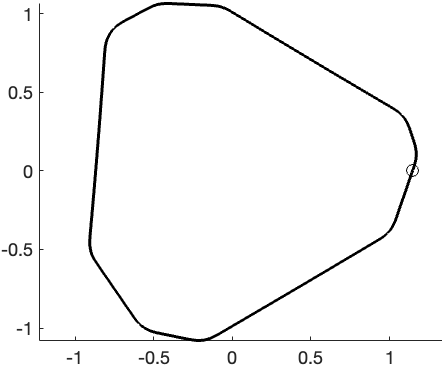}&
\includegraphics[width=1.2in]{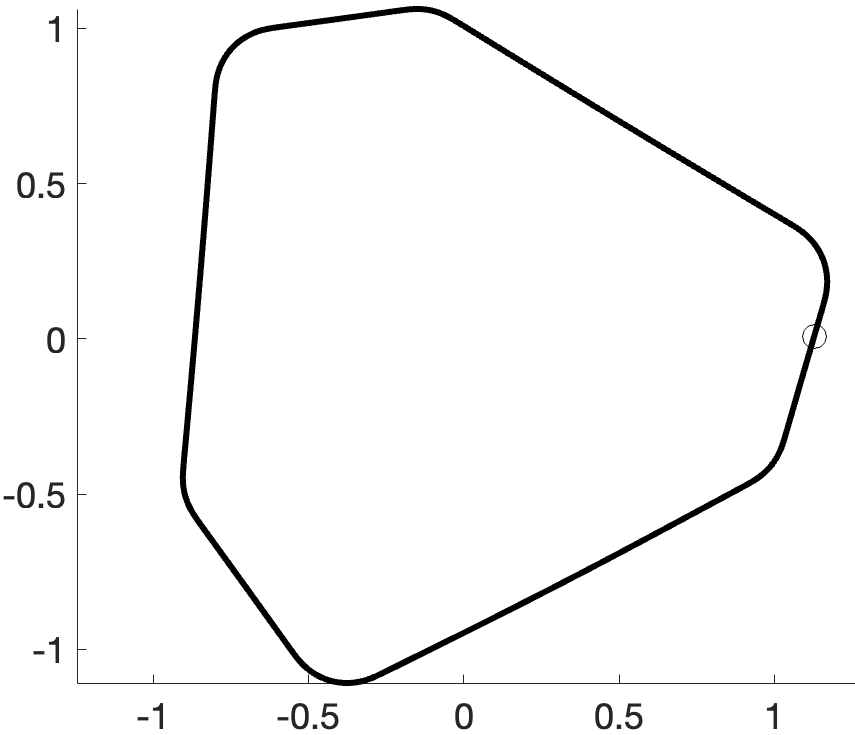}&
\includegraphics[width=1.2in]{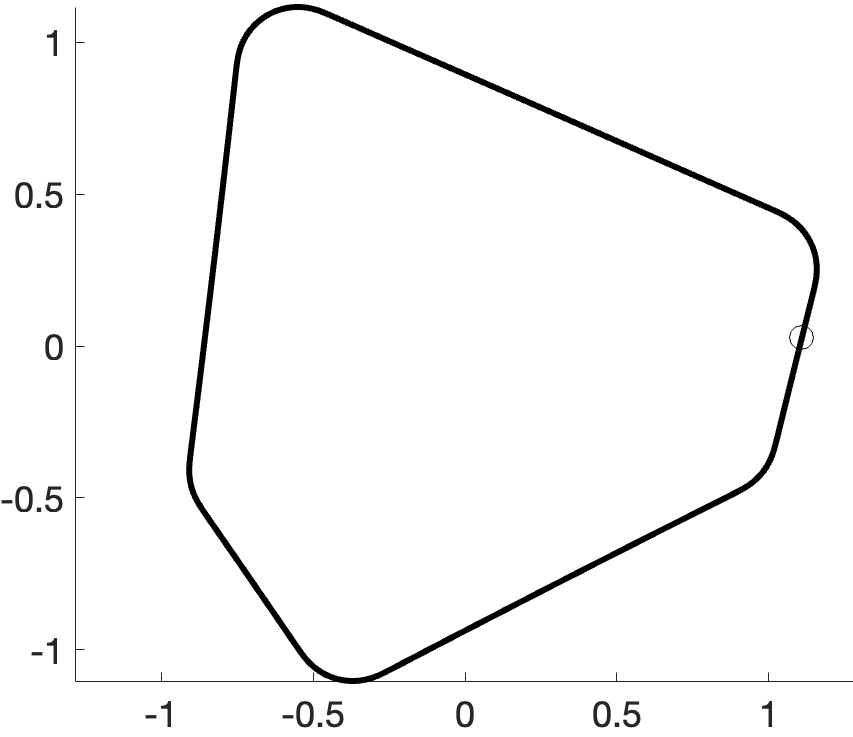}&
\includegraphics[width=1.2in]{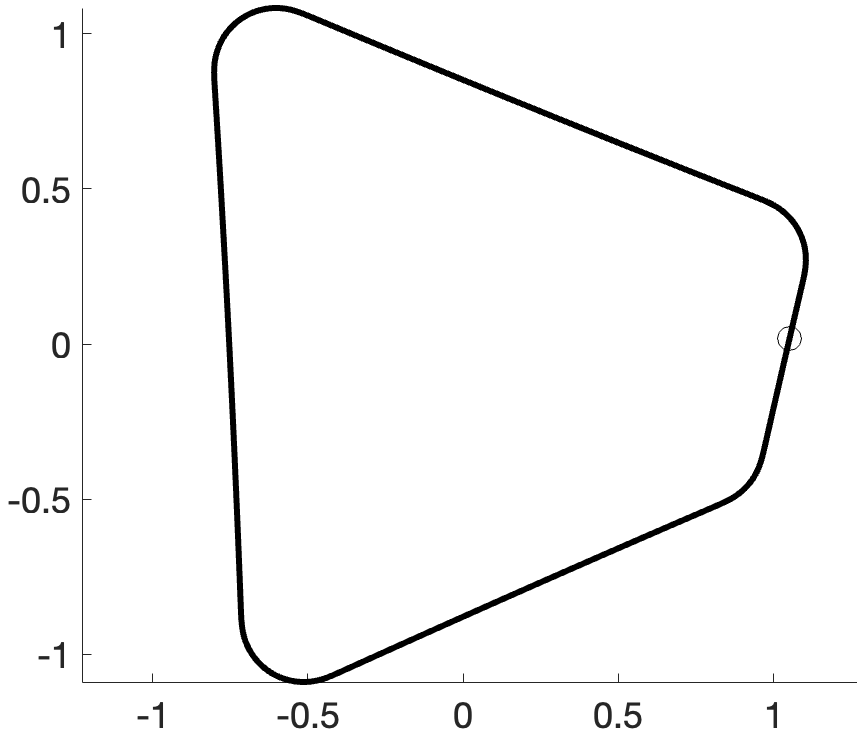} &
\includegraphics[width=1.2in]{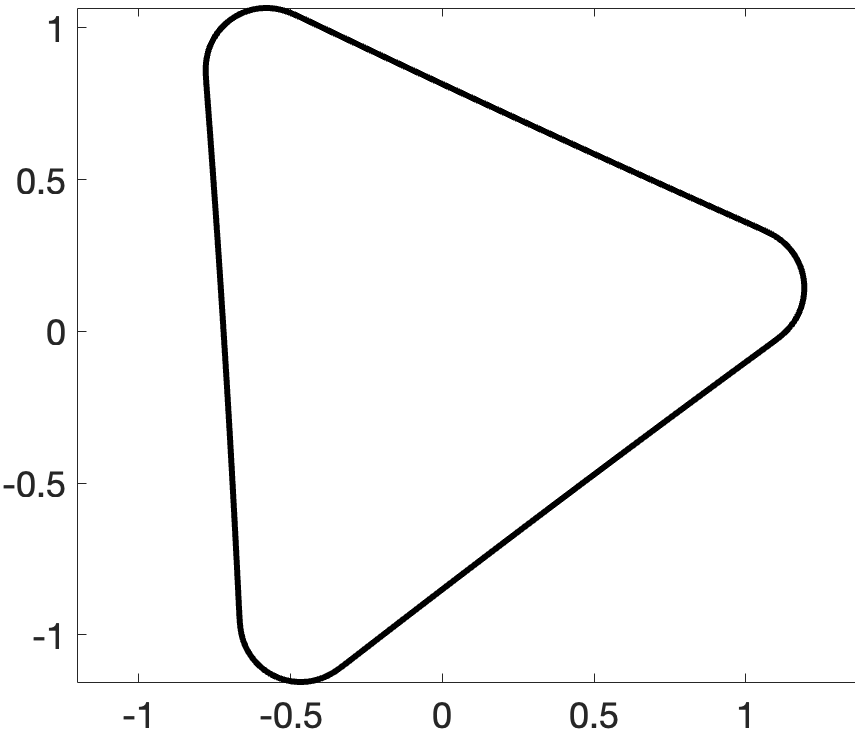}
\\
\includegraphics[width=1.2in]{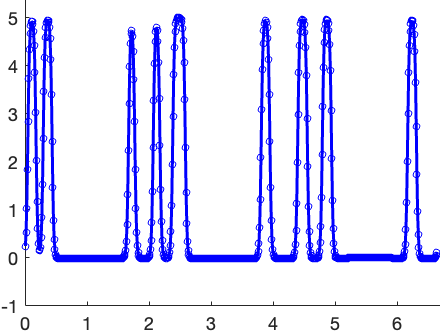}&
\includegraphics[width=1.2in]{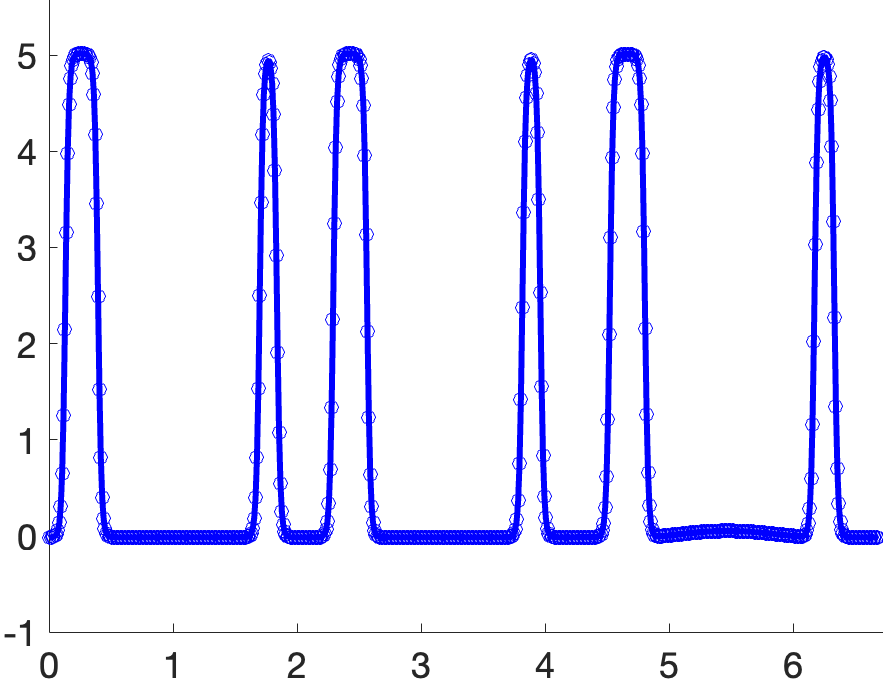}&
\includegraphics[width=1.2in]{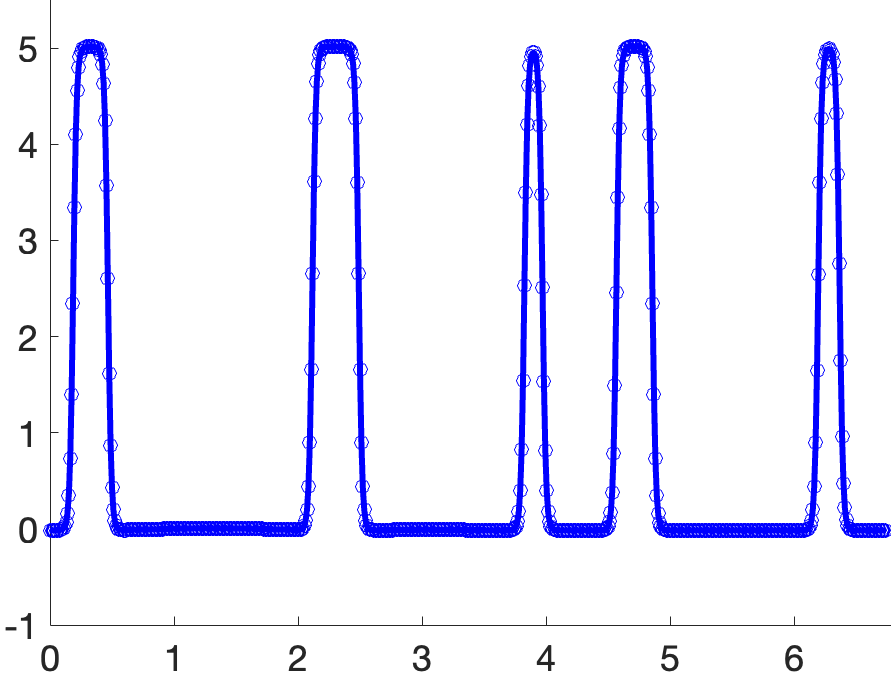}&
\includegraphics[width=1.2in]{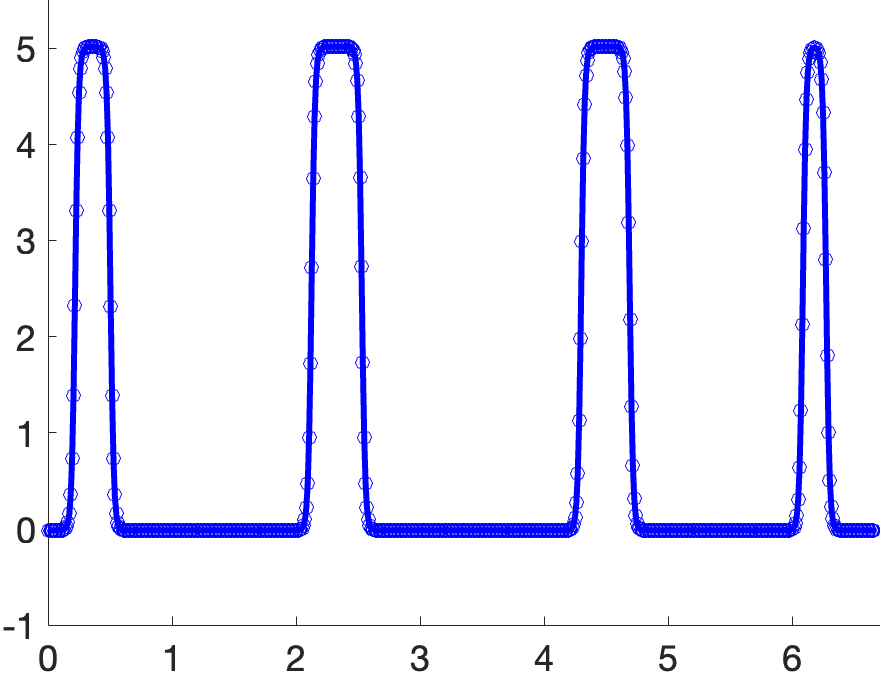} &
\includegraphics[width=1.2in]{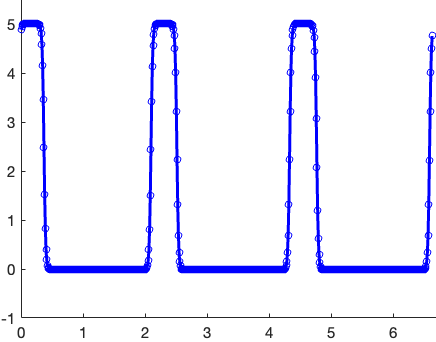}
\end{tabular}
\caption{\small Coarsening of interface under the faceting gradient flow, \eqref{e:Facet} with system parameters $\alpha=0.1, \beta=5, \kappa_*=5$ and $N=600$ grid points. (Top) Images of curve in $\mbbR^2$ and (bottom) plot of the corresponding curvature versus position along $\mbbS$  at times $t=2.2\times10^{-4}, 9.8\times10^{-3}, 3.2\times10^{-2}, 1.35\times10^{-1}$, and $2.0$ respectively.  
The open circle indicates the location $\gamma(0)$ on the curve $\Gamma$.}
\label{f:facet}
\end{figure} 

In the $\alpha\ll1$ limit we anticipate quasi-adiabatic dynamics in the front positions $\vp=(p_1, \ldots, p_{2N})=\vp(t)$.  We write $U_N=(\psi_N(\cdot;\vp\,),g)^t$ and expand
$$U(s;\vp\,)
= U_N +\delta U_1,$$
where  $g$ is a constant. 
The flow \eqref{e:FGF} takes the form
$$ \delta^{-1}\nabla_{\vp}\, \psi_N\cdot \dot{\vp}+
\partial_tU_1  = 
\cR(\cdot;\vp\,) + \cL_{\vp}\, U_1  +\delta\cN(U_1),
$$
where the $\delta$-scaled residual $\cR:=\delta^{-1}\cM\cM^t(\nabla_\mrI\cE)(U_N)=O(1)$ drives the slow dynamics of the front positions. 
The dominant linearity $\cL_{\vp}=\cL(U_N)$ is given in \eqref{e:FOE-Lin}. The exact linearization has $O(\delta)$ corrections that arise from $\psi_N$ being a quasi-equilibrium. These can be incorporated into the residual terms. The nonlinear terms $\cN$ arising from the expansion about $U_N$ are formally lower order. This quasi-adiabatic framework has been successfully applied to rigorously capture the slow front dynamics when the parameters $\vp$  evolve on a slower time scale than that of the relaxation for $U_1,$ \cite{KP_RG1}, \cite{KP_RG2}. This approach requires a spectral dichotomy for the operator $\cL_{\vp}$. The operators $\cL_{\vp}$ are taken to be piece-wise constant in time, subject to updates on time scales that are long compared to the relaxation time of the error term. Arc length can be taken to be spatially and temporally constant on the slow time scales, with spatial reparameterized at each temporal jump point setting the updated arc length to be constant over $\mbbS.$

\begin{figure}
\begin{center}
\includegraphics[width=3.0in]{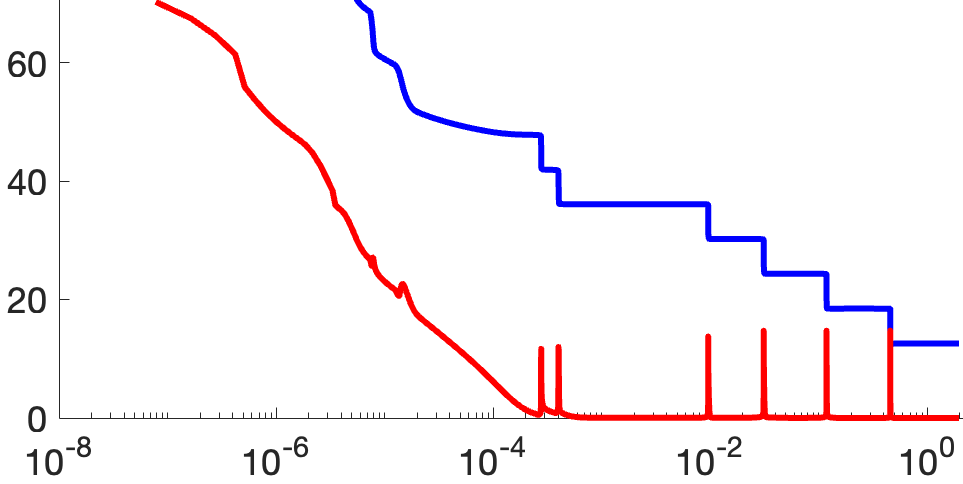}
\end{center}
\caption{\small A semi-log plot of system energy $\cE_{\rm Fc}$ in blue and the functional residual $\cR_{\kappa}$ in red versus computational time for the simulation in Figure\,\ref{f:facet}. }
\label{f:facet_energy}
\end{figure} 

\subsection{Computational Results}
We supplement discussion of the quasi-adiabatic reduction with numerical simulations of the gradient flow \eqref{e:Facet-NV} associated to the faceting energy. A description of the numerical scheme is presented Section\,\ref{s:numer} of the Appendix. Figure\,\ref{f:facet} gives the system parameters and the images $\Gamma=\Gamma(t)$ of the interface and the associated curvature $\kappa(\cdot,t):\mbbS\mapsto\mbbR$ at sequence of five computational times as indicated. The initial spinodal decomposition yields 18 fronts that arrange into 9 pairs of up-down fronts that range from zero curvature to curvature near $\kappa_*.$ These coarsen in several events finally arriving at an equilibrium comprised of six fronts arranged symmetrically to yield an equilateral triangle with $\kappa_*$-rounded corners. Collectively this represents quasi-steady front motion converging to a symmetric equilibrium, illustrating the dynamics suggested by the approach sketched above.  Figure\,\ref{f:facet_energy} presents a semi-log plot of the system energy and functional variation residual $\cR_\kappa$ versus time. The energy decays monotonically, with sharp stair-step drops at the coarsening events. The functional residual converges at $t=2\times 10^{-4}$ to a value of $10^{-3}$, with six subsequent sharp excursions corresponding to the coarsening events that take the 9 pairs of fronts down to 3 pairs. The quasi-adiabatic analysis has the potential to capture the system evolution during the regimes $\cR_{\kappa}\ll 1.$ 



\section{Two-Point Energies}
\label{s:TP}
In many applications interfaces have long range self-interactions. Generically these induce a long-range self-attraction and a short-range self-repulsion.  Models of these are naturally described via a two-point self-interaction kernel that govern interactions between two points of the interface. For a rapidly decaying kernel the two-point interactions are dominated by contributions from near contact points: sections of interface that are well separated in arc length but proximal in physical space. Near contact points are not self-evident from the intrinsic coordinates of the curve. Two-point energies break the intrinsic coordinate formulation of the energy, incorporating a kernel based upon the spatial distance $d(s,\ts)=\gamma(s)-\gamma(\ts)$ between two points on $\Gamma$.

\subsection{Bounded Two-Point Kernels}
We introduce a two-point interaction kernel $\mrA$ and associated energy
\beq
\label{e:AdhE}
\cE_{\mrA}(U) = \int_\mbbS\int_\mbbS \mrA(|d(s,\ts)|^2)\rmd\sigma_s\rmd\sigma_{\ts}
\eeq
where $d=\gamma(s)-\gamma(\ts)$ is the two-point distance vector and the subscript on the surface measure indicates the variable of integration. The two-point interaction kernel $\mrA:\mbbR_+\mapsto\mbbR$.  
Indeed we avoid potentials that are unbounded at $d=0$ partially due to numerical complications associated with infinite self-interaction energy of adjacent parts of the curve but also because self-intersection energies are indeed not infinite in physical reality. 

To take variations of the energy it is convenient to parameterize the variations through the extrinsic vector field $\mbbV.$ Using $\delta$ as a parameterization variable the $\phi(\delta)$
path \eqref{e:path} induces a $\gamma(\cdot;\delta)$ path, whose first variations satisfy
\beq\label{e:Extrinsic-variation}
\begin{aligned}
U_1 &= \cM(U) \mbbV,\\
\gamma_1 &= \mbbV.
\end{aligned}
\eeq
The $\mbbV$-formulation avoids the need to invert $\cM$ to write $\gamma_1$ in terms of $U_1.$ From these relations the variation of the adhesion energy takes the form
\beq \begin{aligned}
\partial_\delta\cE_{\mrA}(\phi)\bigl|_{\delta=0} &= \int_\mbbS\int_\mbbS\left( 2\mrA'(|d|^2) d\cdot (\gamma_1(s)-\gamma_1(\ts)) + \mrA(|d|^2)  \left(\frac{g_1(s)}{g(s)}+\frac{g_1(\ts)}{g(\ts)}\right) \right)\rmd\sigma_s\rmd\sigma_{\ts},\\
&= 
\int_\mbbS\int_\mbbS \Bigl[ 2\mrA'(|d|^2) d\cdot(\mbbV^n(s)n(s)+\mbbV^\tau(s)\tau(s)-\mbbV^n(\ts)n(\ts)-\mbbV^\tau(\ts)\tau(\ts)) + \\
 &\hspace{0.75in}\mrA(|d|^2) \left(\kappa(s) \mbbV^n(s)+\nabla_s\mbbV^\tau(s) - \kappa(\ts)\mbbV^n(\ts)-\nabla_s\mbbV^\tau(\ts) \right)\Bigr] \rmd\sigma_s\rmd\sigma_{\ts}.
\end{aligned}
\eeq

Here $\prime$ acting on $\mrA$ denotes differentiation with respect to its single variable and all terms that are not first variations or extrinsic vector field are evaluated at $U.$ The integrals over $s$ and $\ts$ are interchangeable up to $d\mapsto -d$, so the $\mbbV^n(\ts)n(\ts)$ and $\mbbV^\tau(\ts)\tau(\ts)$  terms can be 
combined with their unmarked siblings. This simplification and an integration by parts  yields the reduced expression
\beq
\label{e:Adh_GF}
\begin{aligned}
\partial_\delta\cE_{\mrA} &= 2\int_{\mbbS^2}\left( 2\mrA'(|d|^2) d\cdot(\mbbV^n(s)n(s)+\mbbV^\tau(s)\tau(s)) +\mrA(|d|^2)(\kappa(s) \mbbV^n(s)+\nabla_s\mbbV^\tau(s))\right)\rmd\sigma_s\rmd\sigma_{\ts},\\
&= 2\int_{\mbbS^2} \left(2\mrA'(|d|^2)d\cdot n(s)+A(|d|^2)\kappa(s)\right)\mbbV^n(s) + 
    2\mrA'(|d|^2)d\cdot (\tau(s)- \nabla_s\gamma(s))\mbbV^\tau(s)\,\rmd\sigma_s\rmd\sigma_{\ts},\\
&= 2\int_{\mbbS^2} \left(2\mrA'(|d(s,\ts)|^2)d\cdot n(s)+A(|d|^2)\kappa(s)\right)\mbbV^n(s) \,\rmd\sigma_s\rmd\sigma_{\ts},
\end{aligned}
\eeq
where the $\mbbV^\tau$ prefactor drops out since $\nabla_s\gamma=\tau.$
Introducing the two-point force
$$ \mbbA(s):= 4 \int_\mbbS  \mrA'(|d|^2) d(s,\ts) \rmd\sigma_{\ts},$$
and the scalar surface energy density
$$ \mrB(s):= 2\int_\mbbS \mrA(|d(s,\ts)|^2) \rmd\sigma_{\ts},$$
the two-point energy extrinsic gradient satisfies 
\beq
\langle\nabla_\mrE \cE_\mrA,\mbbV\rangle_{L^2(\mbbS)} = \int_\mbbS\bpm \mbbA\cdot n+\kappa \mrB\\ 0 \epm \cdot\mbbV\,  \rmd \sigma.
\eeq
This motivates the definition
\beq\label{e:Adh-Ex_grad}\nabla_\mrE\cE_{\mrA}= \bpm \mbbA\cdot n+\kappa \mrB\\ 0 \epm,
\eeq
and the associated two-point $L^2$-gradient normal velocity
\beq\label{e:Adh-NV}
\mbbV^n_{\mrA}= -\left(\mbbA\cdot n+\kappa \mrB\right).
\eeq
The energy independence under rigid body motion implies that $\nabla_\mrE\cE\in\ker(\cM)^\bot$ and  $\nabla_\mrE\cE$ resides in the range of $\cM^\dag$. Surprisingly the rigid body energy independence leads to non-trivial relations, see Lemma\,\ref{l:TP_invar} in Section\,\ref{s:TP_invar} of the Appendix.
This orthogonality allows us to define the associated intrinsic gradient through the inverse of $\cM,$
\beq \label{e:Adh-In_grad}
\nabla_\mrI \cE_{\mrA} = \left(\cM^\dag\right)^{-1} \bpm \mbbA\cdot n+\kappa \mrB\\ 0 \epm.
\eeq
This leads to the two-point $L^2(\mbbS)$-gradient flow
\beq\label{e:Adh-GV}
U_t = -\cM\cM^\dag \nabla_{\mrI}\cE_\mrA=-\cM\nabla_{\mrE}\cE_\mrA=-\cM
\bpm \mbbA\cdot n+\kappa \mrB\\ 0 \epm.
\eeq
At an equilibrium the linearization of the flow takes the form
$$\begin{aligned}
 \cL U_1=-\cM\left[\nabla^2_{\mrE}\cE_{\mrA}\right] \mbbV 
 = -\cM
\left[\nabla_\mrE^2\cE_\mrA \right]\cM^{-1} U_1,
\end{aligned}$$
which is a similarity transformation of the second extrinsic variation of the energy. Unlike the intrinsic 2nd variation, $\nabla_\mrI^2\cE$, the extrinsic 2nd variation is not generically self adjoint, and neither is $\cL.$ Indeed it can also be expressed in the product formulation analogous to \eqref{e:FGF}.
The second variation of the two-point energy is complicated by the mixture of local and nonlocal operators and by the role of the $\cM^{-1}$. Simplifications can be obtained by scaling the two-point kernel to have an effective support that is small in comparison to the inverse of the $L^\infty$ norm of the curvatures of the curve. In this case the integrals over $\mbbS$ can be approximated by local contributions from near-intersection points. We pursue this scaling in the context of an adhesion-repulsion form of the two-point energy. 
\begin{figure}
\begin{tabular}{ccc}
\includegraphics[height=1.6in]{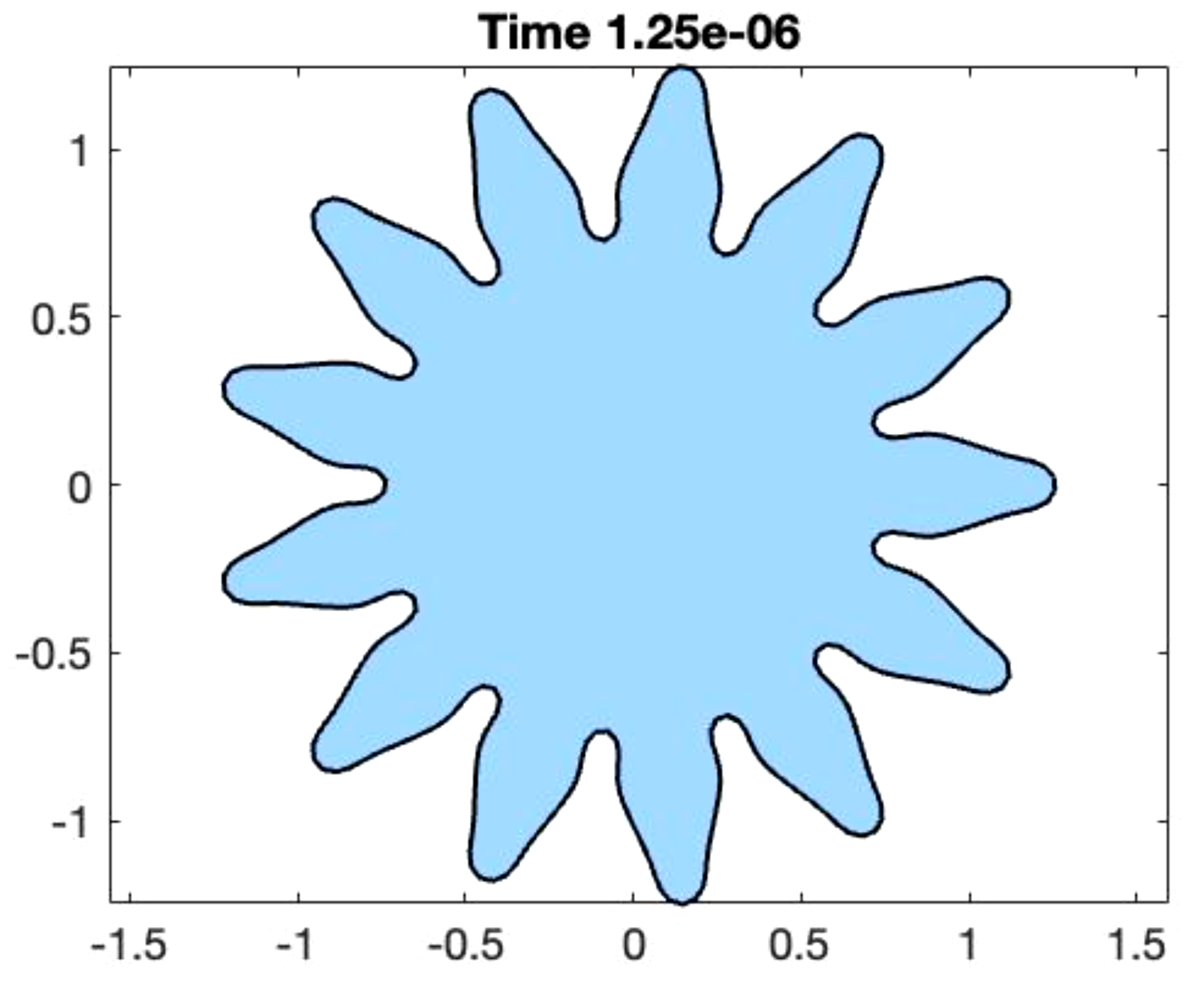} &
\includegraphics[height=1.6in]{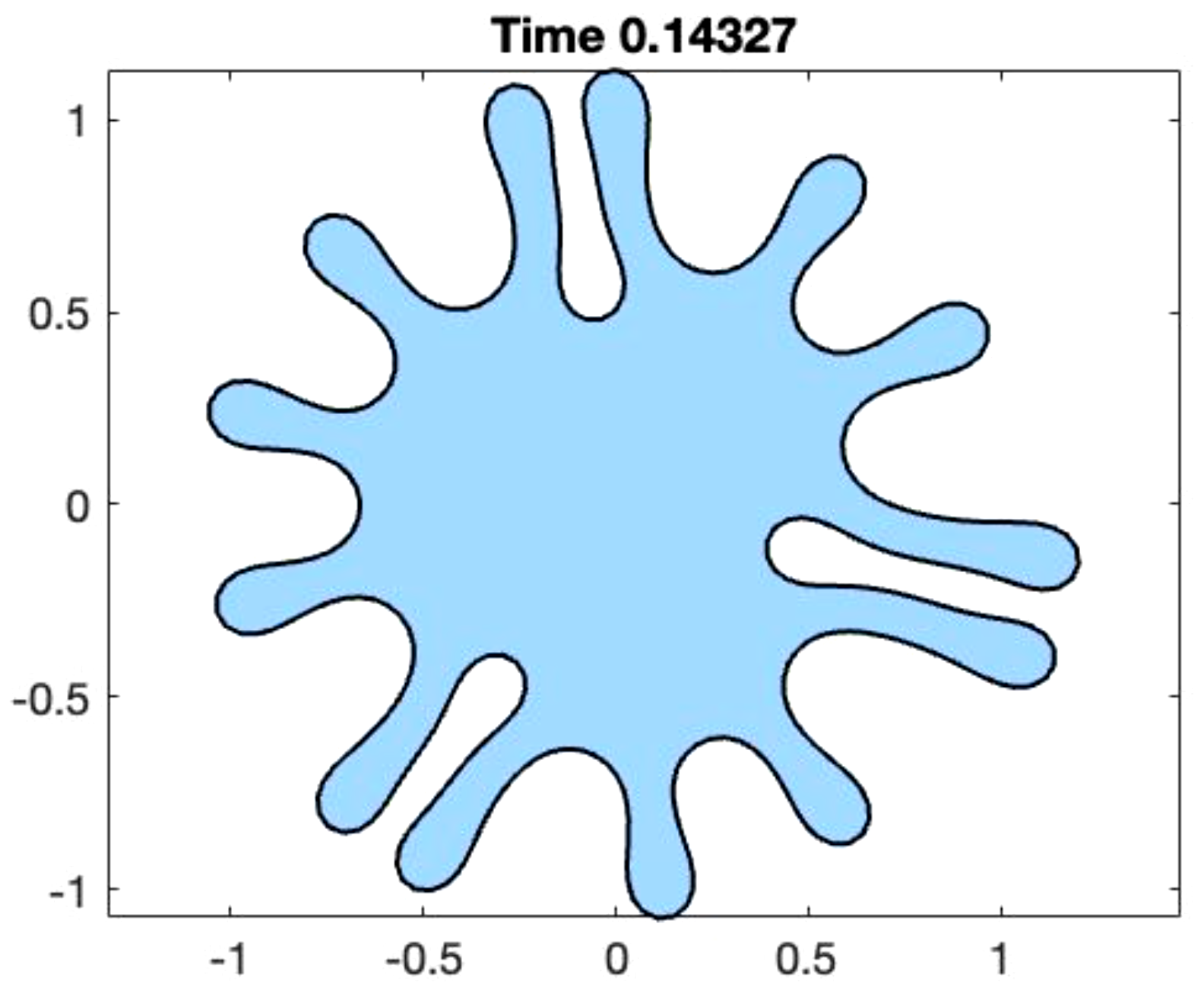} &
\includegraphics[height=1.6in]{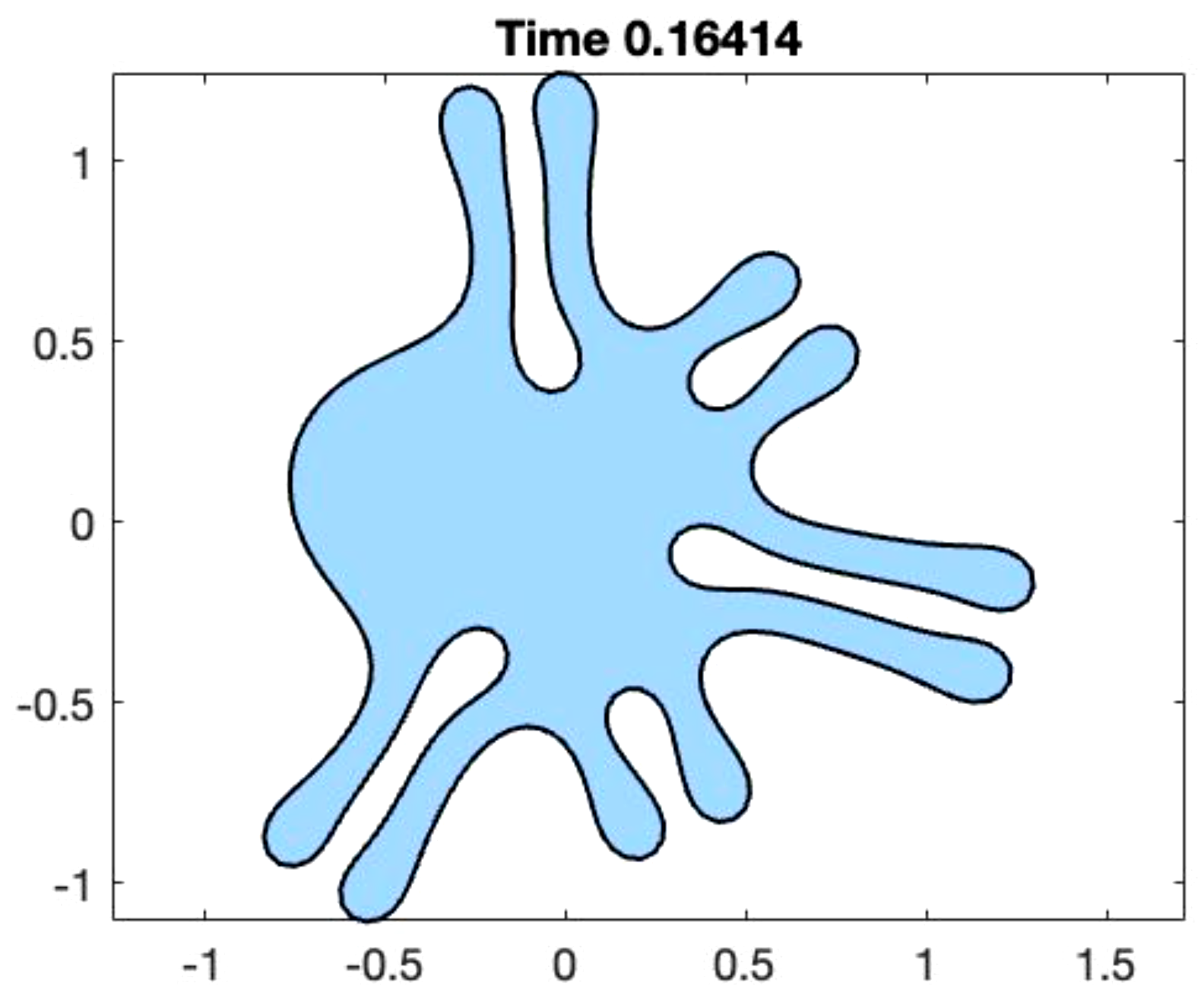} \\
\includegraphics[height=1.6in]{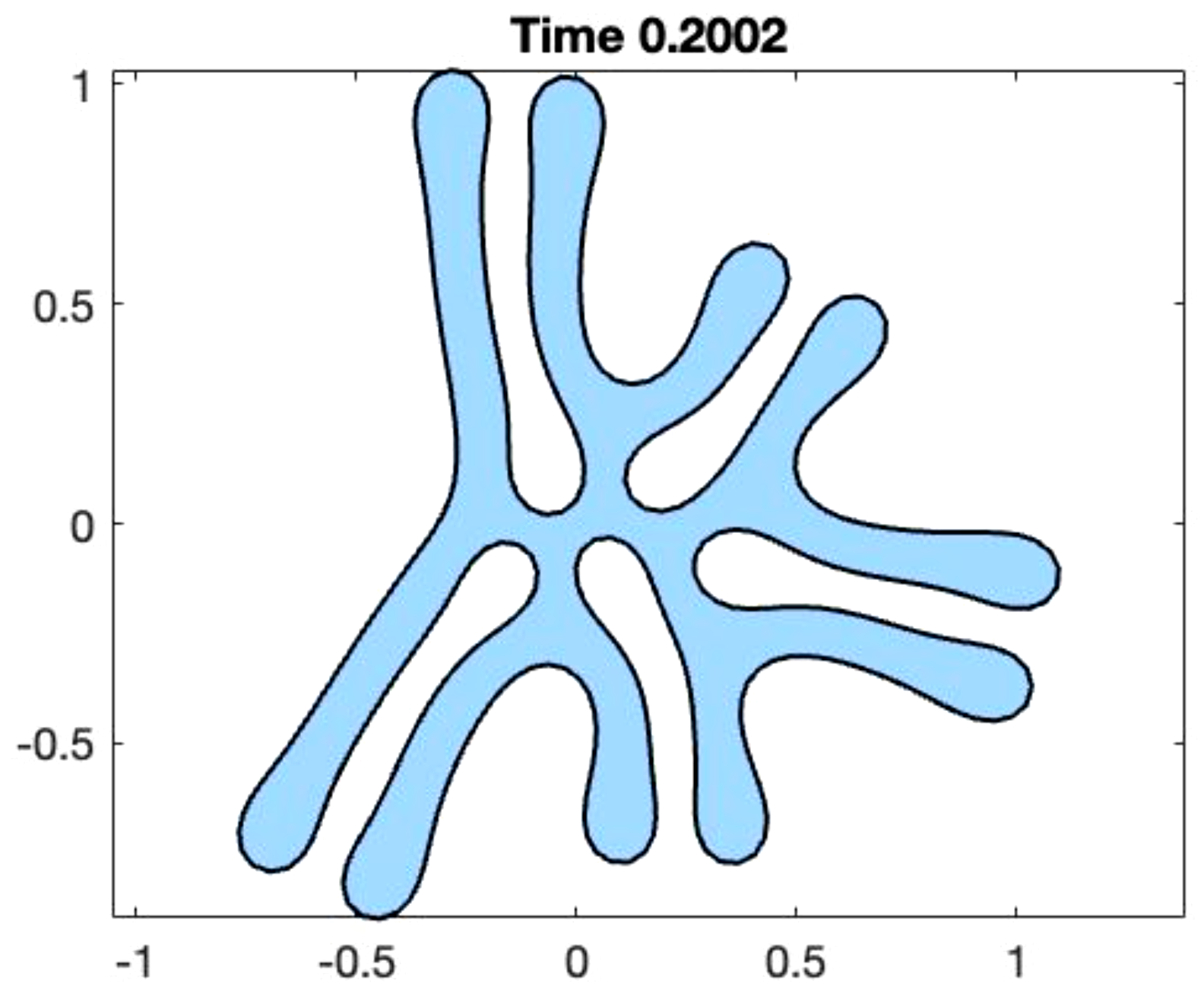} &
\includegraphics[height=1.6in]{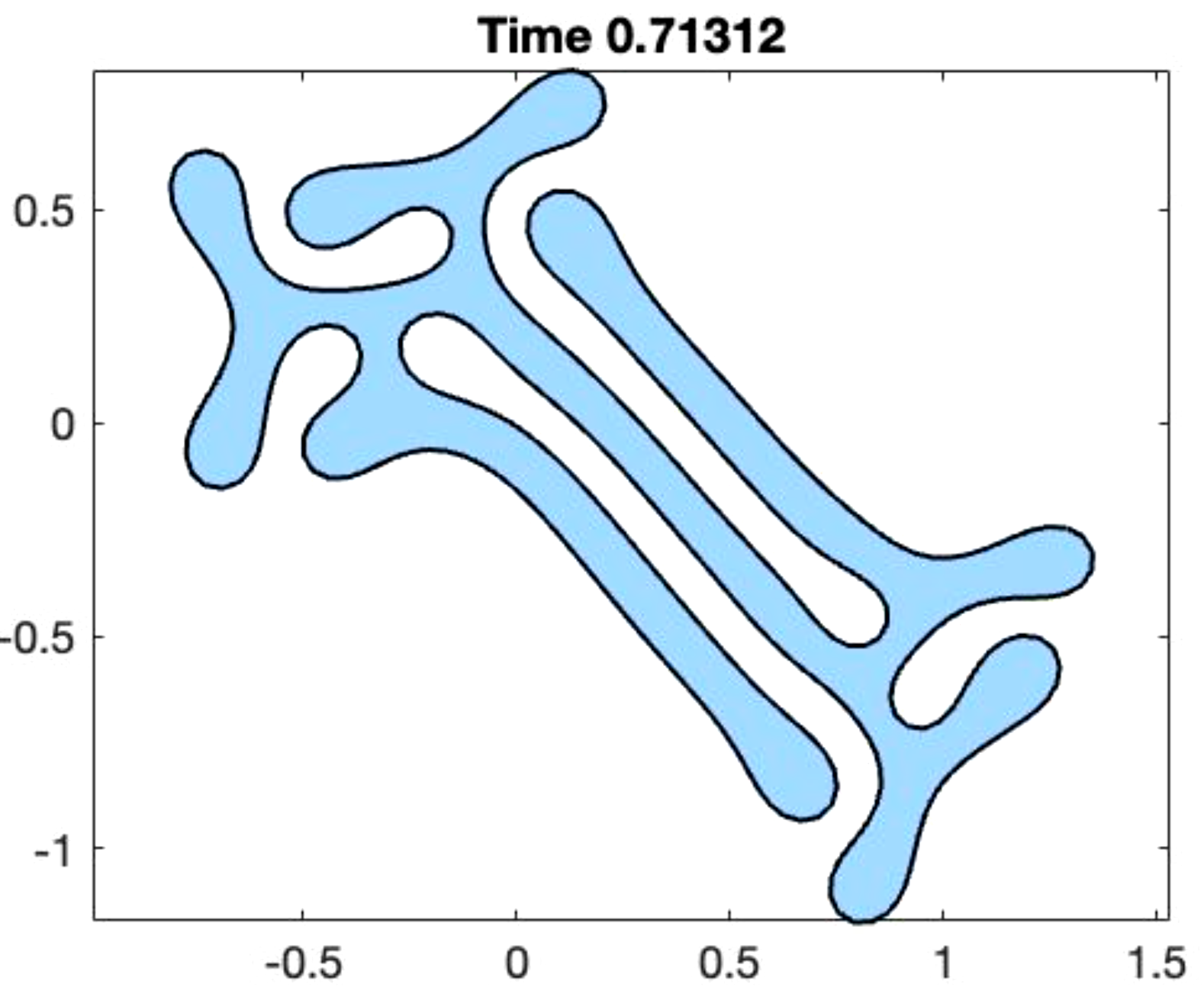} &
\includegraphics[height=1.6in]{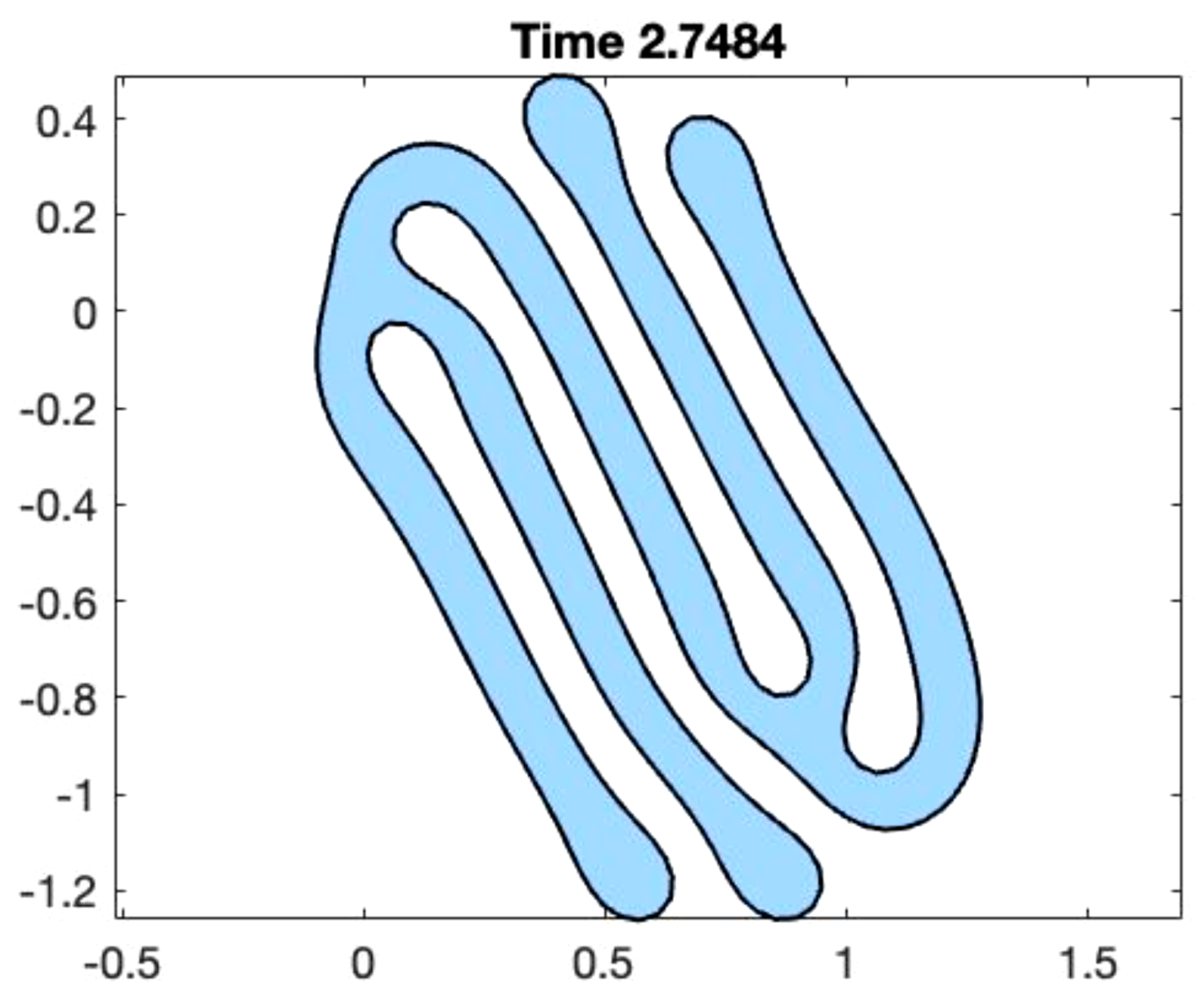} \\
\end{tabular}
\vskip 0.1in
\includegraphics[width=3.0in]{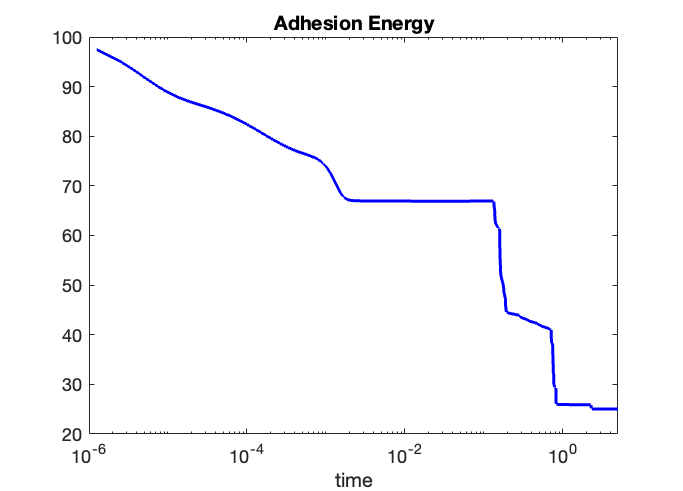} 
\caption{\small (Top two rows) $L^2(\mbbS)$- gradient flow for the energy \eqref{e:CHA-Energy} from crenelated initial data (top-left). Parameters are given in Table\,\ref{T:T1} with $\rho=1.3$ and time of simulation is indicated above image. Interiors are shaded for clarity. (bottom row) Semi-log plot of system energy versus time showing monotonic decay to equilibrium.
}
\label{f:Snowflake}
\end{figure}

\subsection{Application: Canham-Helfrich Adhesion-Repulsion Energy}
\label{s:NSI}
A two-point adhesion repulsion energy arises from a bounded interaction kernel $\mrA$ with a Lenard-Jones type structure. More specifically $\mrA$ has a single, negative minima at a positive separation distance, a positive global maxima at zero, and tends to zero with growing separation distance. The interaction kernel drives the interface to pair disjoint sections of itself at an optimal separation distance while penalizing but not precluding self intersection. We investigate the properties of a two-point adhesion-repulsion kernel by pairing it with a weak Canham-Helfrich curvature term and a strong perimeter penalty,
\beq\label{e:CHA-Energy}
\cE_{\rm CHA}(U)=\cE_{\mrA}(U)+\frac{\beta}{2}(|\Gamma|-\rho|\Gamma_0|)^2 + \int_\mbbS \frac{\eps}{2} \kappa^2\,\rmd\sigma.
\eeq
Here $\Gamma_0$ denotes the image of the initial curve and the ``precompression-factor'' $\rho>1$ is the ratio of the initial curve length to the equilibrium curve length. For $\beta\gg1$ and $\rho=1$ this is a penalty term that models fixed densities common in biological membranes. The case $\rho>1$ corresponds to precompression induces a short transient interface lengthening and concomitant buckling motion that seeds bifurcations in subsequent long-term behavior.
We take the adhesion-repulsion interaction kernel in the functional form
\beq\label{e:LJ}
\mrA(\ell)=  A_0 \left(1- \frac{\ell}{a^2\ell_*^2}\right)\exp(- \ell/\ell_*^2). \eeq
 The potential takes is maximum value $A_0>0$ at $\ell=0$, is negative for $\ell>(a\ell_*)^2$, has a single minima at $\ell_{\rm min}=\ell_*^2(1+a^2)$ with a negative value,
 $$ A(\ell_{\rm min})=-A_0 a^{-2}e^{-1-a^2},$$
 which is independent of $\ell_*$.  This Lennard-Jones type form enforces exponential decay of the interaction as is typical of electrostatic screening from a polar solvent.
 
The normal velocity takes the form
\beq
\label{e:CHA-GF}
\mbbV^n_{\rm CHA} =\left(\eps\left(\Delta_s+\frac{\kappa^2}{2}\right)-\mrB(s)-\sigma(|\Gamma|-\rho|\Gamma_0|)\right)\kappa -\mbbA(s)\cdot n(s).
\eeq

\subsection{Formal Near-self-intersection Analysis}
 The short screening length regime $\ell_*\ll 1$ is  physically interesting for electrochemical settings.   The rapid decay of interaction energy with distance simplifies the analysis. We present a heuristic characterization of the energy in a generic near self-interaction event. Assuming that absolute value of curvature is bounded by $\kappa_M$, then the arguments of a pair of close-approach points $\{\gamma(s),\gamma(\ts)\}$ with $|d(s,\ts)|\ll1$ must satisfy the ``turning radius'' condition  $|s-\ts|>\pi/\kappa_M$. If further, the point $(s,\ts)\in\mbbS^2$ is a local minimum of the distance function then we have $n(s)\cdot n(\ts)=-1$ and $d(s,\ts)\cdot \tau(s)=d(s,\ts)\cdot\tau(\ts)=0.$  In the regime $\ell_*\ll1$ and $\kappa_M=O(1)$ the interaction kernel decays more quickly than the curvature can bend the curve, the interaction energy generated by a close-approach point $(s_*,\ts_*)$ is approximated to leading order by integrals over parallel lines a distance $\ell_0$ apart, both parameterized with constant arc length. The dominant contribution arises from $\mbbA$.  We fix $s_*=0$, for $\ts$ on the same line as $s_*$ the contribution to $\mbbA\cdot n=0$ since $d(s_*,\ts)\cdot n=0$ there. For the integral in $\ts$ over the opposite line we have
$$ \mrA'(\ell) = A_0\left(\frac{\ell}{a^2\ell_*^4}-\frac{1+ a^2}{a^2\ell_*^2}\right)\exp(-\ell/\ell_*^2).$$
 The distance $|d(0,\ts)|^2=\ell_0^2+(\ts-\ts_*)^2,$ while $d(0,\ts)\cdot n(0)=  -\ell_0.$  We have the expression
$$\begin{aligned}
 \mrA'(|d|^2)d\cdot n& =
 - \ell_0 A_0 e^{- \ell_0^2/\ell_*^2} \left(\frac{\ell_0^2+\ts^2}{a^2\ell_*^4}-\frac{1+a^2}{a^2\ell_*^2}\right)\exp(-\ts^2/\ell_*^2).
 \end{aligned}
 $$ 
 We assume $\ell_0\sim\ell_*\ll1$ and introduce the scaled distance  $\td_0=\ell_0/\ell_*=O(1).$ 
 The contribution from the opposing line yields 
$$ 
\mbbA\cdot n = \int_0^\infty \mrA'(|d|^2)d\cdot n\, \rmd \ts=  -\frac{\td_0 A_0 }{a^2} e^{-\td_0^2} \left(c_0 (\td_0^2-(a^2+1)) +c_1\right),
$$
where we have introduced
$$ c_p:=\int_0^\infty s^{2p} e^{-s^2}\,\rmd s.$$
Since $c_0>c_1$ the normal force $\mbbA\cdot n$ imposes a stable equilibrium distance
$$\td_*=\sqrt{a^2+1-c_1/c_0},$$ 
between the near contact pairs. 
For the $\mrB$ term similar arguments show that the dominant contribution arises from the self-interactions of $\gamma(0)$ with the line containing $s=0$. On this line the dominant contribution arises at $s=0,$
$$ \mrB(0)=\frac{A_0\ell_*}{a^2} \left(a^2 c_0-c_1 + e^{-\td^2}\left((a^2-\td^2)c_0-c_1\right)\right)=O(\ell_*).$$
This term is lower order, scaling with $\ell$, while $\mbbA$ converges to a non-zero limit for $\ell_*\ll1.$

Returning to the Canham-Helfrich Adhesion-Repulsion energy, 
for a curve with $|\Gamma|=\rho |\Gamma_0|$ and curvature bounded in $H^2(\mbbS)$, the normal velocity at a near-approach point $(s,\ts)\in\mbbS^2$ with $|d(s,\ts)|=\td_0\ell_*$ satisfies
$$\mbbV^n_{\!\rm CHA}(s)=\mbbV^n_{\!\rm CHA}(\ts)=-m_0 A_0(\td_0^2-\td_*^2) + O(\eps\|\kappa\|_{H^2(\mbbS)},\ell_*), $$
where $m_0>0$ is independent of $\ell_*$ and $A_0.$ For sufficiently large value of $A_0$ and small value of $\ell_*$ the normal velocity will seek to stabilize the scaled optimal separation distance $\td_*^2.$

\begin{table}
\begin{tabular}{|c|c|c|c|}
\hline
$\cE_{\mrA}$ & $\ell_*=0.01$ & $a=1$& $\mrA_0=40$\\
\hline
$\cE_{|\Gamma|}$ \& $\cE_\lie$ & $\beta=60$ & $\rho\in\{1.2, 1.3, 1.4\}$ & $\eps=0.08$ \\
\hline
\end{tabular}
\vskip 0.05in
\caption{Parameters for Canham-Helfrich adhesion-repulsion energy \eqref{e:CHA-Energy} used in simulations depicted in Figure\,\ref{f:Snowflake} and \ref{f:fold}.}
\label{T:T1}
\end{table}

\subsection{Simulation of Canham-Helfrich Adhesion-Repulsion Gradient Flow}
We supplement the formal arguments of the Section\,\ref{s:NSI} with numerical simulations of the gradient flow. Figure\,\ref{f:Snowflake} shows the time evolution of the Canham-Helfrich Adhesion-Repulsion gradient flow, \eqref{e:CHA-GF} starting from a crenulated initial interface as depicted in the upper-left panel. Parameter values are given in Table\,\ref{T:T1}. 
The simulation encounters four distinct regimes. In the first regime, $t\in[0,1.5\times10^{-3}],$ the flow preserves the 13-fold symmetry of the initial data as the interface buckles to achieve its equilibrium length and energy decays at an exponential rate. In the second regime, $t\in[1.5\times10^{-3},0.2]$ symmetry is broken by finger pairing while energy is relatively constant. Interfaces typical of this regime are given by (top-middle, top-right) simulations at $t=0.1432$ and $t=0.16414.$ In the third regime, $t\in[0.2,0.7]$ the evolution is dominated  by labyrinthine patterns depicted in images (middle-left, middle-center) at times $t=0.2$ and $t=0.713$. On this time scale the interfaces rearrange through a sequence of tight packings that lead to increasingly lower energy foldings. The final regime, $t\in[0.7,2.5]$ shows slow relaxation to a locally optimal packing. This equilibrium is a layered state evocative of the Thylakoid membranes shown in Figure\,\ref{f:Thylakoid}. At equilibrium the interface recovers a $\pi$-symmetry not present in intermediate states.

\begin{figure}[h!]
\begin{tabular}{ccc}
\includegraphics[height=1.6in]{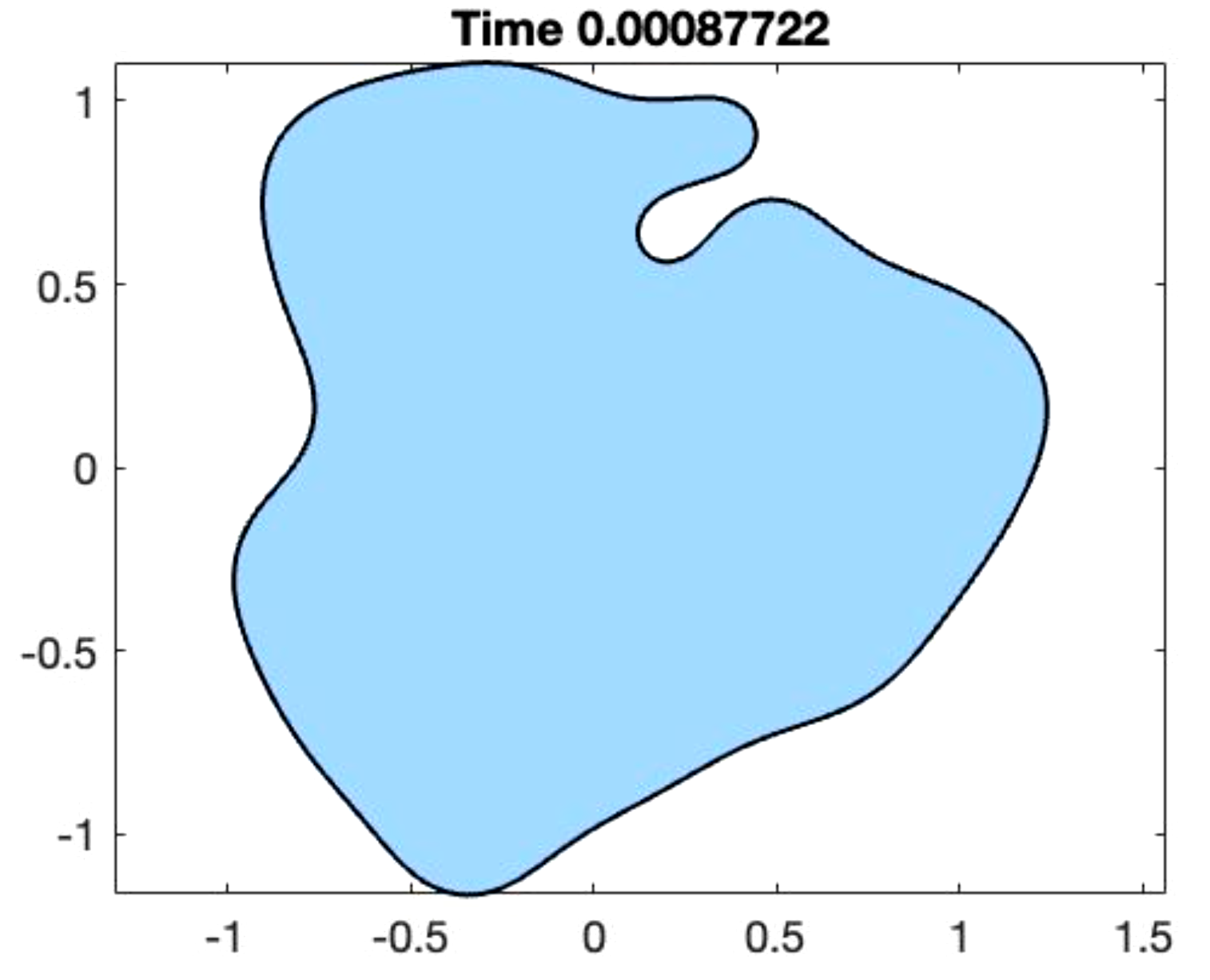} &
\includegraphics[height=1.6in]{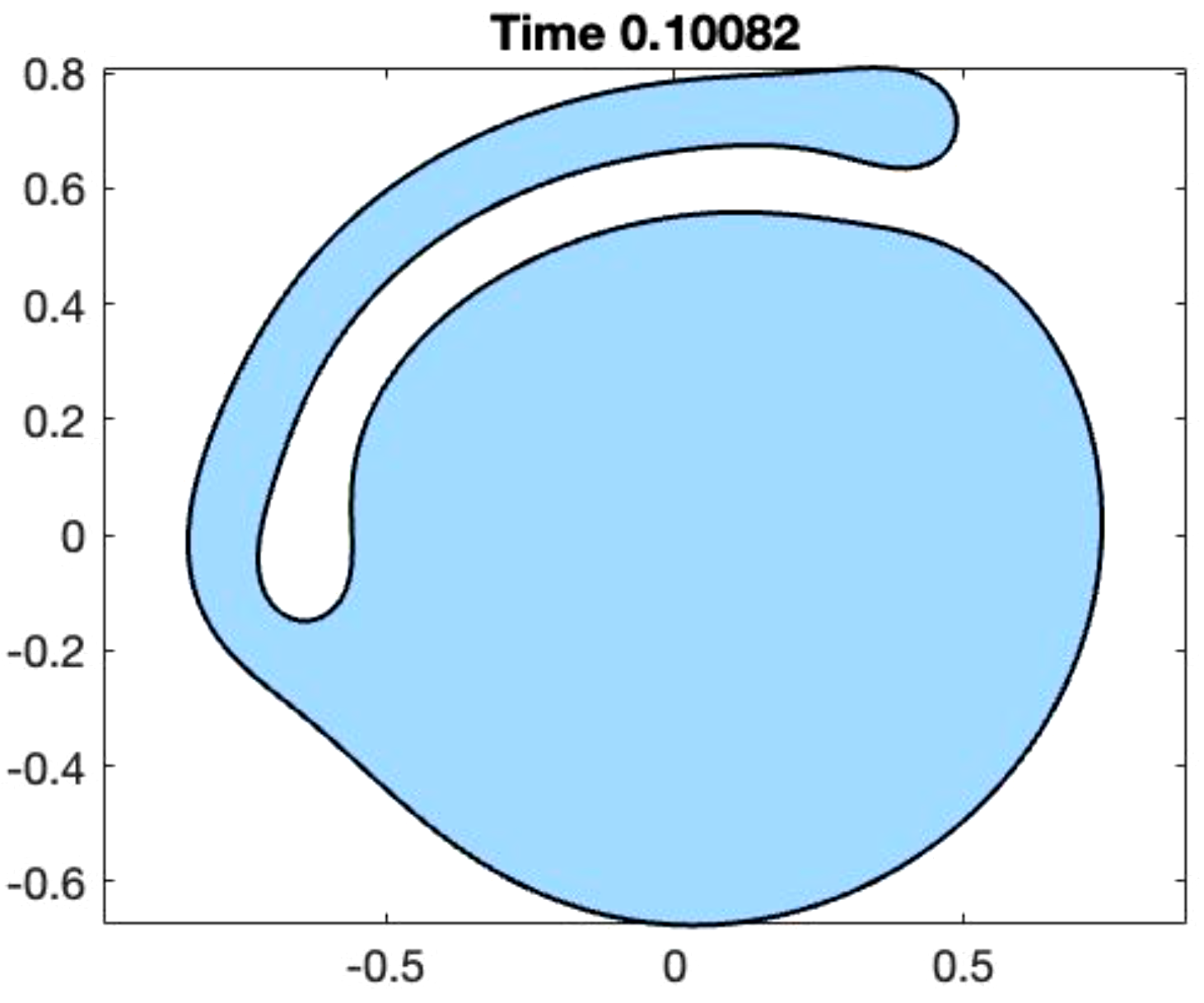} &
\includegraphics[height=1.6in]{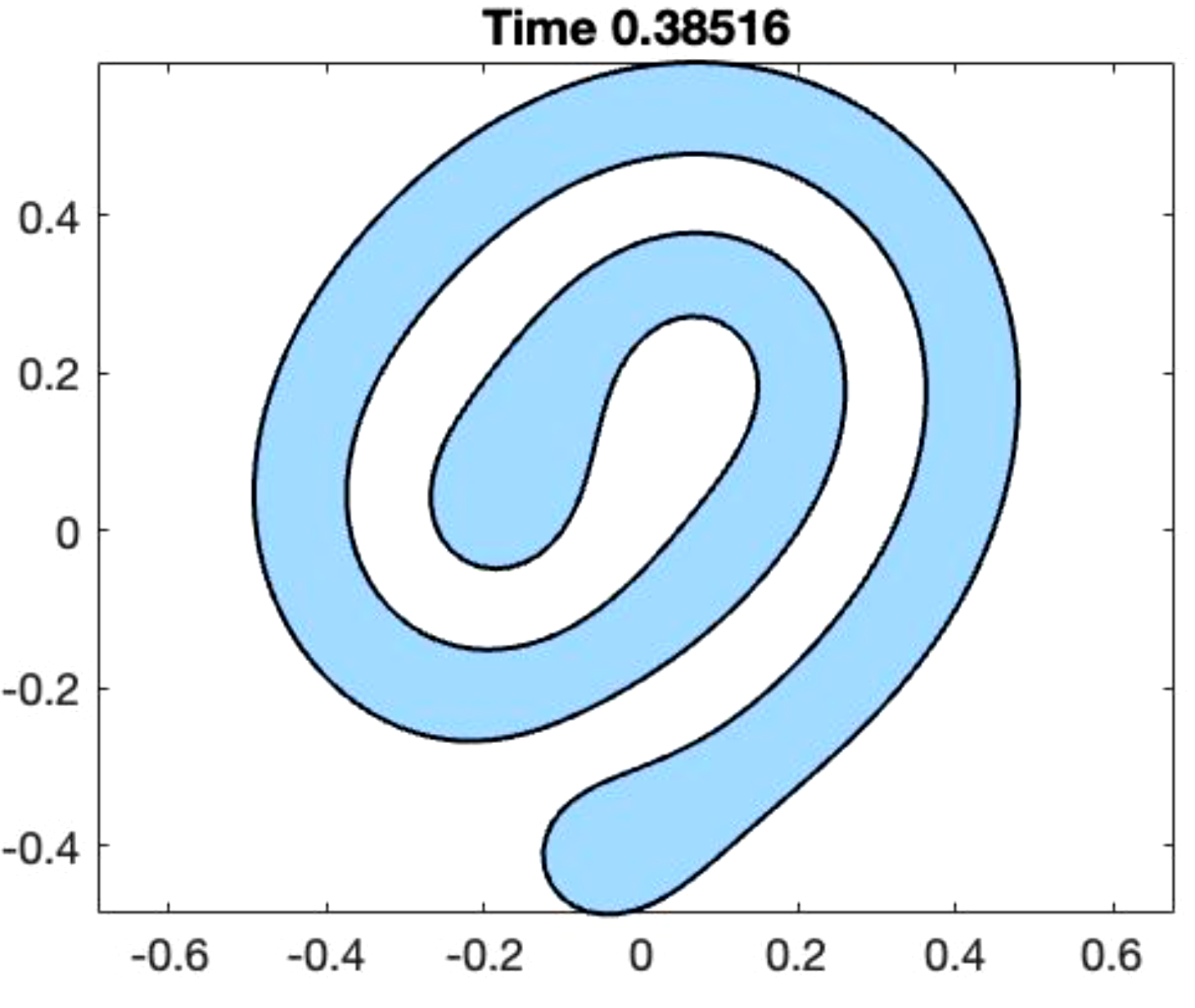} \\
\includegraphics[height=1.6in]{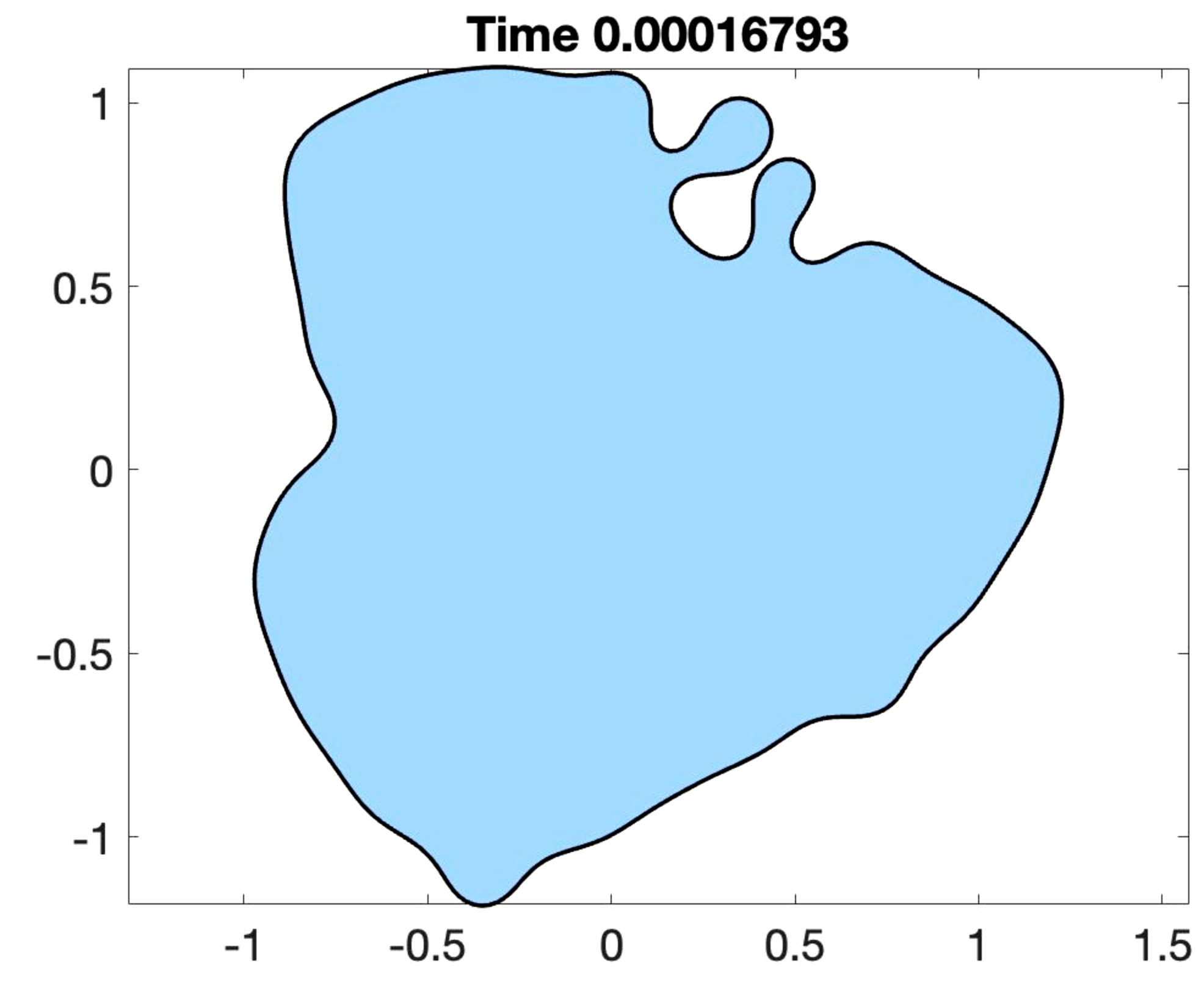} &
\includegraphics[height=1.6in]{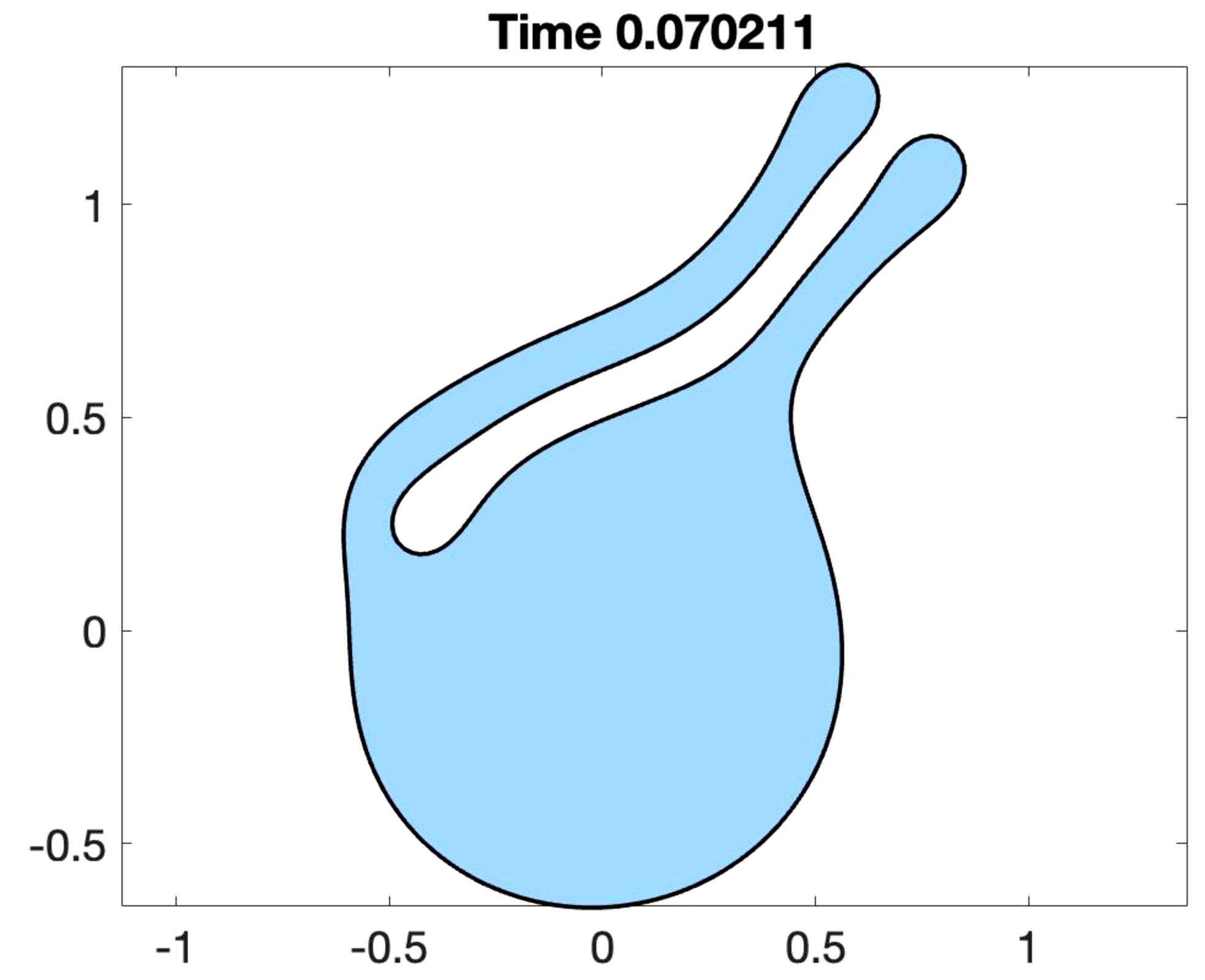} &
\includegraphics[height=1.6in]{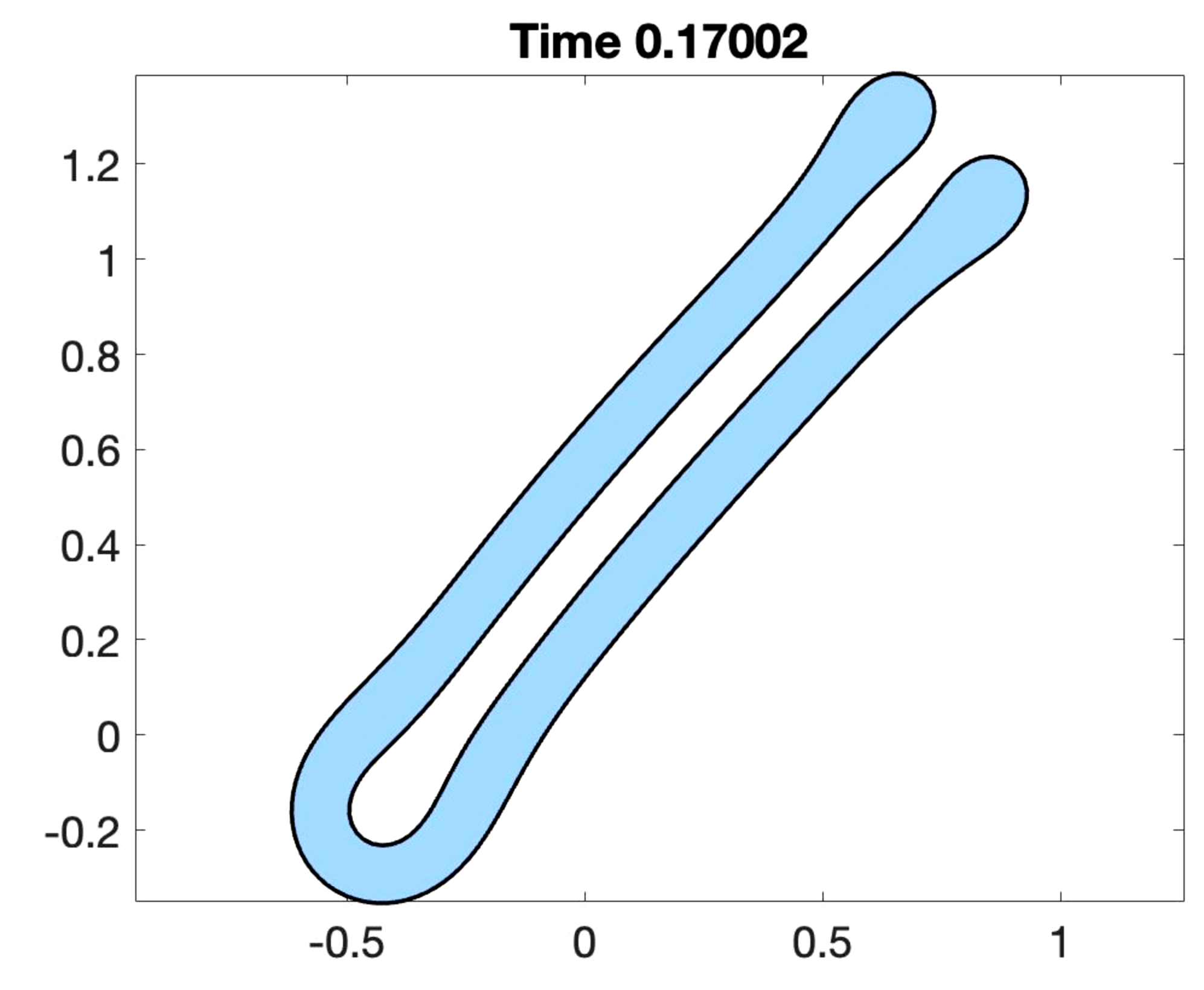} \\
\includegraphics[height=1.6in]{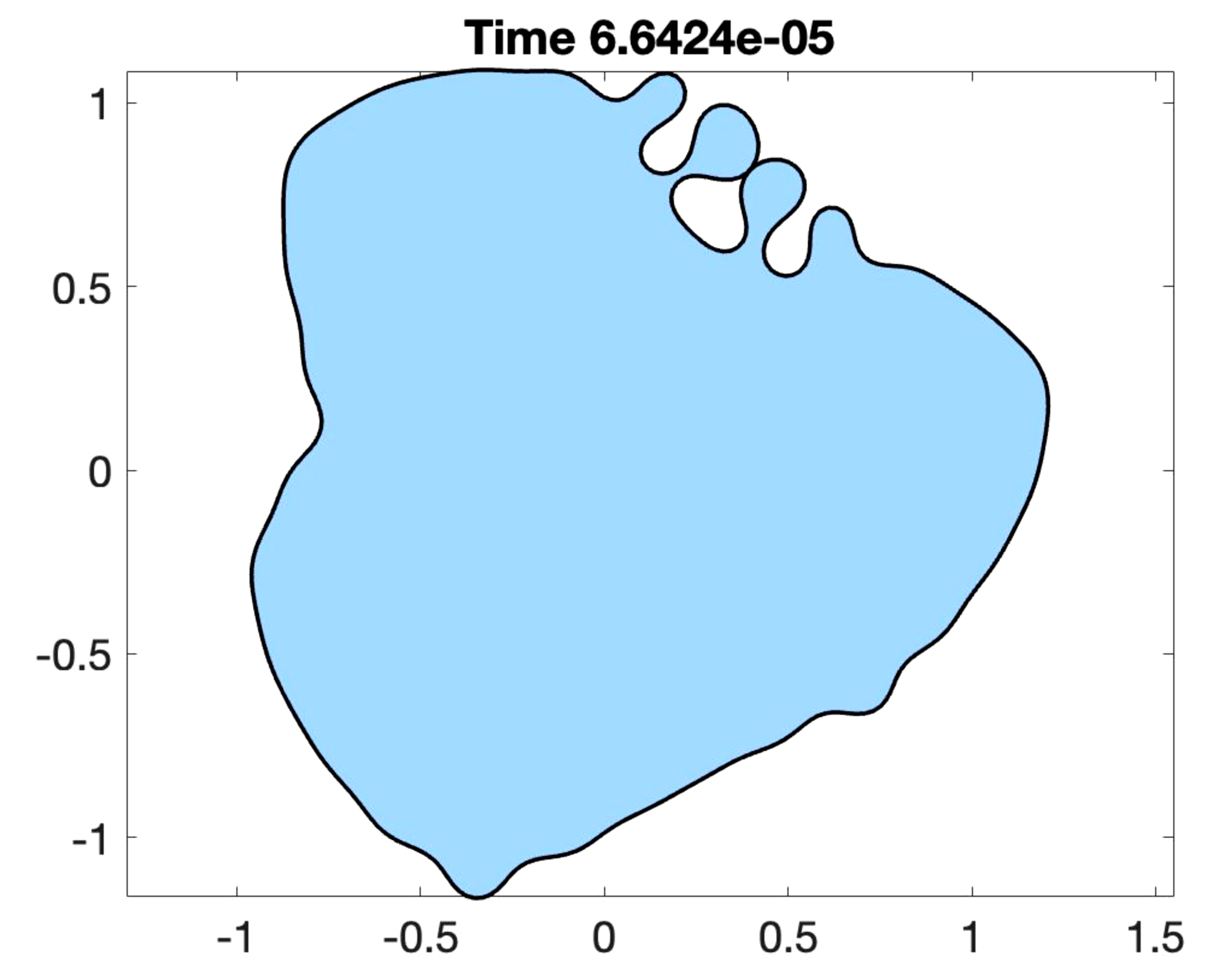} &
\includegraphics[height=1.6in]{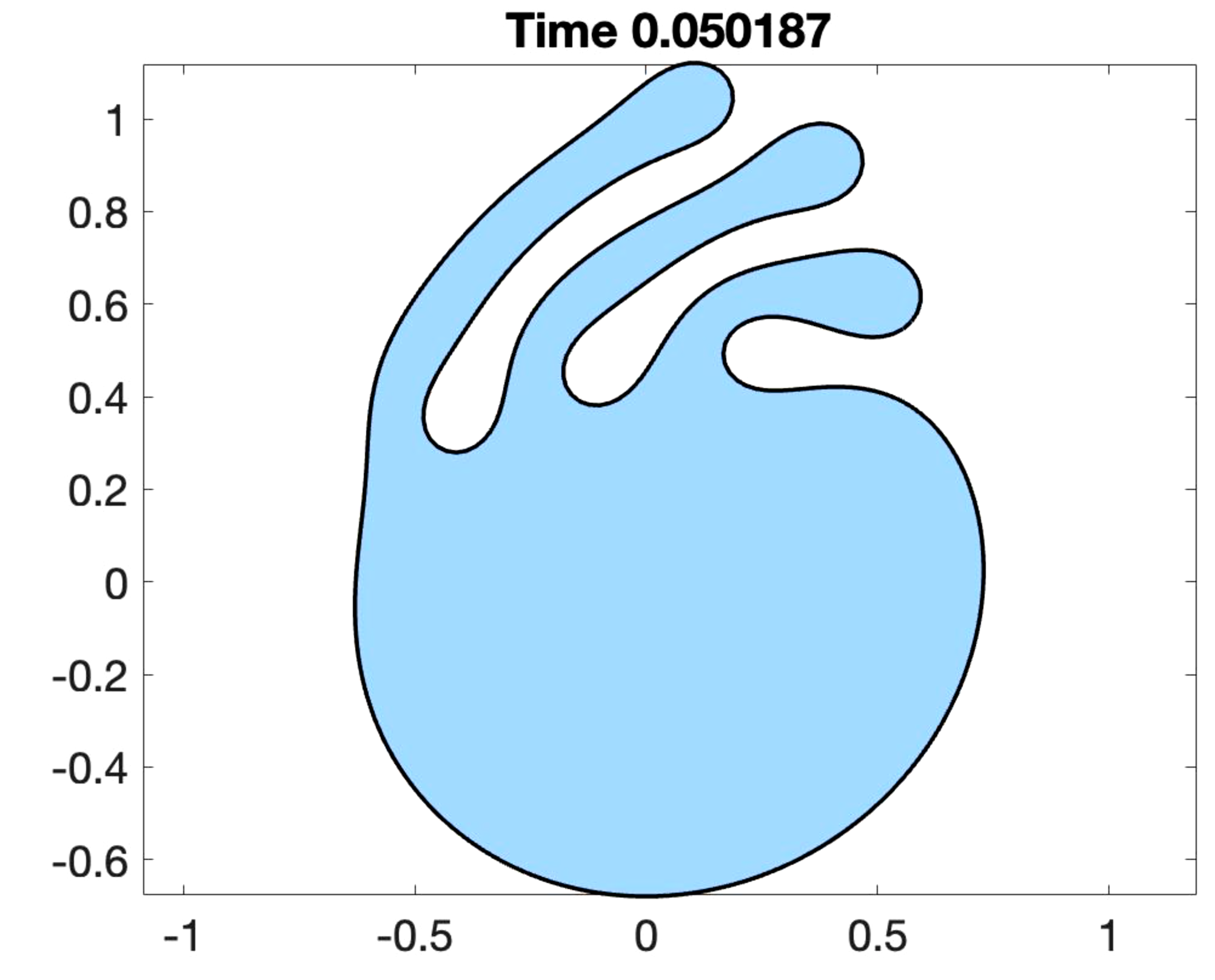} &
\includegraphics[height=1.6in]{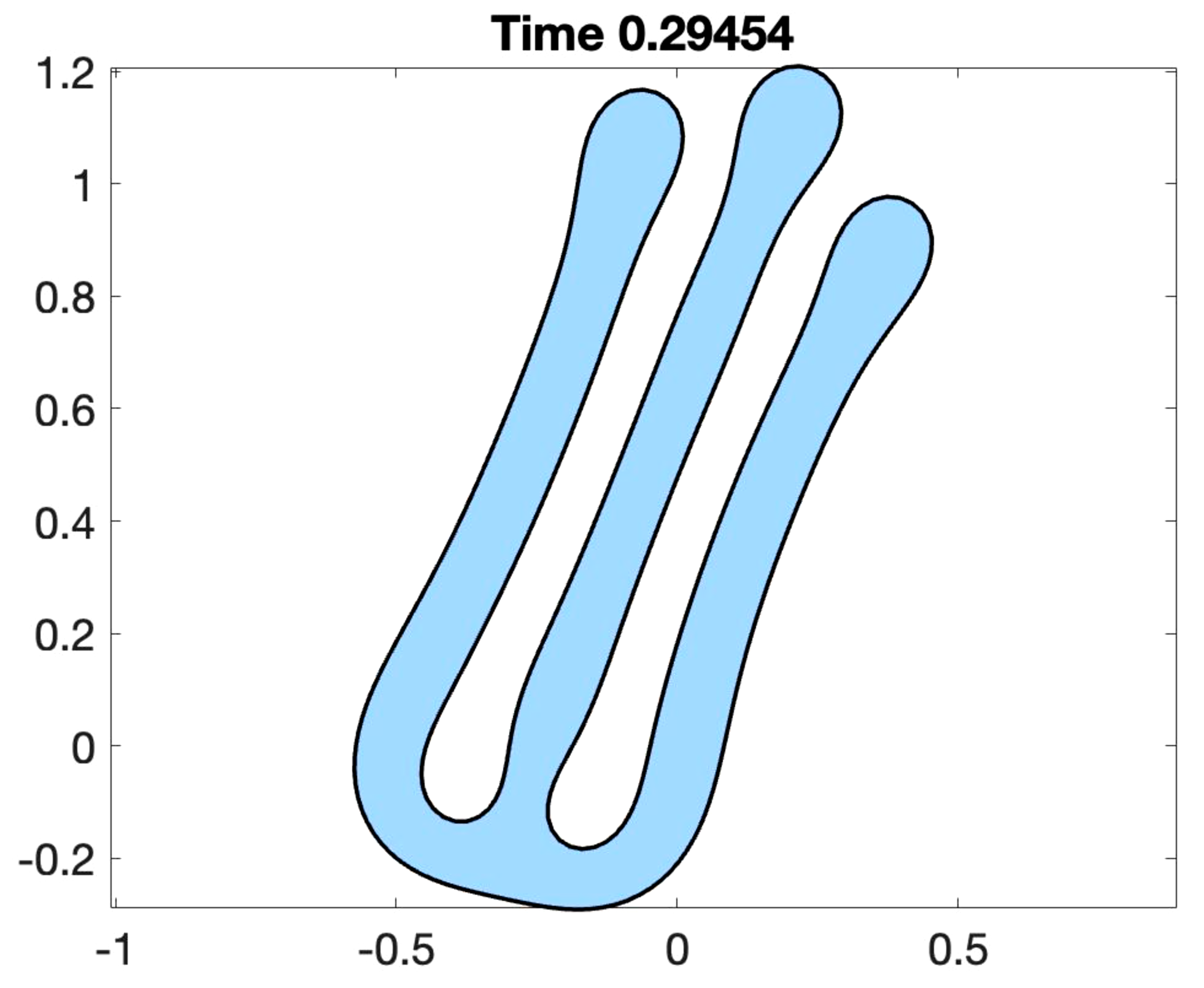} 
\end{tabular}
\put(-395,125){\large\textbf{$\rho=1.2$}} 
\put(-310,125){\large\textbf{$\rightarrow$}}
\put(-155,125){\large\textbf{$\rightarrow$}}
\put(-395, 8) {\large\textbf{$\rho=1.3$}} 
\put(-310,5){\large\textbf{$\rightarrow$}} 
\put(-155,5){\large\textbf{$\rightarrow$}}
\put(-395, -115) {\large\textbf{$\rho=1.4$}} 
\put(-310,-118){\large\textbf{$\rightarrow$}}
\put(-155,-118){\large\textbf{$\rightarrow$}}
\caption{\small  (left to right) Short time, mid-time, and final folded (equilibrium) configurations of interfaces under the adhesion energy gradient flow \eqref{e:CHA-Energy}. Simulations begin with identical initial data and parameters, differing only in the value of the precompression factor: $\rho= 1.2, 1.3,$ and $1.4$ as labeled in the image in the left column. Other parameter values are given in Table \ref{T:T1}.  }
\label{f:fold}
\end{figure}
The large values of $\beta$ drives the system on a fast time scale to arrive at a near equilibrium interface length. On these fast time scales the evolution is dominated by motion {\sl against curvature} regularized by the Willmore flow, see \cite{CP23} for analysis of this motion in a related system. 
The buckling nature of motion against curvature generates folds within the interface.    Figure\,\ref{f:fold} depicts simulations arising from identical initial data with different values of the precompression parameter $\rho.$ Increasing $\rho$ generates more folds on the short time scale which has a significant impact on the final folded state. 
The three simulations are for precompression factors $\rho=1.2, 1.3,$ and $1.4$ in the top, middle, and bottom rows respectively. The initial folds evolve into fingers that initiate a traveling-wave motion. The normal velocity lengthens the fingers, pulling them into parallel sheets on an intermediate time scale depicted in the middle column. On a long time-scale the flow reaches its final folded state, depicted in the third column. Increasing $\rho$ increases the final surface area and the number of folds. The effect on the end state is a bifurcation from a single 'rolled-up' sheet into a family of two- and three-stacked layers of flat sheets. Increasing $\rho$ beyond the value $1.4$ can lead to curve crossing in the fast transient. Curve crossing is inhibited but not excluded as the repulsive term, scaled by $A_0$, is finite, and can be overwhelmed by motion against curvature during the buckling transient. Indeed the $\rho=1.4$ simulation shows a glancing self-intersection event in the short-time column. 

\section{Discussion}
We presented a derivation of the gradient flows of local and two-point interfacial energies in a general framework. Throughout the analysis interfaces are assumed to be far from self-intersection. Self-intersection requires total negative curvature in excess of $2\pi.$ For the faceting energy, which seeks to minimize negative curvatures this assumption generically holds for quasi-equilibrium states. Energies dominated by an adhesion-repulsion two-point energy generically do introduce negative curvatures, but the repulsion core inhibits self-intersection.  Rigorous non-self-intersection results for these systems are plausible given uniform bounds on curvatures. For many of the energies proposed it is straightforward to demonstrate the existence of global minimizers over function spaces on $\mbbS$ that incorporate spatially constant arc-length parameterization, the challenge is to show that any minimizer so constructed does not non-self intersection.

 A foci for future work is to conduct the spectral analysis of linearization about stationary and quasi-steady patterns. The spectral analysis of  spiral waves which arise as reductions of reaction-diffusion systems is well established \cite{SS_23}.  Bifurcations of classes of Delaunay surfaces immersed in $\mbbR^n$ have been conducted as an eigenvalue problem associated to energies whose critical points subject to enclosed volume or surface area constraints have constant mean curvature, \cite{Koiso_15}. Both of these classes of results are encouraging.

 The ultimate goal is to extend the analysis to study quasi-steady patterns and bifurcation in reaction diffusion systems on an evolving interface. Many problems of interest to biological membrane evolution can be formulated in terms of an energy coupling interfacial configuration to surface diffusion and reaction of active materials that are bound to the membrane, \cite{Shemesh-14, Low_14, Gera_18}. The recent development of computational tools for the bifurcations arising in these geometric problems should greatly increase the interest in this rich class of problems, further inspiring model development, \cite{MU_Arxiv}.

\section{Appendix}
This provides discussion of the numerical scheme and results that are tangential to the main issues.

\subsection{Discussion of Numerical Scheme}
\label{s:numer}
We outline the approach to the numerical resolution of the normal velocity and present a numerical convergence study for an example problem.  The approach employs the scaled arc length formulation over $\mbbS$ with $|\mbbS|=1.$ The parameterization satisfies 
\begin{equation}
\label{ap:scaledAL}
g=|\partial_s \gamma| = L 
\end{equation}
for all $s$ where $L(t)$ is the length of the curve. This has been used successfully in several previous works of the authors for a range of applications \cite{apmoyles,appan}. This choice of arc length results in simple expressions for spatial derivatives, in particular 
\begin{eqnarray}
\label{ap:curvature}
    \kappa & = & - \frac{1}{L^3} \gamma_{ss} \cdot \gamma_s^{\perp}, \\
\label{ap:kappass}
    \Delta_s \kappa & = & \frac{1}{L^2} \frac{\partial^2 \kappa}{\partial s^2},
\end{eqnarray}
where $\bot$ denotes rotation by $\pi/2.$

For a convergence study we consider the Canham-Helfrich adhesion energy (\ref{e:CHA-Energy}). We formulate the numerical approximation as a Differential Algebraic Equation (DAE) \cite{apDAE} in which at each time step the vector positions $\gamma(s) = (x(s),y(s))$, curvatures $\kappa(s)$, higher derivatives $\kappa_{ss}(s)$, and normal velocities $V(s)$ are approximated on an $N$ point stencil. The curve length $L$ represents one additional unknown. 
This yields an implicit system for the $5N+1$ total unknowns through a second order finite difference in space in conjunction with implicit Euler time stepping. The equations are summarized below, with capital letters for the discrete variables with subscripts for spatial grid point and superscripts for time step. 
\begin{description}
    \item[N equations] With time step $k$ we have normal motion 
    \[
    ({\bf X}^{n+1}_j-{\bf X}_j^n) \cdot (D_1 {\bf X}^{n+1}_j)^\perp= k L^{n+1} V^{n+1}_j
    \]
    where ${\bf X}_j = (X_j,Y_j)$ and $D_1$ is the periodic second order finite difference approximation of first derivatives 
    \[
    D_1 X_j = \frac{(X_{j+1}-X_{j-1})}{2h}
    \]
    with $j+1$ and $j-1$ taken mod $N$ and $h = 1/N$. 
    \item[N equations] We preserve scaled arc length at the next time step 
    \beq
    \label{ap:scaled}
    |{\bf X}^{n+1}_{j+1} - {\bf X}^{n+1}_j|^2 = h^2[L^{n+1}]^2
    \eeq
    \item[N equations] Relate the curvature to positions, this is a second order finite difference approximation of (\ref{ap:curvature}).
    \item[N equations] Relate the Laplacian of curvature to curvature, this is a second order finite difference approximation of (\ref{ap:kappass}). 
    \item[N equations] Specify normal velocity values $V$ at each grid point 
    \[
    V^{n+1} = \epsilon (\kappa_{ss}^{n+1}-[\kappa^{n+1}]^3) - \kappa \beta(L^{n+1}-\rho L^0) + \mbox{adhesion}
    \]
    where the first two terms above come from (\ref{e:CHA-GF}) and the adhesion terms are discretizations of \eqref{e:Adh-NV} using trapezoidal rule for the discretization of the integrals. 
    \item[1 equation] There is an arbitrary constant in the tangential motion that needs to be fixed. This is accomplished by enforcing zero average tangential velocity: 
    \[
    \sum_{j=1}^N ({\bf X}^{n+1}_j-{\bf X}_j^n) \cdot (D_1 {\bf X}^{n+1}_j) = 0.
    \]
\end{description}
The resulting nonlinear system for the $5N+1$ unknowns at time level $n+1$ is solved using Newton's method with hand coded Jacobian entries. Adaptive time stepping with a user-specified local error tolerance $\sigma$ is used, with the local error for two steps with time step $k$ estimated by comparing to a computation with time step $2k$.

    The discretized systems are  highly nonlinear and stiff, sixth order in spatial variables for the adhesion energy models and eighth order for the faceting models. The high stiffness coupled with the DAE nature naturally suggest implicit time stepping methods. The dynamics feature a wide range of time scales which makes adaptive time stepping necessary. Accurately assessing the local error led to our use of Newton iterations for the implicit time steps. For the Newton iterations to be sparse solves, we used finite difference rather than spectral spatial approximation. Backward Euler was the simplest time stepping method to use, with the advantages of its good stability properties for stiff systems and the ease in which adaptive time stepping can be implemented. The extension of the system to include curvature, the normal velocity and other variables made hand coding the Jacobian matrix for Newton iterations practical. 
    While level set methods \cite{aplevel} can accommodate topological changes and easily extend to surfaces in 3D it would be a challenge to formulate the energetic effects considered in the models presented here. 

For the convergence problem we take initial data 
$$
    \gamma_0(s) =   r(2 \pi s)\tau( 2 \pi s),
$$
with $r(\theta) = 1 + 0.35 \cos(3 \theta)$. This initial parametrization is not scaled arc length, and a pre-processing step is conducted with Newton iterations to convert to scaled arc length in the approximate sense (\ref{ap:scaled}). The code includes a local smoothing strategy if these iterations do not converge, although that is not needed for this example. We take parameters in (\ref{e:CHA-Energy}):   $\beta=1$, $\rho=1.5$,  $\epsilon = 0.1$ and adhesion energy parameters $A_0=10$, $l_* = 1/2$, $a=1$. A run with $N=64$ grid points and local error tolerance $\sigma = 10^{-4}$ is shown in Figure~\ref{ap:figure}. 

\begin{figure}
\centerline{
    \includegraphics[width=5cm]{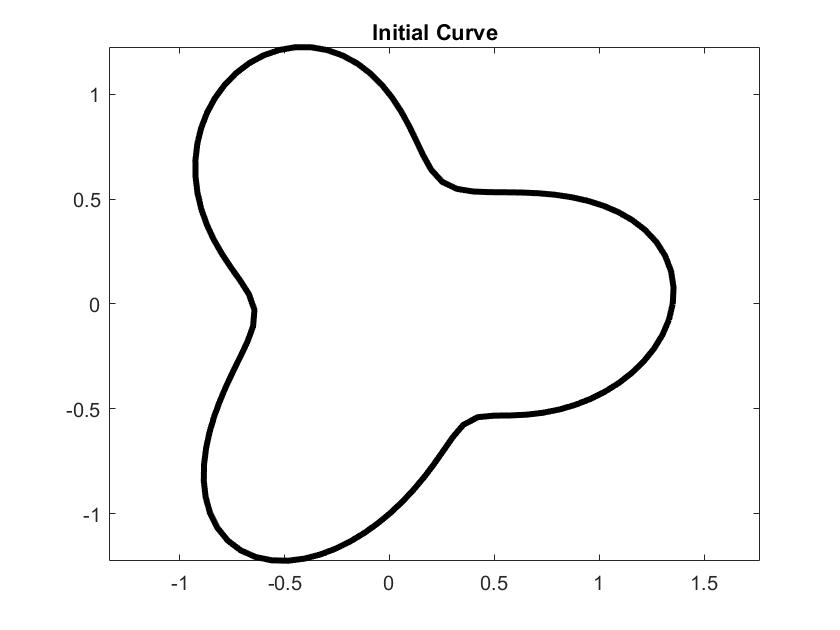}
    \includegraphics[width=5cm]{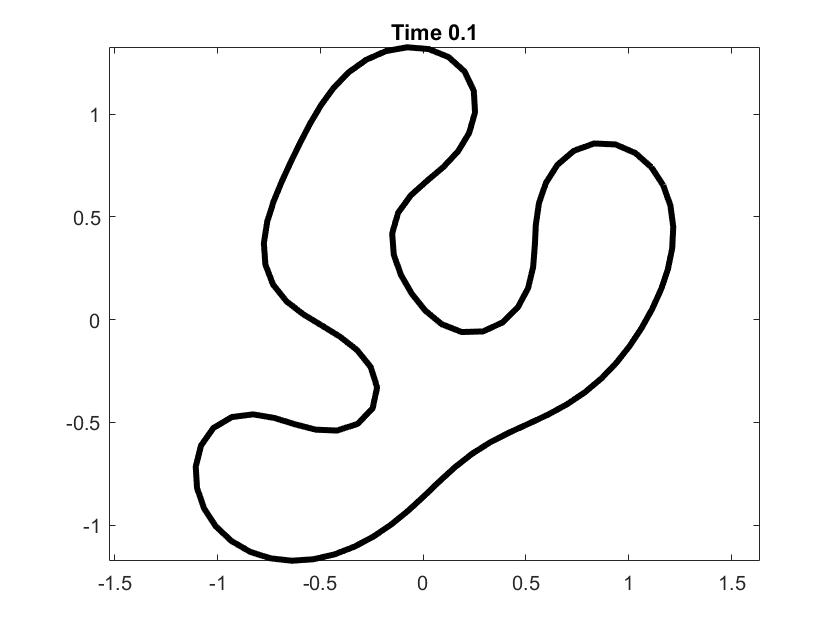}
    \includegraphics[width=5cm]{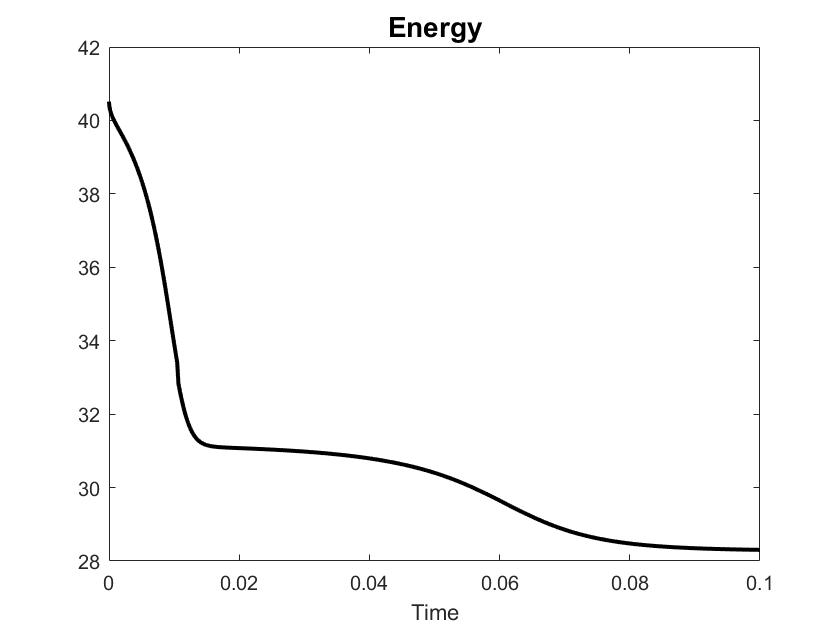}
    }
    \caption{\label{ap:figure} Numerical convergence study example showing from left to right: initial conditions, solution at $t=0.1$, and the Energy as a function of time.}
    \label{f:Numconv}
\end{figure}

Keeping the local error tolerance fixed, we compute solutions with $N=50$, 100, 200, 400, and 800 spatial grid points. We compare the maximum norm difference in position variables at common grid points between successively refined solutions in Table~\ref{ap:t1}.
\begin{table}
    \begin{tabular}{|l|c|c|} \hline
    $N$ & $||{\bf X}(2N) - {\bf X}(N)||_\infty $ & ratio \\ \hline
    50 & 0.1909 & \\
    100 & 0.0397 & 4.8 \\
    200 & 0.0101 & 3.9 \\
    400 & 0.0024 & 4.2 \\ \hline
    \end{tabular}
    \vspace{0.05in}
    \caption{\label{ap:t1} Comparison of computed spatial positions at $t=0.1$ to those from a grid with twice as many points.}
\end{table}
The error decreases by approximately a factor of 4 at each doubling of grid points, confirming second order spatial convergence. 

We also confirm the behavior of the adaptive time stepping by local error tolerance $\sigma = 10^{-3}$, $10^{-4}$, $10^{-5}$, $10^{-6}$ with $N=100$ fixed. The computation takes 65, 180, 571, and 1811 total time steps, respectively. With each ten-fold decrease in error tolerance the number of time steps taken increases approximately by a factor of $\sqrt{10} \approx 3.16$. This is consistent with a correctly identified second order local temporal error.

The MATLAB code that generated the results for the convergence study is available \cite{apcode}

\subsection{Reparameterization and Tangential Flow}
\label{s:reparam}
It well known that reparameterization has a direct connection to tangential motion of an interface,
\cite{Man_11}. For completeness we include a derivation. We consider a one-parameter family of mappings $m:\mbbS\times[-\delta_0,\delta_0]\mapsto\mbbS$
with expansion
\beq\label{e:m-DV} m(s;\delta)=s+\delta m_1(s) +O(\delta^2),
\eeq
where $m_1\in H^2_{\rm per}(\mbbS).$
The reparameterization induces the interface map $\tgam(\cdot;\delta):=\gamma\circ m(\cdot;\delta).$
This reparameterization induces an extrinsic vector field characterized in the follow lemma.
\begin{lemma}
\label{l:DV}
    Given a curve $\gamma$ with curvature and arc length pair $(\kappa, g)$. The first variation of $\gamma$ induced  by the one-parameter family of reparameterization maps $m=m(\delta)$ in \eqref{e:m-DV} induces intrinsic perturbations that correspond to the purely tangential extrinsic vector field $(\mbbV^n,\mbbV^\tau)^t=(0,m_1 g)^t.$
\end{lemma}
\begin{proof}
    
The perturbed interface map admits the expansion
$$ \tilde\gamma(s)=\gamma(s)+\delta m_1\partial_s \gamma +O(\delta^2).$$
Acting with $\partial_s$ yields
$$ \partial_s \tgam = \partial_s\gamma +\delta \left(m_1\partial_s^2\gamma +\partial_s \gamma \partial_sm_1\right)+O(\delta^2).$$
The perturbed arc length has the expansion
$$
\begin{aligned}
\tg &:=|\partial_s \tgam|
=g+\frac{\delta}{|\partial_s\gamma|}\left(m_1\partial_s^2\gamma\cdot\partial_s\gamma+|\partial_s\gamma|^2 \partial_s m_1\right) +O(\delta^2), \\
&=g+ \delta\left(m_1\partial_s g +g \partial_s m_1\right) +O(\delta^2),
\end{aligned}
$$
which we write as
\beq\label{e:DV_g}
\tg = g+ \delta g \nabla_s(m_1 g) +O(\delta^2).
\eeq
From this we deduce the expansion
$$ \begin{aligned}
       \tg^{-1} &= g^{-1}
       \left(1-\delta\nabla_s(m_1g)\right) +O(\delta^2).
\end{aligned}$$
The tangent vector expands according to
$$
\begin{aligned}
    \tilde\tau &:=\tilde{\nabla}_s\tgam = \frac{1}{\tg}\partial_s\tgam, 
    = \nabla_s \gamma + \delta\left(
    -\nabla_s(m_1g)\nabla_s\gamma +m_1\nabla_s\partial_s\gamma +\partial_s m_1 \nabla_s\gamma \right)+O(\delta^2),\\
    &= \tau + \delta\left(
    -m_1\tau \nabla_s g +m_1\nabla_s(g \tau) \right)+O(\delta^2).
\end{aligned}$$
Distributing the derivative and using \eqref{e:tau-n} to replace the surface gradient of the tangent yields
\beq    \tilde\tau =\tau -\delta \kappa m_1 g\,n+O(\delta^2).
\eeq
We use \eqref{e:curv-def} to expand the curvature,
$$\begin{aligned}
\tilde\kappa &:= \pm|\tilde\nabla_s\tilde\tau| = \pm\left|\frac{1}{\tg}(\partial_s\tilde\tau)\right|,\\
&= \pm\left|g^{-1} \left(1-\delta\nabla_s(m_1 g)\right)
 \left(\partial_s\tau -\delta\partial_s(\kappa m_1 g\,n)
 \right)\right| +O(\delta^2),\\
&= \pm\left| -\kappa n + 
\delta\left(\nabla_s(m_1 g) \kappa n -\nabla_s(\kappa m_1 g) n +\kappa^2 m_1 g\,\tau\right)
\right| +O(\delta^2),\\
&= \kappa -\delta 
\left(\nabla_s(m_1 g)\kappa  -\nabla_s(\kappa m_1 g)\right)+O(\delta^2),
\end{aligned}
$$
where we chose the sign that aligns with $\kappa.$ We rewrite the curvature perturbation as
\beq
\label{e:DV_kappa}
\tilde \kappa= \kappa +\delta 
 m_1 g \nabla_s\kappa +O(\delta^2).
\eeq
Combining \eqref{e:DV_g} and \eqref{e:DV_kappa}, the domain variation under the mapping $m$ induces the perturbations
\beq\label{e:DP}
\bpm  \kappa_1\cr g_1 \epm = 
\bpm m_1 g\, \nabla_s\kappa  \cr
g\nabla_s(m_1 g)
 \epm =\cM\bpm 0\cr m_1 g\epm.
 \eeq
 \end{proof}

\begin{lemma}
\label{l:Wind}
The second order energy \eqref{e:LIE} satisfies the first-integral identity \eqref{e:W-invar}.
\end{lemma}
\begin{proof}
It is sufficient to act out the $\nabla_s$ operator on the arc length variation term in \eqref{e:first_var} and observe that it cancels with functional variation term. We consider $F$ in cases by the type of variable dependence. In all cases the factors of $g$ cancel in arc length variation. For $F=F(\Delta_s \kappa)$, 
$$\begin{aligned}g\partial_g\cE &= \nabla_s\left(-F_2\Delta_s\kappa+\nabla_sF_2 \nabla_s\kappa +F\right),\\
&=-\nabla_s F_2 \Delta_s\kappa-F_2\nabla_s\Delta_s\kappa+\Delta_sF_2\nabla_s\kappa +\nabla_sF_2\Delta_s\kappa+F_2\nabla_s\Delta_s\kappa,\\
&= \Delta_sF_2 \nabla_s \kappa=\partial_\kappa \cE \nabla_s\kappa.
\end{aligned}$$
For the case $F=F(\nabla_s \kappa)$ the identify reduces to 
$$\begin{aligned}
g\partial_g\cE&= \nabla_s(-F_1\nabla_s \kappa+F)
=-\nabla_sF_1\nabla_s\kappa -F_1\Delta_s\kappa+F_1\Delta_s\kappa=\partial_\kappa\cE\nabla_s\kappa.
\end{aligned}
$$
Finally for the case $F=F(\kappa)$ the identity is an immediate application of the chain-rule.
\end{proof}

\subsection{Two-Point Curve Invariants.}
\label{s:TP_invar}

We establish the following invariants associated to the two-point interaction kernel acting on smooth, closed, non-self-intersecting curves.
\begin{lemma}
\label{l:TP_invar}
Let $\gamma:\mbbS\mapsto \mbbR^2$ be a smooth map whose image is  closed and non-self-intersecting. Define the distance vector $d:\mbbS\times\mbbS\to \mbbR^2$ via $d(s,\ts)=\gamma(s)-\gamma(\ts).$
The product space admits the decomposition
    $$\mbbS\times\mbbS =\bigcup\limits_{r\geq0} \mrE_r,$$
    in terms of the $d$-level sets  
    $$\mrE_r:=\{(s,\ts): |d(s,\ts)|=r\}.$$
Let $g(s)=|\gamma'(s)|$ denote the arc length of the parameterization of $\gamma$ and $\mathrm H$ denote the one-dimensional Hausdorff measure on $\mrE_r.$ Then $\gamma$ satisfies the equalities
\begin{align}
\label{e:Ad-Rot}
    \int_{\mrE_r} \frac{(n\cdot\gamma)(\tilde\gamma\cdot \tau) -(n\cdot\tilde\gamma)(\gamma\cdot\tau)}
    {\left( g^2 (d\cdot \tau)^2+\tilde g^2 (d\cdot \tilde \tau)^2\right)^{\frac12}}\, g\tilde g\, \rmd \mathrm H\,\rmd r &=0, \\
 \label{e:Ad-Trans}
   \int_{\mrE_r} \frac{(n\cdot d)(n\cdot v_0) +(\tau\cdot d)(\tau\cdot v_0)}
    {\left( g^2 (d\cdot \tau)^2+\tilde g^2 (d\cdot \tilde \tau)^2\right)^{\frac12}}\, g\tilde g\, \rmd \mathrm H\,\rmd r &=0,
\end{align}
for each $r\geq0$ and each $v_0\in\mbbR^2.$ Here $n, \tau$ denote the normal and tangent to $\gamma$ at $s$ while a tilde superscript denotes evaluation at $\ts$.
\end{lemma}
\begin{proof}
Lemma\,\ref{lem:kerM} establishes that the kernel of $\cM$ is comprised of the infinitesimal generators of the rigid body motions of $\gamma$. We first consider the rotational invariant.
Using the formula \eqref{e:Adh-Ex_grad} and the representation of the rigid body rotation the invariance of energy under rigid rotation implies that 
$$
 0=\int_\mbbS\bpm \mbbA\cdot n+\kappa \mrB\\ 0 \epm \cdot\begin{pmatrix}\gamma\cdot \vtau \cr -\gamma\cdot n\end{pmatrix} \, \rmd \sigma  = \int_\mbbS (\mbbA\cdot n) (\gamma\cdot\tau) +\mrB\gamma\cdot\nabla_s n\,\rmd \sigma,
 $$
 where we used \eqref{e:tau-n} to rewrite $\kappa\tau.$
 Since $\nabla_s\gamma\cdot n=0,$ integrating by parts on the second term yields
 $$ \int_\mbbS \mrB \gamma\cdot \nabla_sn\,\rmd\sigma =
 -\int_\mbbS \gamma\cdot n \nabla_s\mrB\,\rmd \sigma.$$
 On the other hand
 $\nabla_s\mrB=\mbbA\cdot\tau,$
 so that
 \beq\label{e:Ad-Invar}
 \begin{aligned}
     0&=\int_\mbbS (\mbbA\cdot n)(\gamma\cdot\tau) -(\mbbA\cdot\tau)(\gamma\cdot n)\,\rmd \sigma,\\
     &=\int_{\mbbS\times\mbbS}\mrA'(|d|^2)
     \bigl((n\cdot d)(\tau\cdot\gamma)-(n\cdot\gamma) (d\cdot\tau)\bigr)\,\rmd\sigma_s\rmd\sigma_{\ts},\\
     &= \int_{\mbbS\times\mbbS}     \mrA'(|d|^2)
     \bigl((n\cdot \tilde\gamma)(\tau\cdot\gamma)-(n\cdot\gamma) (\tilde\gamma\cdot\tau)\bigr)
     \,\rmd\sigma_s\rmd\sigma_{\ts}.
     \end{aligned}
     \eeq
We apply the co-area formula to rewrite the integral over $\mbbS$ to one over the level sets of $u(s,\ts):=|d(s,\ts)|$.  Specifically we rewrite the last integral in terms of $f:\mbbS^2\mapsto\mbbR$, so that
$$ \int_{\mbbS^2}f(s,\ts) |\nabla_{s,\ts} u|\,\rmd s \rmd \ts = \int_{\mbbR_+}\int_{\mrE_r} f(s,\ts) \mrd \mathrm H \mrd r.$$
The gradient of $u$ satisfies
$$ \left| \nabla_{s,\ts}\,u\right| =\frac{1}{|d|}\left|\bpm g d\cdot \tau\cr 
-\tilde g d\cdot\tilde{\tau}\epm\right|=
\frac{1}{|d|} \left( g^2 (d\cdot \tau)^2+\tilde g^2 (d\cdot \tilde \tau)^2\right)^{\frac12},$$
where a tilde denotes evaluation at $\ts.$ 
Since $|d|=r$ the co-area formula allows us to rewrite the last integral in \eqref{e:Ad-Invar} as
    $$0=\int_{\mbbR_+}\mrA'(r^2)r\int_{\mrE_r} \frac{(n\cdot\gamma)(\tilde\gamma\cdot \tau) -(n\cdot\tilde\gamma)(\gamma\cdot\tau)}
    {\left( g^2 (d\cdot \tau)^2+\tilde g^2 (d\cdot \tilde \tau)^2\right)^{\frac12}}\, g\tilde g\, \rmd \mrH\,\rmd r.$$
 Within this derivation the two-point interaction kernel $\mrA:\mbbR_+\mapsto\mbbR$ is an arbitrary smooth function. The integral over $\mrE_r$ is in $L^2(\mbbR_+)$ and is zero at $r=0$ where $\gamma=\tilde\gamma.$ Since $r \mrA'(r^2)$ can take dense values in $L^2(\mbbR)$ the result \eqref{e:Ad-Rot} follows.

 Rigid translates are induced by the extrinsic vector field $\mbbV=(v_0\cdot n,v_0\cdot\tau)$ where $v_0\in\mbbR^2$ is an arbitrary constant vector. We deduce that
 $$ 0   = \int_{\mbbS\times\mbbS}\mrA'(|d|^2)\bigl((n\cdot d)( v_0 \cdot n)+(\tau\cdot d) (v_0\cdot\tau) \big)\,\rmd\sigma_s\rmd\sigma_{\ts}.$$
 Following the same steps yields
 \eqref{e:Ad-Trans}.
 \end{proof}
 \begin{remark}

  The sets $\{\mrE_r\}_{r\geq0}$ are each comprised of a finite collection of disjoint components of $\mbbS^2.$ Under the assumption that $\gamma$ does not self-intersect, $\mrE_0$ comprises the diagonal of the product space and for $r>0$ small enough $\mrE_r$ comprises two disjoint curves that are symmetric about the diagonal. More generally, from the implicit function theorem, for each non-degenerate level set, defined below, there are a finite collection of functions $\ts_r:\mbbS\mapsto\mbbS$ such that $|d(s,\ts_r(s))|=r$ that satisfy the relation
    $$ \frac{\rmd \ts_r}{\rmd s}= \frac{\gamma^\prime(s)\cdot d(s,\ts_r)}{\gamma^\prime(\ts_r)\cdot d(s,\ts_r)}=\frac{g(s)\tau(s)\cdot d(s,\ts_r)}{g(\ts_r)\tau(\ts_r)\cdot d(s,\ts_r)},$$
    subject to initial data from distinct connected components of $\mrE_r.$ The equation for $\ts_r$ is well defined so long as $\tau(\ts)\cdot d(s,\ts_r)\neq 0.$ This quantity is zero when the circle of radius $r$ about $\gamma(s)$ intersects the image $\Gamma$ tangentially, indicating a boundary point of a component of the level set $\mrE_r.$  Non-degeneracy means that zeros of $\tau(\ts)\cdot d(s,\ts)$ are transverse in $\ts$ for fixed $s.$
 \end{remark}

\section*{Acknowledgment}
The second author recognizes support from the NSF through grant NSF-2205553. The third author recognizes support from NSERC Canada. The authors anticipate development of a freeware code library for general interfacial gradient flows described in this paper.


\begin{thebibliography}{99}

\bibitem[Aland et al (2014)]{Low_14}
S. Aland, S. Egerer, J. Lowengrub, A. Voigt, Diffuse interface models of locally inextensible vesicles in a viscous fluid, {\sl J. Comput. Physics}, {\bf 277} (2014), 32-47.

  \bibitem[Alikakos and Fusco (2008)]{AF_08}
  N. Alikakos and G. Fusco, On the connection problem for potentials with several global minima, {\sl Indiana University Mathematics Journal}, {\bf 57} (4) (2008), 1871-1906.

  \bibitem[Alikakos et al (2006)]{ABC_06}
  N. Alikakos, S. Betelu, and X. Chen,
  Explicit stationary solutions in multiple well dynamics and non-uniqueness of interfacial energy densities, {\sl Euro. Jnl. of Applied Mathematics},
  {\bf 17} (2006) 525-556.

  \bibitem[Ascher and Petzhold (1998)]{apDAE}
U. Ascher and L. Petzhold, Computer Methods for Ordinary Differential Equations and Differential-Algebraic Equations, Society for Industrial and Applied Mathematics, Philadelphia (USA), ISBN 978-0898714128. 

  \bibitem[Austin and Staehelin (2011)]{AS-11}
  J. Austin and A. Staehelin, Three dimensional architecture of Grana and Stroma Thylakoids of higher plants as determined by electron tomography, {\sl Plant Physiology} {\bf 155} (2011) 1601-1611.



 \bibitem[M'barek et al (2017)]{MBarek-17}
 K. M'Barek, D. Ajjaji, A. Chorlay, S. Vanni, L. Foret, A. Thaim,
 ER membrane phospholipids and surface tension control cellular lipid droplet formation,
 {\sl Developmental Cell} {\bf 41} (2017) 519-604.

\bibitem[Barrett et al (2020)]{Barrett_2020}
J. Barrett, H. Garcke, R. N\"urnberg.
``Parameteric finite element approximations of curvature-driven interface evolutions''. In: {\sl Geometric Partial Differential Equations -- Part I.}
Ed. by Andrea Bonito and Ricardo H. Nochetto, Vol.21. Handbook of Numerical Analysis. Elsevier, 2020. Chap. 4, pp 275-423, DOI 10.1016/bs.hna.2019.05.002.

\bibitem[Blanazs et~al. (2009)]{Blanazs-09}
A. Blanazs, S.P. Armes, and A. J. Ryan, Self-assembled block copolymer aggregates: 
from micelles to vesicles and their biological applications, 
{\sl Macromolecular rapid communications}, {\bf 30} (2009) 267--277.

\bibitem[Canham (1970)]{Can_70}
P. Canham, Minimum energy of bending as a possible explanation of biconcave
shape of human red blood cell, {\sl J. Theoret. Biol.} {\bf 26} (1970) 61–81.

\bibitem[Helfrich (1973)]{Hel_73}
W. Helfrich, Elastic properties of lipid bilayers -- theory and possible experiments, {\sl Z. Nat. Forsch. C} {\bf 28} (1973) 693-703.

\bibitem[Chen 2007]{XChen_94}
Xinfu Chen,
Spectrum for the Allen-Cahn, Cahn-Hilliard, and phase field equations for generic interfaces,
{\sl Communications in Partial Differential Equations},
{\bf 19} (7) (1994), 1371-1395.

\bibitem[Chen et al (2020)]{CDPV-20}
Y. Chen, A. Doelman, K. Promislow, F. Veerman, Robust stability of multicomponent membranes: the role of glycolipids, {\sl Archives Rational Mechanics and Analysis} {\bf 238} (3) (2020), 1521-1557 .

\bibitem[Chen and Promislow (2023)]{CP23}
Y. Chen and K. Promislow, Curve Lengthening via regularized motion against curvature from the strong FCH gradient flow, {\sl J. Dynamics and Differential Eqns.}, {\bf 25} (2023) 1785-1841.

\bibitem[Doelman et al (2007)]{KP_RG2}
A. Doelman, T. Kaper, and K. Promislow, Nonlinear asymptotic stability of the semi-strong pulse dynamics in a regularized Gierer-Meinhardt model, {SIAM J. Math. Analysis}, {\bf 38} (2007) 1760-1787.

\bibitem[Gau et al (2006)]{Gau06}
H. Gau, T. Sage, and K. Osteryoung, FZL, an FZO-like protein in plants, is a determinant of thylakoid and chloroplast morphology, {\sl Proc. National Academy Sciences}, {\bf 103} (2006).

\bibitem[Gera and Salac (2018)]{Gera_18}
P. Gera and D. Salac,Modeling of multicomponent three-dimensional vesicles, {\sl Computers and Fluids} {\bf 172} (2018) 362-383.

\bibitem[Giaquinta and Hildebrandt (1996)]{GH_96}
{\sl Calculus of Variations. I}, Grundlehren der Mathematischen Wissenschaften [Fundamental Principles of Mathematical Sciences], vol 310, Springer-Verlag, Berling 1996, ISBN 3-540-59625-X. 

\bibitem[Gibson (2010)]{Gibson10}
M. Gibson, Slowing the growth of ice with synthetic macromolecules: beyond antifreeze(glyco) proteins, {\sl Polymer Chemistry} {\bf 1} (2010) 1129-1340.

\bibitem[Koiso et al (2015)]{Koiso_15}
M. Koiso, N. Palmer, and P. Piccione,
Bifurcation and symmetry breaking of nodoids with fixed boundary, {\sl Adv. Calc. Var.} {\bf 8} (4) (2015) 337-370.

 \bibitem[Mantegazza (2011)]{Man_11}
 C. Mantegazza, Lecture Notes on Mean Curvature Flow, Progress in Mathematics, {\bf 290}, Birkh\"auser/Springer Basel AG, Basel (2011)

 \bibitem[Meinder and Uecker (Arxiv)]{MU_Arxiv}
 A. Meinder and J. Uecker, Differential Geometric Bifurcatoin problems in pde2path -- algorithms and tutorial examples, ArXiv:2309.03546v1 (2023).
 

\bibitem[Moyles and Wetton (2015)]{apmoyles}
I. Moyles and B. Wetton, A numerical framework for singular limits of a class of reaction diffusion problems, {\sl Journal of Computational Physics} {\bf 300} (2015) 308-326.

\bibitem[Osher and Fedkiw (2003)]{aplevel}
S. Osher and R. Fedkiw, Level Set Methods and Dynamic Implicit Surfaces, 
Springer, New York (USA), https://doi.org/10.1007/b98879.

\bibitem[Pan and Wetton (2008)]{appan}
Z. Pan and B. Wetton, A numerical method for coupled surface and grain boundary motion, {\sl European Journal of Applied Mathematics} {\bf 19} (2008) 311-327.

\bibitem[Pismen (2006)]{Pismen_06} L. M. Pismen,
Patterns and interfaces in dissipative dynamics,
{\bf Springer Series in Synergetics}, Springer Complexity, (2006)

\bibitem[Promislow (2002)]{KP_RG1}
K. Promislow, A renormalization method for modulational stability of quasi-steady patterns in dispersive systems,
{\sl SIAM J. Math. Analysis}, {\bf 33} (2002), 1455-1482.

\bibitem[Rupp (2024)]{Rupp_24}
F. Rupp,  The Willmore flow with prescribed isoperimetric ratio, {\sl Comm. PDE}, {\bf 49} (2024) 148-84.

\bibitem[Sandstede and Scheel (2023)]{SS_23}
B. Sandstede and A. Scheel, Spiral Waves: Linear and Nonlinear Theory, {\sl Memoirs of the AMS} {\bf 285} (1413) (2023) 1-126.

\bibitem[Shemesh et al (2014)]{Shemesh-14} T. Shemesh, R. Klemm, F. Romano,  S. Wang, J. Vaughan, X. Zhuang,  H. Tukachinsky, M. Kozlov, and T. Rapoport, A model for the generation and interconversion of ER morphologies, {\sl
Proceedings of the National Academy of Sciences}, {\bf 111} no 49 (2014) pp 
E5243--E5251.
    

\bibitem[Tachikawa \& Mochizuki (2017)]{Golgi-17}
M. Tachikawa and A. Mochizuki, Golgi apparatus self-organizes into the characteristic shape via postmitotic reassembly dynamics,
{\sl Proc. National Academy Sciences}, {\bf 114} (2017) 5177-5182. 

\bibitem[Tabliabue et al (2024)]{Tag24}
A. Tagliabue, C. Micheletti, M. Mella, Effect of Counterion Size on Knotted Polyelectrolyte Conformations, {\sl The Journal of Physical Chemistry B} (2024) {\bf 128} (17), 4183-4194, DOI: 10.1021/acs.jpcb.3c07446

\bibitem[Wetton (2024)]{apcode}
GitHub Repository ``InterfaceGradient" available at https://github.com/bwetton/InterfaceGradient (verified December 27, 2024). 

\bibitem[Wikimedia Commons]{WC_chloro}
Wikimedia Commons: https://en.wikipedia.org/wiki/ \\{\textrm File:Chloroplast\_in\_leaf\_of\_Anemone\_sp\_TEM\_85000x.png}

 







\end{thebibliography}
\end{document}